\documentclass[11pt,a4paper]{article}

\usepackage{pdfsync}
\usepackage[hidelinks]{hyperref} \usepackage{latexsym,amsfonts,amsmath,amsthm,graphics,multirow}
\usepackage{mathrsfs}
\usepackage[normalem]{ulem}
\usepackage{enumitem}
\usepackage[dvipsnames,svgnames,table]{xcolor}
\usepackage{epsfig}
\usepackage{here}

\newtheorem{theorem}{Theorem}

\newtheorem{remark}[theorem]{Remark}

\newcommand{\bsgamma}{{\boldsymbol{\gamma}}}
\newcommand{\bseta}{{\boldsymbol{\eta}}}

\newcommand{\bsnu}{{\boldsymbol{\nu}}}

\newcommand{\bsb}{{\boldsymbol{b}}}
\newcommand{\bsm}{{\boldsymbol{m}}}
\newcommand{\bse}{{\boldsymbol{e}}}

\newcommand{\bsx}{{\boldsymbol{x}}}
\newcommand{\bsh}{{\boldsymbol{h}}}

\newcommand{\bst}{{\boldsymbol{t}}}
\newcommand{\bsz}{{\boldsymbol{z}}}

\newcommand{\bsy}{\boldsymbol{y}}

\newcommand\setu{{\mathfrak{u}}}
\newcommand\setv{{\mathfrak{v}}}

\newcommand{\bszero}{\boldsymbol{0}}
\newcommand{\rd}{{\mathrm{d}}}

\newcommand{\R}{\mathbb{R}}

\newcommand{\calA}{{\mathcal{A}}}

\newcommand{\calO}{{\mathcal{O}}}

\newcommand{\calK}{{\mathcal{K}}}

\newcommand{\calS}{{\mathcal{S}}}
\newcommand{\bbZ}{{\mathbb{Z}}}
\newcommand{\bbN}{{\mathbb{N}}}
\newcommand{\bbR}{{\mathbb{R}}}

\newcommand{\rmi}{{\mathrm{i}}}

\newcommand{\rme}{{\mathrm{e}}}
\newcommand{\supp}{{\mathrm{supp}}}
\newcommand\Tstrut{\rule{0pt}{2.6ex}}         \newcommand\Bstrut{\rule[-0.9ex]{0pt}{0pt}}

\newcommand{\mask}[1]{}

\newcommand{\bs}[1]{\ensuremath{\boldsymbol{#1}}}

\usepackage{soul}
\soulregister\cite{7}

\definecolor{darkred}{RGB}{139,0,0}
\definecolor{darkgreen}{RGB}{30,130,80}
\definecolor{darkmagenta}{RGB}{139,0,139}
\definecolor{darkorange}{RGB}{180,60,0}
\definecolor{darkcyan}{RGB}{0,139,139}

\usepackage[hang,flushmargin]{footmisc} 

\title{Fast approximation by periodic kernel-based lattice-point interpolation with application in uncertainty quantification}

\author{{Vesa Kaarnioja}\footnotemark[2], {Yoshihito Kazashi}\footnotemark[3], {Frances Y.~Kuo}\footnotemark[2], {Fabio Nobile}\footnotemark[3], {Ian H.~Sloan}\footnotemark[2] }

\date{\today}

\begin{document}
\maketitle
\renewcommand{\thefootnote}{\fnsymbol{footnote}}
\footnotetext[2]{School of Mathematics and Statistics, University of New South Wales, Sydney NSW 2052, Australia\\ (\texttt{vesa.kaarnioja@iki.fi}, \texttt{f.kuo@unsw.edu.au}, \texttt{i.sloan@unsw.edu.au}).}
\footnotetext[3]{CSQI, Institute of Mathematics, \'{E}cole Polytechnique F\'{e}d\'{e}rale de Lausanne, 1015 Lausanne, Switzerland\\ (\texttt{y.kazashi@uni-heidelberg.de}, \texttt{fabio.nobile@epfl.ch}).}

\begin{abstract}
This paper deals with the kernel-based approximation of a multivariate
periodic function by interpolation at the points of an integration lattice---a setting that, as pointed out by Zeng, Leung, Hickernell (MCQMC2004, 2006) and Zeng, Kritzer, Hickernell (Constr.\ Approx., 2009), allows
fast evaluation by fast Fourier transform, so avoiding the need for a linear solver. The main
contribution of the paper is the application to the approximation problem
for uncertainty quantification of elliptic partial differential equations,
with the diffusion coefficient given by a random field that is periodic in
the stochastic variables, in the model proposed recently by Kaarnioja,
Kuo, Sloan (SIAM J.\ Numer.\ Anal., 2020). The paper gives a full error
analysis, and full details of the construction of lattices needed to
ensure a good (but inevitably not optimal) rate of convergence and an
error bound independent of dimension. Numerical experiments support the
theory.
\end{abstract}

\section{Introduction}
We consider a \emph{kernel-based approximation} for a multivariate
\emph{periodic} function by \emph{interpolation} at a quasi-Monte Carlo \emph{lattice} point set.
Kernel-based interpolation methods are by now well established
(see, e.g.,~\cite{Wen05} and more discussion below).
It is the unique combination of \emph{a periodic kernel plus a lattice point set} here that
will deliver us the significant advantage in computational efficiency.
As already advocated by Hickernell and colleagues in \cite{ZLH06,ZKH09}, the
combination of a periodic reproducing kernel with the group structure of
lattice points means that the linear system for constructing the kernel interpolant involves a circulant matrix, thus can be solved very efficiently using the fast Fourier transform.
So, our kernel method can be fast even if the dimensionality is high.

As also advocated in \cite{ZLH06,ZKH09}, a kernel interpolant is in many
settings optimal among all approximation algorithms that use the same
function values (see also known results on optimal recovery, e.g.,
\cite{MicRiv77,MicRiv85}). We can therefore analyze the worst case
approximation error of our kernel method by using, as upper bound, the
worst case error of an auxiliary algorithm based on a Fourier series
truncated at a hyperbolic cross index set. Using recent works
\cite{CKNS-part1,CKNS-part2}, we here construct a lattice generating
vector with a guaranteed good error bound for our kernel interpolant.
Note, importantly, that \emph{neither the construction of our lattice
	generating vector, nor the implementation of our kernel method, requires
	explicit knowledge or evaluation of the auxiliary hyperbolic cross index
	set}. In short, we know how to find a good lattice point set so that our
kernel method has a small error in addition to being of low cost.

In this paper, the main contribution is to apply and analyze this
periodic-kernel-plus-lattice method to uncertainty quantification of
elliptic partial differential equations (PDEs), where the diffusion
coefficient is given by a random field that is \emph{periodic in the
	stochastic variables}, as in the model proposed recently by Kaarnioja,
Kuo, and Sloan \cite{KKS}. We tailor our lattice generating vector to the
regularity of the PDE solution with respect to the stochastic variables.
Our numerical results beat the theoretical predictions, indicating that
the theory based on worst case analysis may not be sharp.

The kernel approximation developed here may have a role as a surrogate model for complicated forward problems.
One popular use for surrogate models is to allow efficient sampling of the original system.
If the solution of some particularly difficult PDE problem with high accuracy takes a week for a given parameter choice $\bsy$, then having a kernel interpolant that can be evaluated in hours or minutes could be very useful.
A second possible use for the kernel interpolant is in the easy generation of derivatives, needed for example in gradient-based optimization  algorithms.
The surrogate might be even more useful for Bayesian inverse problems.

\smallskip We now elaborate key points.\smallskip

\noindent\textbf{Periodic-kernel-plus-lattice method.}
Let $f(\bsy)=f(y_1,\ldots,y_s)$ be a real-valued function on the
$s$-dimensional unit cube $[0,1]^s$, with a somewhat smooth 1-periodic
extension to $\bbR^s$. Our main interest is in problems where the
dimension $s$ is large. Following \cite{SW01}, we assume that $f$ has an
absolutely convergent Fourier series, and belongs to a weighted mixed
Sobolev space $H:=H_{s,\alpha,\bsgamma}$ which is characterized by a
smoothness parameter $\alpha>1$ and a family of positive numbers
$\bsgamma=(\gamma_\setu)_{\setu\subset\bbN}$ called \emph{weights}; the
details are given in Section~\ref{sec:space}.

Our ultimate application, to be analyzed in Section~\ref{sec:PDE},
concerns a class of elliptic PDEs parameterized by a very high (or
possibly countably infinite)  number of stochastic parameters, for which the solution, as a function of the parameters, is periodic and belongs to the weighted space $H$ for a suitable choice of $\alpha$ and
$\bsgamma$.

The important feature of the space $H$ is that it is a reproducing kernel
Hilbert space (RKHS), with a simple reproducing kernel $K(\bsy,\bsy')$.
This opens the way to the use of kernel methods to approximate functions
in $H$ from given point values. In particular, in this paper we focus on
the \emph{kernel interpolation}: given $f\in H$ and a suitable set of
points $\bst_1,\dots,\bst_n\in [0,1]^s$, we seek for an approximation
$f_n\in H$ of the form
\begin{align} \label{eq:interpolant}
f_n(\bsy) \,:=\, \sum_{k=1}^{n} a_{k}\,K(\bst_k,\bsy),
\qquad\bsy\in[0,1)^s,
\end{align}
which satisfies the interpolation condition $f_n(\bst_{k}) = f(\bst_{k})$,
$k=1,\dots,n$. We refer to $f_n$ as the \emph{kernel interpolant} of $f$.

We will interpolate the function at a set of $n$ \emph{lattice points}
specified by a \emph{generating vector} $\bsz \in \bbZ^s$. The points are
then given by the formula $\bst_k = {(k\bsz \bmod n)/{n}}$,
$k=1,\ldots,n$, with $\bst_n=\bst_0 = \bszero$. A lattice point set has an
additive group structure, implying that the difference of two lattice
points is another lattice point (after taking into account periodicity).

A key property of our reproducing kernel is that it depends only on the
difference of the two arguments, thus $K(\bsy,\bsy') =
K(\bsy-\bsy',\bszero)$, and $K(\cdot,\bszero)$ is a periodic function with
an easily computable expression when $\alpha$ is an even integer.
Combining this with the group structure of lattice points means that the
matrix $[K(\bst_k-\bst_{k'},\bszero)]_{k,k'=1,\ldots,n}$ contains only $n$
distinct values and indeed is a circulant matrix. Therefore the linear
system arising from collocating \eqref{eq:interpolant} at the points
$\bst_{k'}, k' = 1,\ldots,n$, can be solved using the fast Fourier
transform with a cost of $\calO(n\log(n))$.

Once we have the coefficients $a_k$, we can use \eqref{eq:interpolant} to
evaluate the interpolant $f_n$ at $L$ arbitrary points $\bsy_\ell$,
$\ell=1,\ldots,L$, with a cost of $\calO(Ln)$. Remarkably, with almost the
same cost we can evaluate $f_n$ at all the $Ln$ points of the \emph{union
	of shifted lattices} $\bsy_\ell + \bst_{k'}$, $\ell=1,\ldots,L$,
$k'=1,\ldots,n$. Indeed, since $K(\bst_k,\bsy_\ell + \bst_{k'}) =
K(\bst_k-\bst_{k'},\bsy_\ell)$ and the matrix
$[K(\bst_k-\bst_{k'},\bsy_\ell)]_{k,k'=1,\ldots,n}$ is circulant, we have
\[ f_n(\bsy_\ell+ \bst_{k'}) \,=\, \sum_{k=1}^{n} a_{k}\,K(\bst_k-\bst_{k'},\bsy_\ell),
\] which can be evaluated for each $\bsy_\ell$ for all $\bst_{k'}$ together
by fast Fourier transform with a cost of $\calO(n\log(n))$, leading to the
total cost of $\calO(Ln\log(n))$. Comprehensive cost analysis taking into
account also the evaluations of $f$ and $K$ is given in
Section~\ref{sec:cost}.
\medskip

\noindent\textbf{Brief survey on kernel methods in high dimensions.}
Griebel and Rieger \cite{Griebel.M_Rieger_2017_kernel} considered a
(non-interpolatory) kernel approximation based on a regularized
reconstruction technique from machine learning for a class of
parameterized elliptic PDEs similar to the one considered in this work,
yet with non-periodic dependence on the parameters.
They used an anisotropic kernel,
behaving differently in different variables, to address the high
dimensionality of the problem. However, their error estimate was in terms
of the mesh norm or fill distance of the point set, which is the Euclidean
radius of the largest Euclidean ball that contains no points in its
interior.  Since the fill distance behaves at best like $n^{-1/s}$, where
$n$ is the number of sampling points, their estimates inevitably suffer
the \emph{curse of dimensionality}.

Kempf et al.~\cite{Kempf.R_etal_2019_kernel_pde} considered the same PDE
problem and anisotropic kernel as \cite{Griebel.M_Rieger_2017_kernel}.
However, they considered a penalized least-squares approach for kernel
approximation and an isotropic sparse grid as point set, which allowed
them to obtain error estimates with a mitigated (but still present) curse
of dimensionality.

As noted above, lattice points have already been used in a kernel
interpolation method. Zeng et al.~\cite{ZLH06} seem to be the first to
work in this direction, however the question of dependence on dimension
was not considered in their analysis. Zeng et al.~\cite{ZKH09} established
dimension independent error estimates in weighted spaces in the case of
product weights (i.e., weights that have the form
$\gamma_{\setu}=\prod_{j\in \setu} \gamma_j$).
We note, however, that the assumption of product weights is rather limiting. For instance, for integration problems involving parameterized PDEs, the best convergence rates known up to now are obtained by considering weighted space for the parameter-to-solution map with (S)POD weights \cite{Graham.I_eta_2015_Numerische,KKS,Kuo.F_Schwab_Sloan_2012_SINUM},
whereas weighted spaces with product weights lead to the best known rates only for special models  \cite{Gantner.R_etal_2018_affine_local,Herrmann.L_Schwab_2019_local_lognormal,Kazashi.Y_2019_product}.
In this paper we extend these results to the case of kernel approximation
(as opposed to integration) of the parameter-to-solution map, and we are
able to show dimension-independent convergence rates using (S)POD weights
in the general case. To the best of our knowledge, this is the first paper
to use non-product weights for approximation in parameterized PDE
problems.
\medskip

\noindent\textbf{PDEs with periodic dependence on random variables.}
Our motivating application is a class of parameterized elliptic PDEs with
periodic dependence on the  parameters, for which we will establish dimension independent
error estimate for the kernel interpolant, by deriving suitable choices of
smoothness parameter and weights for the problem at hand. To the best of
our knowledge, this is the first paper presenting dimension independent
kernel approximation methods using lattice points for this class of
problems.

We consider uncertainty quantification for an elliptic PDE (see details in
Section~\ref{sec:PDE}) on a physical domain $D\subset\mathbb{R}^d$, $d= 1,
2$ or $3$, in a probability space $(\Omega, \mathscr A, \mathbb{P})$, with
an input random field of the form
\[
a(\bsx,\omega) \,=\,
a_0(\bsx)+\sum_{j\geq 1} \Theta_j(\omega)\,\psi_j(\bsx),\qquad \bsx\in D,~\omega\in\Omega,
\]
where $a_0$ and $\psi_j$ are uniformly bounded in $D$, and
$\Theta_j(\omega)$ are i.i.d.\ random variables following a prescribed
distribution. In the popular affine model, $\Theta_j$ are i.i.d.\ random variables  uniformly distributed on $[-\frac{1}{2},\frac{1}{2}]$.
In the periodic model \cite{KKS}, $\Theta_j$
are i.i.d.~random variables distributed according to the  arcsine distribution and can be parameterized as
\[
\Theta_j \,=\, \frac{1}{\sqrt{6}} \sin(2\pi y_j), \qquad j\ge 1,
\]
with $y_j$ uniformly distributed on $[-\frac{1}{2},\frac{1}{2}]$. The mean
of the random field is $a_0$, and the scaling ${{1}/{\sqrt{6}}}$ is chosen here
so that the covariance of the random field is also exactly the
same as in the affine case. Higher moments are of course somewhat
different, but as argued in \cite{KKS}, there seems to be no clear reason
for preferring one over the other.

Due to periodicity, it is equivalent to work with $y_j$ uniformly distributed in the interval $[0,1]$ instead of $[-\frac{1}{2},\frac{1}{2}]$, thus from now on we consider the parameter space
\[
\bsy \in U \,:=\, [0,1]^\mathbb{N}.
\]

In the earlier paper \cite{KKS} the aim was to develop and analyze a method for computing the expected value of a given quantity of interest, expressed as a linear functional of the PDE solution, hence facing a high dimensional integration problem.
Here, in contrast, the aim is to develop and analyze
a fast method for approximating the solution  $u(\bsx,\bsy)$, or some quantity of interest $Q(\bsy)$ derived from  $u(\bsx,\bsy)$, as an
explicit function of $\bsy$. To that end we will develop a kernel-based
approximation, using the kernel of a reproducing kernel Hilbert space of
periodic functions, and interpolation at a lattice point set.
\medskip

\noindent\textbf{Structure of the paper.} In Section~\ref{sec:optimality} we define the function space setting and
the kernel interpolant, and establish its principal properties, while giving a simple proof of a known optimality result, namely that in the sense of worst case error the kernel interpolant is an optimal $L_{p}$ approximation among all approximations that use the same information about the target function $f\in H$.
Then in Section 3 we establish upper and lower bounds on the error.
For the  upper bound we use the optimality result together
with the error analysis for a trigonometric polynomial method established
by two of the current authors together with Cools and Nuyens
\cite{CKNS-part1,CKNS-part2}. For the lower bound we provide another proof
of a recent result by Byrenheid et al.~\cite{BKUV17}, namely that a
method that draws information only from function values at lattice points
inevitably has a rate of convergence that is at best only half of the best
possible rate, thereby obtaining matching upper and lower bounds up to
logarithmic factors. In Section~\ref{sec:PDE}, we apply the error analysis
developed in Section~\ref{sec:error-analysis} to a parameterized PDE
problem, thereby obtaining rigorous upper error bounds that are
independent of dimension and have explicit rates of convergence.
Section~\ref{sec:cost} is concerned with the cost analysis of our proposed
method. In Section~\ref{sec:numerical} we give the results of some
numerical experiments.

\section{The kernel interpolant}\label{sec:optimality}

\subsection{The function space setting} \label{sec:space}

Let $f(\bsy)=f(y_1,\ldots,y_s)$ be a real-valued function on $[0,1]^s$
with a somewhat smooth $1$-periodic extension to $\R^s$ with respect to
each variable $y_j$. Our main interest is in problems where the dimension
$s$ is large. Following \cite{SW01}, we assume that $f$ has absolutely
convergent Fourier series (and so is continuous),
\begin{align*}
f(\bsy) \,=\, \sum_{\bsh\in\bbZ^s} \widehat{f}(\bsh)\,\rme^{2\pi \rmi \bsh\cdot\bsy},
\qquad\mbox{with}\qquad
\widehat{f}(\bsh) &\,:=\, \int_{[0,1]^s} f(\bsy)\,\rm{e}^{-2\pi \rmi \bsh\cdot\bsy}\,\rd\bsy;
\end{align*}
and moreover belongs to a weighted mixed Sobolev space
$H:=H_{s,\alpha,\bsgamma}$, a Hilbert space with inner product and norm
\begin{align*} \langle f,g \rangle_H
&\,:=\, \langle f,g \rangle_{s,\alpha,\bsgamma}
\,:=\, \sum_{\bsh\in \bbZ^s} r(\bsh)\,
\widehat{f}(\bsh)\,\overline{\widehat{g}(\bsh)}, \\
\|f\|_H
&\,:=\, \|f\|_{s,\alpha,\bsgamma}
\,:=\, \bigg(\sum_{\bsh\in \bbZ^s} r(\bsh)\,|\widehat{f}(\bsh)|^2\bigg)^{1/2}, \nonumber
\end{align*}
where
\begin{align*}
r(\bsh) \,:=\, r_{s,\alpha,\bsgamma}(\bsh)
\,:=\, \frac{1}{\gamma_{\supp(\bsh)}} \prod_{j\in \supp(\bsh)} |h_j|^\alpha,
\end{align*}
with $\supp(\bsh) := \left\{j\in\{1:s\} :h_j\ne 0\right\}$ and
$\{1:s\}:=\{1, 2, \ldots, s\}$, and with the $\bsh=\bs0$ term in the sum
to be interpreted as $\gamma_\emptyset^{-1} |\widehat{f}(\bs0)|^2$. The
weighted space $H_{s,\alpha,\bsgamma}$ is characterized by the smoothness
parameter $\alpha>1$ and a family of positive numbers
$\bsgamma=(\gamma_\setu)_{\setu\subset \bbN}$ called \emph{weights}, where
a positive weight $\gamma_\setu$ is associated with each subset $\setu
\subseteq \{1:s\}$. We fix the scaling of the weights by setting
$\gamma_\emptyset := 1$, so that the norm of a constant function in $H$
matches its $L_2$ norm.

It can easily be verified that if $\alpha$ is an even integer then the
norm can be rewritten as the norm in an ``unanchored'' weighted Sobolev
space of dominating mixed smoothness of order ${{\alpha}/{2}}$,
\begin{align}\label{eq:Hnorm}
\|f\|_H=\sqrt{\sum_{\setu\subseteq\{1:s\}}\frac{1}{(2\pi)^{\alpha |\setu|} \gamma_\setu}
	\int_{[0,1]^{|\setu|}}\!\!\;\bigg|\int_{[0,1]^{s-|\setu|}} \bigg(\prod_{j\in\setu} \frac{\partial^{\alpha/2}}{\partial y_j^{\alpha/2}}\bigg)
	f(\bsy)\,\rd\bsy_{-\setu}\bigg|^2\,\rd\bsy_\setu},
\end{align}
where $\bsy_\setu$ denotes the components of $\bsy$ with indices that
belong to the subset $\setu$, and $\bsy_{-\setu}$ denotes the components
that do not belong to $\setu$, and $|\setu|$ denotes the cardinality of $\setu$.

The important feature of the space $H$ is that it is an RKHS, with an explicitly known and analytically simple
reproducing kernel, namely
\[
K(\bsy,\bsy') \,:=\, K_{s,\alpha,\bsgamma}(\bsy,\bsy')
\,:=\, \sum_{\setu\subseteq{\{1:s\}}}\gamma_{\setu}
\prod_{j\in\setu}\eta_\alpha(y_j,y_j'),
\]
where
\[
\eta_{\alpha}(y,y') \,:=\, \eta_{\alpha}(y-y')
\,:=\, \sum_{h\neq0}\frac{\rme^{2\pi\rmi h(y-y')}}{|h|^{\alpha}}
\,=\, \sum_{h\neq0}\frac{\cos{2\pi h(y-y')}}{|h|^{\alpha}}.
\]
Note that the \emph{reproducing property}
\begin{equation} \label{eq:repro}
\langle f,K(\cdot,\bsy)\rangle_H \,=\, f(\bsy)
\qquad\mbox{for all $f\in H$ and all $\bsy \in [0,1]^s$},
\end{equation}
is easily verified.

Of special interest are even integer values of $\alpha$, because, when
$\alpha$ is even, $\eta_{\alpha}$ can be expressed in the especially simple
closed form
\[
\eta_{\alpha}(y,y') \,=\,
\frac{(2\pi)^{\alpha}}{(-1)^{\alpha/2+1}\alpha!}B_{\alpha}(\{y-y'\}),\qquad y,y'\in[0,1],
\]
where the braces indicate that $y-y'$ is to be replaced by its fractional
part in $[0,1)$, and $B_{\alpha}(y)$ is the Bernoulli polynomial of degree
$\alpha$. For example, for $\alpha=2$ and $\alpha=4$ we have
\[
B_{2}(y) \,=\, y^{2}-y+ \frac{1}{6}
\qquad\mbox{and}\qquad
B_{4}(y) \,=\, y^{4}-2y^{3}+y^{2}-\frac{1}{30}.
\]

\subsection{The kernel interpolant}

We are interested in approximating a given function $f\in H$ by an
approximation of the form
\begin{equation}\label{eq:form}
A^*_n(f) :=
f_n(\bsy) := f_{s,\alpha,\bsgamma,n,\bsz}(\bsy)
:= \sum_{k=1}^{n} a_{k}\,K\Big(\Big\{\frac{k\bsz}{n}\Big\},\bsy\Big),
\qquad\bsy\in[0,1)^s,
\end{equation}
where $\bsz\in\{1,\dots,n-1\}^{s}$, and the braces around the vector of
length $s$ indicate that each component of the vector is to be replaced by
its fractional part. The points
\begin{equation} \label{eq:lat_pts}
\bst_{k} \,:=\, \Big\{\frac{k\bsz}{n}\Big\} \qquad\mbox{for}\qquad k=1,\dots,n
\end{equation}
are the points of a lattice cubature rule of rank $1$, see \cite{SJ94}.
{In what follows, we omit these braces because functions we consider are, unless otherwise stated, periodic.}

In particular, we define $f_n \in H$ to be the function of the form
\eqref{eq:form} that \emph{interpolates} $f$ at the lattice points,
\begin{equation} \label{eq:interp}
f_n(\bst_{k}) \,=\,f(\bst_{k})
\qquad\mbox{for all}\qquad k=1,\dots,n,
\end{equation}
and refer to $f_n$ as the \emph{kernel interpolant} of $f$.

The coefficients $a_{k}$ in \eqref{eq:form} are given by the linear system
based on \eqref{eq:interp}
\begin{equation} \label{eq:lin_sys}
\sum_{k=1}^{n} \calK_{k,k'}\, a_{k} \,=\, f\left(\bst_{k'}\right)
\qquad\mbox{for all}\qquad k'=1,\dots,n,
\end{equation}
where
$
\calK_{k',k} \,=\, \calK_{k,k'} \,:=\, K(\bst_{k},\bst_{k'})$, $ k,k'=1,\ldots,n$.
Note that the matrix elements can be expressed, using periodicity, as
\[
\calK_{k,k'} \,=\, K \biggl({\frac{(k-k')\bsz}{n}},\bs0\biggr),
\]
where $\bs0$ is the $s$-vector of all zeroes. It follows that the $n\times
n$ matrix $\calK$ is a \emph{circulant matrix}, which contains only $n$
distinct elements, and can be diagonalised in a time of order $n\log n$ by
fast Fourier transform. This is a major motivation for using lattice
points.

\subsection{The kernel interpolant is the minimal norm interpolant}

The following property is a well known result for interpolation in a
reproducing kernel Hilbert space; for completeness we give a proof.

\begin{theorem}
	The kernel interpolant $f_n$ defined by \eqref{eq:form},
	\eqref{eq:lat_pts} and \eqref{eq:interp} is the minimal norm interpolant
	in~$H$.
\end{theorem}

\begin{proof}
	Denoting the linear span of the kernels with one leg at
	$\bst_{k},k=1,\ldots,n$ by
	\[
	P_n \,:=\, \mathrm{span}\{K(\bst_{k},\cdot):k=1,\ldots,n\},
	\]
	we observe the well known fact (see, e.g.,
	\cite{deBoor.C_Lynch_1966_spline_minimum,Golomb.M_Weinberger_1959_optimal}),
	that $f_n$ is the orthogonal projection of $f$ on $P_n$ with respect to
	the inner product $\langle\cdot,\cdot\rangle_H$, since from the reproducing property~\eqref{eq:repro} and the interpolation
	property \eqref{eq:interp} we have
	\[
	\langle f-f_n,K(\bst_{k},\cdot) \rangle_H
	\,=\, f(\bst_{k})-f_n(\bst_{k}) \,=\, 0
	\qquad\mbox{for all}\qquad k=1,\dots,n.
	\]
	In turn, there follows the Pythagoras theorem,
	\begin{equation} \label{eq:Pythag}
	\|f\|_H^{2} \,=\, \|f-f_n\|_H^{2}+\|f_n\|_H^{2},
	\end{equation}
	and the \emph{minimal norm} property of $f_n$,
	\[
	f_n \,=\, \mathrm{argmin}\big\{\|g\|_H \;:\; g\in H \mbox{ and }  g(\bst_{k})=f(\bst_{k}) \mbox{ for all } k=1,\ldots,n\big\},
	\]
	since if $g$ is any other interpolant of $f$ at the lattice points then
	\[
	\langle g-f_n,f_n\rangle_H
	\,=\, \sum_{k=1}^{n} a_{k}\,\langle g-f_n, K(\bst_{k},\cdot)\rangle_H \,=\, 0,
	\]
	and hence $
	\|g\|_H^{2} \,=\, \|g-f_n\|_H^{2}+\|f_n\|_H^{2}$,
	from which the uniqueness of the minimal norm interpolant also follows.
	\end{proof}

\subsection{The kernel interpolant is optimal for given function values}

In this subsection we show that the kernel interpolant $f_n$ defined by
\eqref{eq:form}, \eqref{eq:lat_pts} and \eqref{eq:interp} is optimal among
all approximations that use only the same function values of $f$, in the
sense of giving the least possible worst case error measured in any given
norm $\|\cdot\|_{W}$ such that $H\subset W$ for functions in $H$. This is
a special case of a general result for optimal recovery problems in
Hilbert spaces (see for example \cite[Example~1.1]{MicRiv77} and
\cite[Section~3]{MicRiv85}), but for completeness we give a short proof
here. Our proof follows the exposition of \cite[Proof of Theorem
13.5]{Wen05}, but suitably adapted to our setting. 

Let $A_n:H\to H$ be an algorithm (linear or non-linear) that uses as
information about the argument only its values at the points
\eqref{eq:lat_pts}, i.e., it is a mapping of the form $A_n(f)=\mathcal{I}_n
(f(\bst_{1}),\dots,f(\bst_{n}))$ for a mapping $\mathcal{I}_{{n}} \colon
\mathbb{R}^n\to H$. The worst case {$W$-error} for this algorithm is defined
by
\[ e^{\rm wor}(A_n;W) \,:=\, \sup_{f\in H, \, \|f\|_H\le 1} \| f - A_n(f)\|_W. \] \begin{theorem} \label{thm:optimal}
	Let $A_n:H\to H$ be an algorithm \textnormal{(}linear or
	non-linear\textnormal{)} such that $A_n(f)$ uses as information about $f$
	only its values $f(\bst_{1}),\ldots,f(\bst_{n})$ at the points
	\eqref{eq:lat_pts}. For $f\in H$, let $A^*_n(f) := f_n$ be the kernel
	interpolant defined by \eqref{eq:form}, \eqref{eq:lat_pts} and
	\eqref{eq:interp}. Then, for any normed space $W\supset H$ we have
	\begin{equation*}
	e^{\rm wor}(A^*_n;W) \,\le\, e^{\rm wor}(A_n;W).
	\end{equation*}
\end{theorem}
\begin{proof}
	Define $
	\mathcal{C}:=\{g\in H
	\;:\;  \|g\|_H\leq1\text{ and }g(\bst_{k})=0 \mbox{ for all }k=1,\ldots,n\}$.
	For any $g\in \mathcal{C}$ we have
	\begin{align} \label{eq:optimal_proof}
	\|g\|_{W}
	&\le \frac{1}{2}(\|g-A_n(0)\|_{W}+\|g+A_n(0)\|_{W})\nonumber \\
	&\le \max\big(\|g-A_n(0)\|_{W},\|g+A_n(0)\|_{W}\big) \nonumber \\
	&= \max\big(\|g-A_n(g)\|_{W},\|(-g)-A_n(-g)\|_{W}\big)
	\le e^{\rm wor}(A_n;W),
	\end{align}
	where in the penultimate step we used $g(\bst_{k})=0$ for all
	$k=1,\ldots,n$, from which it follows that $A_n(0)= A_n(g) = A_n(-g)$. For any
	$f\in H$ such that $\|f\|_H\leq1$, since $f_n$ is interpolatory, the
	Pythagoras theorem \eqref{eq:Pythag} implies $\|f-f_n\|_{H}\le 1$, and
	hence $f-f_n\in\mathcal{C}$. Thus it follows from \eqref{eq:optimal_proof}
	that
	\[
	\|f-A^*_n(f)\|_W \,=\, \|f-f_n \|_{W} \,\le\, e^{\rm wor}(A_n;W).
	\]
	The theorem now follows.
	\end{proof}

In the above result, we may, for example, take $W=L_{p}$ for any $1\leq
p\leq\infty$.

\section{Lower and upper error bounds}\label{sec:error-analysis}

\subsection{Lower bound on the worst case $L_{p}$ error {($1\leq p \leq \infty$)}}\label{sec:wce-lower-bound}

A recent paper \cite{BKUV17} showed (with a different definition of the
parameter $\alpha$) that the worst case $L_2$ error for an approximation
that uses the points of a rank-$1$ lattice cannot have an order of
convergence better than $n^{-\alpha/4}$ (with our definition of $\alpha$).
Bearing in mind that $H$ is a (Hilbert) space of functions of dominating
mixed smoothness of order ${{\alpha}/{2}}$, this is just half the rate
$n^{-\alpha/2}$ of the best approximation. Since the function space
setting in that paper is rather different from ours (here we use a Fourier
description and a so-called unanchored space, and have introduced weights)
we briefly reprove the main result here, obtaining a sharp lower bound
expressed in terms of the weights. Furthermore, in our setting we make the
result stronger by showing that the same lower bound holds for the worse
case $L_1$ error.

\begin{theorem}\label{thm:lowerbound}
	Let $s\ge 2$. Assume that the weights for the subsets of
	$\{1:s\}$ containing a single
	element satisfy ${\gamma_{\{j\}} } >0$ for all $j\in \{1:s\}$, and that
	$\bsz\in \{0,\ldots,n-1\}^s$ is given. 
	Let $A_n:H\to L_p$ be an
	algorithm
	\textnormal{(}linear or non-linear\textnormal{)} that uses information
	only at the lattice points \eqref{eq:lat_pts} and satisfies $A_n(0)=0$. Then
	for $1\leq p\leq \infty$ the worst case $L_{{p}}$ error for algorithm $A_n$
	satisfies
	\[ e^{\rm wor}(A_n;L_{{p}}) \,\ge\,
	\sqrt{\frac{2}{1/\gamma_{\{{j}\}} + 1/\gamma_{\{{k}\}}}} \,n^{-\alpha/4}
	\quad\text{ for any } j,k\in\{1:s\} \mbox{ with } j\ne k.
	\] In particular, if $\gamma_{\{1\}}\geq \gamma_{\{2\}}\geq\dotsb>0$ then
	\[
	e^{\rm wor}(A_n;L_{p})\geq\sqrt{{{2}/{\big(\gamma_{\{1\}}^{-1} +
				\gamma_{\{2\}}^{-1}\big)}}} \,n^{-\alpha/4}.\]
\end{theorem}

\begin{proof}
	Without loss of generality we assume $\gamma_{\{1\}}\geq
	\gamma_{\{2\}}\geq\dotsb>0$. The heart of the matter is that there exists
	a non-zero integer vector $\bsh^*$ of length $s$ in the $2$-dimensional
	set
	\[
	D_n \,:=\, \left\{(h_1,h_2,0,\ldots,0) \;:\;
	h_j\in \mathbb{Z},\,0\le |h_j| \le \lfloor \sqrt{n} \rfloor,\, j=1,2\right\},
	\]
	such that
	\begin{equation}\label{eq:dual_prop}
	\bsh^*\cdot\bsz \equiv 0 \pmod n.
	\end{equation}
	(In the language of dual lattices, see \cite{SJ94}, there exists a point
	of the dual lattice in $D_n\setminus \{\bszero\}$.) To prove this
	fact, we define $\widetilde{D}_{n}$, the positive quadrant of $D_{n}$, by
	\[
	\widetilde{D}_{n}\,:=\,
	\left\{ (h_{1},h_{2},0,\ldots,0)\;\colon\;h_{j}\in\mathbb{Z},\,0\le h_{j}\le\lfloor\sqrt{n}\rfloor,\,j=1,2\right\} ,
	\]
	noting that if $\bsh,\bsh'\in\widetilde{D}_{n}$ then $\bsh-\bsh'\in
	D_{n}$. Now define
	\[
	E_{n}(\bsz)\,:=\,\{ (\bsh\cdot\bsz\bmod n)\in\{0,\dots,n-1\}\ \colon\ \bsh\in\widetilde{D}_{n}\}.
	\]
	Since $|E_{n}(\bsz)|\leq n$ and
	$|\widetilde{D}_{n}|=(1+\lfloor\sqrt{n}\rfloor)^{2}>n$, it follows from
	the pigeonhole principle that two distinct elements of
	$\widetilde{D}_{n}$, say $\bsh$ and $\bsh'$, yield the same element of
	$E_{n}(\bsz)$; from this it follows that $\bsh^{*}:=\bsh-\bsh'$ satisfies
	\eqref{eq:dual_prop}. 
	A ``fooling function'' is then defined by
	\[
	q(\bsy) \,:=\, \mathrm{e}^{2\pi \mathrm{i} h_1^* \bse_1\cdot\bsy}
	-\mathrm{e}^{-2\pi \mathrm{i} h_2^* \bse_2\cdot\bsy}
	\,=\, \mathrm{e}^{-2\pi \mathrm{i} h_2^* \bse_2\cdot\bsy}
	\big(\mathrm{e}^{2\pi \mathrm{i} \bsh^*\cdot\bsy} - 1\big), \qquad {\bsy \in \mathbb{R}^s},
	\]
	where $\bse_1$ and $\bse_2$ are the unit vectors corresponding to
	variables $1$ and $2$.  By construction, $q$ vanishes at all the lattice
	points \eqref{eq:lat_pts}.  For this function, since the two terms in $q$
	are orthogonal with respect to the inner products in $H =
	H_{s,\alpha,\bsgamma}$, the squared $H$ norm satisfies
	\begin{align*}
	\|q\|_{H}^2 = r(h_1^*\bse_1) + {r}(-h_2^*\bse_2)
	\,=\, \frac{|h_1^*|^{\alpha}}{\gamma_{\{1\}}} + \frac{|-h_2^*|^{\alpha}}{\gamma_{\{2\}}}
	\,\le\, \left(\frac{1}{\gamma_{\{1\}}}+ \frac{1}{\gamma_{\{2\}}}\right)\, n^{\alpha/2}.
	\end{align*}
	On the other hand, the $L_{p}$ norm is bounded from below by
	\begin{align*}
	\|q\|_{L_{p}}\geq\|q\|_{L_{1}}
	&\,=\, \int_{[0,1]^{2}}|\mathrm{e}^{2\pi\mathrm{i}(h_{1}^{*}y_{1}+h_{2}^{*}y_{2})}-1|\,\mathrm{d}y_{1}\,\mathrm{d}y_{2}\\
	&\,=\, 2\int_{[0,1]^{2}}|\sin{(\pi(h_{1}^{*}y_{1}+h_{2}^{*}y_{2}))}|\,\mathrm{d}y_{1}\,\mathrm{d}y_{2}.
	\end{align*}
	This integrand is even with respect to $h_{1}^*$ and $h_{2}^*$ separately,
	so both $h_{1}^*$ and $h_{2}^*$ can be considered as non-negative. First
	assume that both $h_{1}^*$ and $h_{2}^*$ are positive, and partition the
	square into boxes of size $1/h_{1}^*\times1/h_{2}^*$. It is easy to see
	that each box gives the same contribution to the integral, and hence
	\begin{align*}
	2\int_{[0,1]^{2}}|&\sin{(\pi(h_{1}^{*}y_{1}+h_{2}^{*}y_{2}))}|\,\rd y_{1}\,\rd y_{2}
	\\
	&\,=\, 2\,{h_{1}^*h_{2}^*}\int_{0}^{1/{h_{1}^*}}\!\!\!\int_{0}^{1/{h_{2}^*}}|\sin{(\pi(h_{1}^{*}y_{1}+h_{2}^{*}y_{2}))}|
	\,\rd y_{1}\,\rd y_{2}\\
	&\,=\, 2\int_0^1\int_0^1|\sin(\pi(z_1+z_2))|\,\rd z_1 \,\rd z_2
	\,=\, \frac{4}{\pi}.
	\end{align*}
	(For the last step it may be useful to note that the
	integrand in the inner integral is 1-periodic, making the inner integral
	independent of $z_2$.)  If we have $h_{1}^*>0$ and $h_{2}^*=0$ or vice
	versa, we again have $\|q\|_{L_{1}}=4/\pi$. Since $\boldsymbol{h}^*$ is
	non-zero, we obtain
	\begin{equation} \label{eq:rhs}
	\frac{\|q\|_{L_{{p}}}}{\|q\|_H}
	\,\ge\, \frac{4/\pi}{\sqrt{1/\gamma_{\{1\}}+1/\gamma_{\{2\}}}}\, n^{-\alpha/4}.
	\end{equation}

	If we now define $g:= q/\|q\|_H$, then $g$ belongs to the unit ball in $H$
	and vanishes at all the points of the lattice \eqref{eq:lat_pts}, and
	$\|g\|_{L_{p}}$ is bounded {below} by the right-hand side of
	\eqref{eq:rhs}. Since $A_n(g)$ depends on $g$ only through its values at the
	lattice points, and $g$ vanishes at all those points, it follows that
	$A_n(g) = A_n(0) = 0$, with the last step following from the assumption on
	$A_n$.  From the definition of worst case error we conclude that
	$e^{\mathrm{wor}} (A_n,L_p)\ge \|g-A_n(g)\|_{L_p} = \|g\|_{L_p}$, which is
	bounded below by the right-hand side of \eqref{eq:rhs}, completing the
	proof.
	\end{proof}

\subsection{Upper bound on the worst case $L_2$ error}
In this section, we obtain explicit $L_{2}$ error bounds for the kernel
interpolant by using Theorem~\ref{thm:optimal} combined with error bounds
given for an explicit trigonometric polynomial approximation in
\cite{CKNS-part1,CKNS-part2} which extends the construction from
\cite{KSW06,KSW08} to general weights. (An alternative approach to obtain
an upper bound would be to use a ``reconstruction lattice'', see, e.g.,
\cite{BKUV17,KPV15,KMNN}.)

The lattice algorithm $A^\dagger_{n,M}$ applied to a target function $f\in H$
takes the form
\begin{equation} \label{eq:def-AzNM}
(A^\dagger_{n,M}(f))(\bsy) \,:=\, \sum_{\bsh\in\calA_{s}(M)}\bigg(
\frac{1}{n}\sum_{k=1}^{n}
f\biggl({\frac{k\bsz}{n}}\biggr)\rme^{-2\pi\rmi k\bsh\cdot\bsz/n}\bigg)
\rme^{2\pi\rmi\bsh\cdot\bsy},
\end{equation}
which is obtained by applying a lattice integration rule to the Fourier
coefficients in the orthogonal projection onto a finite index set defined
for some parameter $M>0$ by
\begin{equation} \label{eq:AdM}
\calA_{s}(M):=\{\bsh\in\bbZ^{s}:r (\bsh)\leq M\}.
\end{equation}
The error for this algorithm consists of the error from truncation to the
index set $\calA_s(M)$ together with the quadrature error from
approximating those Fourier coefficients with indices $\bsh\in\calA_s(M)$,
leading to a worst case $L_2$ approximating error bound of the form
\begin{align} \label{eq:balance1}
e^{\mathrm{wor}}(A^\dagger_{n,M};L_2) \,\le\, \bigg(\frac{1}{M} + M\, \calS_s(\bsz) \bigg)^{1/2}.
\end{align}
The quantity $\calS_s(\bsz)$ (see \cite{CKNS-part1} for details) can be
used as a search criterion in a component-by-component (CBC) construction
for finding suitable lattice generating vectors $\bsz$, and has the key
advantage that it does \emph{not} depend on the index set $\calA_s(M)$.
The analysis in \cite{CKNS-part1} together with the optimality of the
kernel interpolant (see Theorem~\ref{thm:optimal}) leads to the following
theorem.

\begin{theorem} \label{thm:wce}
	Given $s\ge 1$, $\alpha>1$, weights $(\gamma_\setu)_{\setu\subset\bbN}$
	with $\gamma_\emptyset := 1$, and prime~$n$, the worst case $L_2$
	approximation error of the kernel interpolant $A^*_n(f) = f_n$ defined by
	\eqref{eq:form}, \eqref{eq:lat_pts} and \eqref{eq:interp}, using the
	generating vector $\bsz$ obtained from the CBC construction with search
	criterion $\calS_s(\bsz)$ in \cite{CKNS-part1,CKNS-part2}, satisfies for
	all $\lambda\in (\frac{1}{\alpha},1]$,
	\begin{align} \label{eq:final-err}
	e^{\mathrm{wor}}(A^*_{n};L_2)
	&\le \sqrt{2}\, \big[\calS_s(\bsz) \big]^{1/4} \\
	&\le
	\frac{\kappa}{n^{1/(4\lambda)}} \bigg(
	\sum_{\setu\subseteq\{1:s\}} \!\!
	\max(|\setu|,1)\,\gamma_{\setu}^\lambda\, [2\zeta(\alpha\lambda)]^{|\setu|}\bigg)^{1/{(2\lambda)}},
	\nonumber
	\end{align}
	with $\kappa :=
	\sqrt{2}\,[\max(6,2.5+2^{2\alpha\lambda+1})]^{{{1}/{({4\lambda})}}}$
	and $\zeta(x):=\sum_{k=1}^\infty k^{-x}$ denoting the Riemann zeta
	function
	for $x>1$. Hence
	\[
	e^{\mathrm{wor}}(A^*_n;L_2) \,=\, \calO(n^{-\alpha/4 + \delta})
	\qquad
	\text{for every }\
	\delta\in(0,\alpha/4),
	\]
	where the implied constant depends on $\delta$ but is independent of
	$s$
	provided that
	\[
	\sum_{\substack{\setu\subset\bbN\\ \,|\setu|<\infty}} \max(|\setu|,1)\,\gamma_{\setu}^{\frac{1}{\alpha-4\delta}}\,
	[2\zeta\big(\tfrac{\alpha}{\alpha-4\delta}\big)]^{|\setu|}
	\,<\, \infty.
	\]
\end{theorem}

\begin{proof}
	The optimality of the kernel interpolant established in
	Theorem~\ref{thm:optimal} means that $e^{\mathrm{wor}}(A^*_n;L_2) \le
	e^{\mathrm{wor}}(A^\dagger_{n,M};L_2)$ for all $M$, and therefore the
	upper bound in \eqref{eq:balance1} also serves as an upper bound for the
	kernel interpolant. It is easy to verify that the bound in
	\eqref{eq:balance1} can be minimized by setting $M =
	\sqrt{1/\calS_s(\bsz)}$, leading to \eqref{eq:final-err}. The subsequent
	bound follows from \cite[Theorem~3.5]{CKNS-part1}. The big-$\calO$ bound
	is then obtained by taking $\lambda = 1/(\alpha-4\delta)$.
	\end{proof}

From this result (which by Theorem~\ref{thm:lowerbound} is almost best possible with respect
to {the} order of convergence) we immediately obtain an error bound for the
kernel interpolant.

\begin{theorem} \label{thm:wce2}
	Under the conditions of Theorem~\ref{thm:wce}, and with lattice generating
	vector $\bsz$ obtained by the CBC construction in
	\cite{CKNS-part1,CKNS-part2}, for any $f\in H$, we have for the kernel
	interpolant $f_n$ defined by \eqref{eq:form}, \eqref{eq:lat_pts} and
	\eqref{eq:interp},
	\begin{align*}
	\|f-f_n\|_{L_{2}}
	&\,\le\, \frac{\kappa}{n^{1/(4\lambda)}} \bigg(
	\sum_{\setu\subseteq\{1:s\}} \max(|\setu|,1)\,\gamma_{\setu}^\lambda\, [2\zeta(\alpha\lambda)]^{|\setu|}
	\bigg)^{1/(2\lambda)}\,
	\|f\|_H.
	\end{align*}
\end{theorem}

We stress again that the CBC construction in \cite{CKNS-part1,CKNS-part2}
does \emph{not} require the explicit construction of the index set
$\calA_s(M)$ in order to determine an appropriate generating vector
$\bsz$. However, the expression $\calS_s(\bsz)$ (see \cite{CKNS-part1} for
details) used as the search criterion does depend in a complicated way on
the weights $\gamma_\setu$, and therefore the target dimension $s$ needs
to be fixed at the start of the CBC construction (except for the case of
product weights). For weights with no special structure, the computational
cost will be exponentially large in $s$. We consider some special forms of
weights:
\begin{itemize}
	\item[\textbullet] \textbf{Product weights}: $\gamma_\setu \,=\, \prod_{j\in\setu}
	\gamma_j$, specified by one sequence ${(\gamma_j)}_{j\ge 1}$.
	\item[\textbullet] \textbf{POD weights} (product and order dependent):
	$\gamma_\setu \,=\, \Gamma_{|\setu|} \prod_{j\in\setu} \gamma_j$,
	specified by two sequences ${(\Gamma_\ell)}_{\ell\ge 0}$ and
	${(\gamma_j)}_{j\ge
		1}$.
	\item[\textbullet] \textbf{SPOD weights} (smoothness-driven product and order
	dependent) with degree $\sigma\ge 1$:
	\begin{align*}
	\gamma_\setu \,=\, \sum_{\bsnu_\setu\in \{1:\sigma\}^{|\setu|}} \Gamma_{|\bsnu_\setu|} \prod_{j\in\setu} \gamma_{j,\nu_j},
	\end{align*}
	specified by the sequences ${(\Gamma_\ell)}_{\ell\ge 0}$ and
	${(\gamma_{j,\nu})}_{j\ge 1}$ for each $\nu=1,\ldots,\sigma$,
	where $|\bsnu_\setu| :=\sum_{j\in\setu} \nu_j$.
\end{itemize}
Fast CBC construction of lattice generating vector for $L_2$ approximation
has the cost of
\begin{align*}
&\calO\big(s\,n\log(n)\big) && \mbox{for product weights}, \\
&\calO\big(s\,n\log(n) + s^2\log(s) n\big) && \mbox{for POD weights}, \\
&\calO\big(s\,n\log(n) + s^3\sigma^2 n\big) && \mbox{for SPOD weights with degree $\sigma\ge 2$},
\end{align*}
plus storage cost and  pre-computation cost for POD and SPOD
weights, see~\cite{CKNS-part2}.

\section{Application to PDEs with random coefficients}\label{sec:PDE}

As an application, we apply our kernel interpolation scheme to a
forward uncertainty quantification problem, namely, a PDE problem with an
uncertain, {\em	periodically} parameterized diffusion coefficient, fitting
the theoretical framework considered in the preceding sections. The
kernel interpolant can be postprocessed with low computational cost to
obtain statistics of the PDE solution itself or functionals of the
solution for uncertainty quantification.

Letting $D\subset\mathbb{R}^d$, $d\in\{1,2,3\}$, be a bounded domain with
Lipschitz boundary, we consider the problem of finding $u\colon D\times
\Omega\to\mathbb{R}$ that satisfies
\begin{align}\label{eq:pdestrong}
-\nabla\cdot (a(\bsx,\omega)\,\nabla u(\bsx,\omega)) &\,=\, q(\bsx),&& \bsx\in D,\\
u(\bsx,\omega) &\,=\, 0, && \bsx\in \partial D,\label{eq:pdestrong2}
\end{align}
for almost all events $\omega\in\Omega$ in the probability space
$(\Omega,\mathscr{A},\mathbb{P})$ with
\begin{equation}\label{qqq}
a(\bsx,\omega) \,=\,
a_0(\bsx)+\frac{1}{\sqrt{6}}\sum_{j\geq 1}\sin(2\pi Y_j(\omega))\,\psi_j(\bsx),\qquad \bsx\in D,~\omega\in\Omega,
\end{equation}
where $a_0\in L_\infty(D)$, $\psi_j\in L_\infty(D)$ for all $j\geq 1$ are
such that $\sum_{j\ge 1}|\psi_j(\bsx)| < \infty$ for any $\bsx\in D$,
and $Y_1,Y_2,\ldots$~are i.i.d.~random variables uniformly distributed on
$[-\frac12,\frac12]$.
This type of random field is not new in the context
of uncertainty quantification. Indeed,
the random variable $\sin(2\pi Y_j(\omega))$ induces the arcsine measure
as its distribution: for if $Y(\omega)$
is uniformly distributed on $[-\tfrac{1}{2}, \tfrac{1}{2}]$, then
$Z(\omega):=\sin(2\pi Y(\omega))$  has the probability density
$\tfrac{1}{\pi}\tfrac{1}{\sqrt{1-z^{2}}}$ on $[-1,1]$.
Thus, $a$ is identical, up to the law
{to the random field
	\begin{equation}\label{ppp}
	\hat{a}(\bsx,\omega)=a_0(\bsx)+\frac{1}{\sqrt{6}}\sum_{j\ge 1}Z_j(\omega)\psi_j(\bsx)
	\end{equation}
	with $Z_j$ i.i.d.\ random variables with arcsine distribution on $[-1,1]$. Expression \eqref{ppp} would be the starting point for deriving a polynomial chaos approximation \cite{XK02} of the solution in terms of Chebyshev polynomials of the first kind \cite{Rauhut.H_Schwab_2016_Cheb}. In this paper, however, we want to exploit periodicity, hence we consider rather the formulation \eqref{qqq} and a different approximation method based on kernel interpolation.}

{
	Since the expression \eqref{qqq} is periodic in the random variable $Y_j$, we can shift those random variables so that their range is $[0,1]$ instead of $[-\frac{1}{2},\frac{1}{2}]$, i.e., we consider the equivalent parametric space
	$
	U \,:=\, [0,1]^{\mathbb{N}}.
	$}
Let $\mathcal{B}(U)$ be the Borel $\sigma$-algebra
corresponding to the product topology on $U=[0,1]^{\mathbb{N}}$,
and equip $(U,\mathcal{B}(U))$ with the product uniform measure; see, for example, \cite{Rogers.L.C.G_Williams_2000_book_I} for details. The weak formulation of~\eqref{eq:pdestrong}--\eqref{eq:pdestrong2} can
then be stated parametrically as: for $\bsy\in U$, find $u(\cdot,\bsy)\in
H_0^1(D)$ such that
\begin{align}
\int_D {a}(\bsx,\bsy)\,\nabla u(\bsx,\bsy)\cdot \nabla \phi(\bsx)\,{\rm d}\bsx
\,=\,
\langle q,\phi\rangle_{H^{-1}(D),H^1_0(D)},
\quad \forall\phi\in H_0^1(D),\label{eq:pdeweak}
\end{align}
where the datum $q\in H^{-1}(D)$ is fixed and the diffusion coefficient is given by
\begin{equation} \label{eq:a-param}
a(\bsx,\bsy) \,=\,
a_0(\bsx)+\frac{1}{\sqrt{6}}\sum_{j\geq 1}\sin(2\pi y_j)\,\psi_j(\bsx),
\qquad \bsx\in D,~\bsy \in U.
\end{equation}
Here $H^1_0(D)$ denotes the subspace of the $L_2$-Sobolev space $H^1(D)$
with vanishing trace on $\partial D$, and $H^{-1}(D)$ denotes the
topological dual of $H^1_0(D)$, and $\langle\cdot,\cdot\rangle_{H^{-1}(D),H^1_0(D)}$ denotes the
duality pairing between $H^{-1}(D)$ and $H_0^1(D)$. We endow the Sobolev
space $H_0^1(D)$ with the norm $\|v\|_{H_0^1(D)}:=\|\nabla v\|_{L_2(D)}$.

Since we now have two sets of variables
$\bsx\in D$ and $\bsy\in U$, from here on we will make the domain $D$ and
$U$ explicit in our notation. We state the following assumptions and refer
to them as they become needed:
\setlength{\leftmargini}{2.5em}
\begin{itemize}
	\item[(A1)] $a_0\in L_\infty(D)$, $\psi_j\in L_\infty(D)$ for all
	$j\geq 1$, and $\sum_{j\ge 1} \|\psi_j\|_{L_\infty(D)}< \infty$;
	\item[(A2)] there exist positive constants $a_{\min}$ and $a_{\max}$ such that $0<a_{\min}\leq a(\bsx,\bsy)\leq a_{\max}<\infty$ for all $\bsx\in D$ and $\bsy\in U$;
	\item[(A3)] $\sum_{j\geq 1}\|\psi_j\|_{L_\infty(D)}^p<\infty$ for some
	$0<p<1$;
	\item[(A4)] $a_0\in W^{1,\infty}(D)$ and $\sum_{j\geq 1}\|\psi_j\|_{W^{1,\infty}(D)}<\infty$, where
	$$
	\|v\|_{W^{1,\infty}(D)}:=\max\{\|v\|_{L_\infty(D)},\|\nabla v\|_{L_\infty(D)}\};
	$$
	\item[(A5)] $\|\psi_1\|_{L_\infty(D)}\geq \|\psi_2\|_{L_\infty(D)}\geq \cdots$;
	\item[(A6)] the physical domain $D\subset\mathbb{R}^d$, $d\in \{1,2,3\}$, is a convex and bounded polyhedron with plane faces.
\end{itemize}

{\sloppy Let assumptions \textup{(A1)} and \textup{(A2)} be in effect. Then the Lax--Milgram lemma~\cite{ciarlet} implies unique solvability of the problem~\eqref{eq:pdeweak} for all $\bsy\in U$, with the solution satisfying the {\em a priori} bound
	\begin{align}
	\|u(\cdot,\bsy)\|_{H_0^1(D)}\leq \frac{\|q\|_{H^{-1}(D)}}{a_{\min}}\qquad\text{for all}~\bsy\in U.\label{eq:apriori}
	\end{align}
}Moreover, from the recent paper {\cite[Theorem 2.3]{KKS}} we know, after
differentiating the PDE \eqref{eq:pdeweak}, that the mixed derivatives of
the PDE solution are 1-periodic and bounded by
\begin{align} \label{eq:pde-der}
\|\partial_{\bsy}^{\bsnu} u(\cdot,\bsy)\|_{H^1_0(D)} \le
\frac{\|q\|_{H^{-1}(D)}}{a_{\min}}(2\pi)^{|\boldsymbol{\nu}|}
\sum_{\bsm\leq\boldsymbol{\nu}}|\bsm|!
\prod_{{j}\geq 1}
(b_j^{m_j}S(\nu_{j},m_{j}))
\end{align}
for all $\bsy\in U$ and all multindices $\bsnu\in\mathbb{N}_0^\infty$
with finite order $|\bsnu|:= \sum_{j\ge 1} \nu_j <\infty$, and we define
\begin{align} \label{eq:def-bj}
b_j \,:=\, \frac{1}{\sqrt{6}}\frac{\|\psi_j\|_{L_\infty(D)}}{ a_{\min}}\qquad
\text{ for all }~j\geq 1.
\end{align}
Furthermore, $S(\sigma,m)$ denotes the \emph{Stirling number of the
	second kind} for integers $\sigma\geq m\geq 0$, with the convention that
$S(\sigma,0)=\delta_{\sigma,0}$. In \cite{KKS} we considered a function
space with respect to $\bsy$ with a supremum norm rather than an
$L_2$-based norm, so here we need to write down the relevant $L_2$-based
norm bound instead. Moreover, we want to approximate the solution $u$
directly, rather than a bounded linear functional $G(u)$ of the PDE
solution.

For our proposed approximation scheme, we require the target function to
be pointwise well-defined with respect to both the physical variable and
the parametric variable. In terms of our PDE application, this can be
achieved either by assuming additional regularity of both the diffusion
coefficient $a$ and the source term $q$ or, alternatively, by analyzing
instead the construction of the kernel interpolant for the finite element
approximation of $u$ (which is naturally pointwise well-defined
everywhere). Here we focus on the latter case, in which the kernel
interpolant is crafted for the finite element approximation of $u$. This
is also the setting that arises in practical computations, where one only
ever has access to a numerical approximation of the solution
to~\eqref{eq:pdeweak}, with the diffusion coefficient~\eqref{eq:a-param}
truncated to a finite number of terms. To this end, we split our analysis
into three parts: \emph{dimension truncation error}, \emph{finite element
	error}, and \emph{kernel interpolation error}.

\subsection{Dimension truncation error}\label{sec:dimtruncerror}

In anticipation of the forthcoming discussion we define the dimensionally
truncated solution of~\eqref{eq:pdeweak} as
\[
u_s(\cdot,\bsy)
\,:=\, u_s(\cdot,(y_1,\ldots,y_s))
\,:=\, u(\cdot,(y_1,\ldots,y_s,0,0,\ldots)),\qquad \bsy\in U.
\]
Moreover, let us introduce the shorthand notations $U_s:=U_{\leq s}:=[0,1]^s$, $U_{>s}:=\{(y_j)_{j\geq s+1}: y_j\in [0,1]\}$, and $\bsy_{>s}:=(y_{s+1},y_{s+2},\ldots)$.

For an $\bbR$-valued function on $U$ that is Lebesgue integrable with  respect to the uniform measure on $\mathcal{B}(U)$, we use the notation $\int_{U}F(\bsy)\,{\rm d}\bsy$ for the integral of $F$ over $U$. Similarly, for an integrable function $\tilde{F}$ on $U_{> s}$, we denote the integral over $U_{> s}$ with respect to the uniform measure by  $\int_{U_{> s}}\tilde{F}(\bsy_{>s})\,{\rm d}\bsy_{>s}$.

Arguing as in~\cite[Theorem~5.1]{Kuo.F_Schwab_Sloan_2012_SINUM}, it is  not difficult to see that \[\sup_{\bsy\in U}\|u(\cdot,\bsy)-u_s(\cdot,\bsy)\|_{H_0^1(D)}=	\mathcal{O}(s^{-1/p+1})\] holds under assumptions (A1)--(A3) and (A5).
In what follows, we consider the dimension truncation error in the $L_2$-norm in the stochastic parameter, and establish the rate $\mathcal{O}(s^{-{1}/{p}+{1}/{2}})$, which is one half order better.
This case does not appear to have been considered in the existing literature.
Notably, this rate is only half that of the rate proved in~\cite{KKS} for integration problem with respect to $\bsy$:
\begin{equation*}
\bigg|\int_U G(u(\cdot,\bsy)-u_s(\cdot,\bsy))\,{\rm d}\bsy\bigg|=\mathcal{O}(s^{-2/p+1}),\qquad G\in H^{-1}(D).
\end{equation*}

\begin{sloppypar}
	We will establish a dimension truncation error for a general class of parametrized random fields that includes~\eqref{eq:a-param},
	without the periodicity assumption. Our proof adapts the argument by Gantner~\cite{gantnermcqmc2018} to the  $L^2(U;H^1_0(D))$-norm estimate.
\end{sloppypar}
\begin{theorem}\label{thm:dimensiontruncation}
	Suppose that \textup{(A1)}, \textup{(A3)} and \textup{(A5)}
	hold. Let $\xi:[0,1]\to\mathbb{R}$ be an  $L_\infty([0,1])$-function
	such that
	\begin{equation}
	\int_{0}^{1}\xi(y)\,{\rm d}y=0.\label{eq:property-xi}
	\end{equation}
	Suppose further that the function
	\begin{equation}
	a(\bsx,\bsy)\,=\,a_{0}(\bsx)+\sum_{j\geq1}\xi(y_{j})\,\psi_{j}(\bsx),\qquad\bsx\in D,~\bsy\in U,
	\label{eq:a-tilde}
	\end{equation}
	satisfies $\textup{(A2)}$. Then for any $s\in\mathbb{N}$, there
	exists a constant $C>0$ such that
	\begin{align*}
	\sqrt{\int_{U}\int_{D}(u(\bsx,\bsy)-u_{s}(\bsx,\bsy))^{2}\,\mathrm{d}\bsx\,\mathrm{d}\bsy}
	&\leq c_{D}\sqrt{\int_{U}\int_{D}|\nabla(u-u_{s})|^{2}\,{\rm d}\bsx\,{\rm d}\bsy}\\
	&\leq\,C\,\|q\|_{H^{-1}(D)}\,s^{-(\frac{1}{p}-\frac{1}{2})},
	\end{align*}
	where $u\in H_{0}^{1}(D)$ denotes the solution of the equation \eqref{eq:pdeweak}
	but with {$a(\bsx, \bsy)$ given by}  \eqref{eq:a-tilde}, $u_{s}\in H_{0}^{1}(D)$
	denotes the corresponding dimensionally truncated solution, {$c_D>0$ is the Poincar\'{e} constant of the embedding
		$H_0^1(D)\hookrightarrow L_2(D)$,} and the
	constant $C>0$ is independent of $s$ and $q$.
\end{theorem}

\begin{proof}
	We begin by introducing some helpful notations. For $\bsy\in U$,
	let us define the operators $B,B^{s}\colon H_{0}^{1}(D)\to H^{-1}(D)$
	by
	\begin{align*}
	B\,:=\,B(\bsy)\,:=\,B_{0}+\sum_{k=1}^{\infty}\xi(y_{k})B_{k}\quad\text{and}\quad B^{s}\,:=\,B^{s}(\bsy)\,:=\,B_{0}+\sum_{k=1}^{s}\xi(y_{k})B_{k},
	\end{align*}
	where the operators $B_{k}\colon H_{0}^{1}(D)\to H^{-1}(D)$ are defined
	by \[\langle B_{0}v,w\rangle_{H^{-1}(D),H_{0}^{1}(D)}:=\langle a_{0}\nabla v,\nabla w\rangle_{L_{2}(D)}\]
	and $\langle B_{k}v,w\rangle_{H^{-1}(D),H_{0}^{1}(D)}:=\langle\psi_{k}\nabla v,\nabla w\rangle_{L_{2}(D)}$
	for $v,w\in H_{0}^{1}(D)$ and $k\geq1$. This allows the equation \eqref{eq:pdeweak} with the coefficient $a$ given by~\eqref{eq:a-tilde}
	to be written as $Bu=q$. It is easy to see that the assumptions (A1)
	{and (A2)} ensure that both $B(\bsy)$ and $B^{s}(\bsy)$ are
	boundedly invertible linear maps for all $\bsy\in U$, with the norms
	of $B$ and ${B^{s}}$ both bounded by $a_{{\max}}$, and the
	norms of both $B^{-1}$ and ${(B^{s})^{-1}}$ bounded by $a_{{\min}}^{-1}$.
	Thus we can write $u:=u(\bsy):=B^{-1}q$ and $u_{s}:=u_{s}(\bsy):=(B^{s})^{-1}q$
	for all $\bsy\in U$.

	Only in this proof, we redefine \eqref{eq:def-bj} by  $b_{j}:=\|\xi\|_{\infty}\|\psi_{j}\|_{L_{\infty}(D)}/a_{\min}$,
	with $\|\xi\|_{\infty}:=\|\xi\|_{L_\infty([0,1])}$.
	Notice that with $\xi=\frac1{\sqrt{6}}\sin(2\pi\cdot)$ we recover \eqref{eq:def-bj}.
	Let $s'\in\mathbb{Z}_{+}$
	be such that
	\begin{equation}
	\sum_{j=s'+1}^{\infty}b_{j}<\frac{1}{2}.\label{eq:btail}
	\end{equation}
	Without loss of generality, we can assume that $s\ge s'$ since the
	assertion in the theorem can subsequently be extended to all values
	of $s$ by making a simple adjustment of the constant $C>0$ (see
	the end of the proof). Then for all $j\ge s'+1$ and all $s\ge s'$
	we have
	\begin{align}
	& b_{j}<\frac{1}{2},\label{eq:bbound}\\
	& \sup_{\bsy\in U}\|(B^{s})^{-1}(B-B^{s})\|_{H_{0}^{1}(D)\to H_{0}^{1}(D)}\leq\sum_{j=s+1}^{\infty}b_{j}<\frac{1}{2}<1.\label{eq:opbound}
	\end{align}
	The bound~\eqref{eq:opbound} permits the use of a Neumann series
	expansion
	\begin{align}
	u&-u_{s}  =B^{-1}q-u_{s}=[I+(B^{s})^{-1}(B-B^{s})]^{-1}(B^{s})^{-1}q-u_{s}\nonumber \\
	& =\sum_{k=0}^{\infty}(-(B^{s})^{-1}(B-B^{s}))^{k}(B^{s})^{-1}q-u_{s}\nonumber \\
	& =\sum_{k=1}^{\infty}(-(B^{s})^{-1}(B-B^{s}))^{k}u_{s}=\sum_{k=1}^{\infty}(-1)^{k}\bigg(\sum_{i=s+1}^{\infty}\xi(y_{i})(B^{s})^{-1}B_{i}\bigg)^{k}u_{s},\label{eq:Neumann}
	\end{align}
	where it is assumed that the product symbol respects the non-commutative
	nature of the operators $(B^{s})^{-1}B_{j}$, $j\geq1$.

	{Our strategy is to}  estimate first
	\[
	S:=\int_{U}\int_{D}a_{s}(\bsx,\bsy)|\nabla(u-u_{s})|^{2}\,{\rm d}\bsx\,{\rm d}\bsy
	\]
	and {then deduce} by the Poincar\'{e} inequality $\|u\|_{L_{2}(D)}\leq c_{D}\|u\|_{H_{0}^{1}(D)}$,
	with $c_{D}>0$ depending only on the domain $D$, together with uniform coercivity,
	that
	\[
	\int_{U}\int_{D}(u-u_{s})^{2}\,{\rm d}\bsx\,{\rm d}\bsy\leq c_{D}^{2}\int_{U}\int_{D}|\nabla(u-u_{s})|^{2}\,{\rm d}\bsx\,{\rm d}\bsy\leq\frac{c_{D}^{2}}{a_{\min}}S.
	\]
	Let {$\mathscr{B}_s: H_0^1(D) \to H_0^1(D)$  be defined by}
	\[
	{\mathscr{B}_{s}}(\bsy):=\sum_{i=s+1}^{\infty}\xi(y_{i})(B^{s}(\bsy))^{-1}B_{i},
	\]
	and observe that ${\mathscr{B}_{s}}$ is self-adjoint with respect
	to the inner product
	\[
	\langle v,w\rangle_{\bsy_{\leq s}}
	\!
	:=
	\!\!\!\;
	\int_{D}
	\!\!\;\!\!\;\!
	a\bigl(\bsx,\!\!\;(\bsy_{\leq s},0,\dots)\bigr)\nabla v(\bsx)\cdot\nabla w(\bsx)\,\mathrm{d}\bsx
	=\langle B^s\!\!\;(\bsy)v,w\rangle_{H^{-1}(D),H^1_0(D)}
	.
	\]
	Indeed, {for any $v,w\in H_{0}^{1}(D)$ we have}
	\begin{align*}
	&\langle{\mathscr{B}_{s}}(\bsy)v,w\rangle_{\bsy_{\leq s}}  =
	\sum_{i=s+1}^{\infty}\xi(y_{i})\langle B^{s}(\bsy)(B^{s}(\bsy))^{-1}B_{i}v,w\rangle_{H^{-1},H_{0}^{1}}\\
	&
	=\sum_{i=s+1}^{\infty}\xi(y_{i})\!
	\int_{D}\psi_{i}\nabla v\cdot\nabla w\,\mathrm{d}\bsx
	=\!\!
	\sum_{i=s+1}^{\infty}\xi(y_{i})\langle B_{i}w,v\rangle_{H^{-1},H_{0}^{1}} =\langle{\mathscr{B}_{s}}(\bsy)w,v\rangle_{\bsy_{\leq s}}.
	\end{align*}
	Hence, from \eqref{eq:Neumann} we have
	\begin{align}
	\langle &u-u_{s},  u-u_{s}\rangle_{\bsy_{\leq s}}=\sum_{k=1}^{\infty}\sum_{\ell=1}^{\infty}(-1)^{k+\ell}\langle{\mathscr{B}_{s}^{k}}u_{s},{\mathscr{B}_{s}^{\ell}}u_{s}\rangle_{\bsy_{\leq s}}\nonumber\\
	& =\sum_{k=1}^{\infty}\sum_{\ell=1}^{\infty}(-1)^{k+\ell}\langle{\mathscr{B}_{s}^{k+\ell}}u_{s},u_{s}\rangle_{\bsy_{\leq s}} =\sum_{m=2}^{\infty}(-1)^{m}(m-1)\langle{\mathscr{B}_{s}^{m}}u_{s},u_{s}\rangle_{\bsy_{\leq s}}\nonumber\\
	& =\sum_{m=2}^{\infty}(-1)^{m}(m-1)\nonumber\\
	& \phantom{spaces}\times\!
	\sum_{\boldsymbol{\eta}\in\{s+1:\infty\}^{m}}\biggl(\prod_{j=1}^{m}\frac{\xi(y_{\eta_{j}})}{{\|\xi\|_{\infty}}}\biggr)\bigg\langle\prod_{j=1}^{m}{\|\xi\|_{\infty}}(B^{s}(\bsy)^{-1}B_{\eta_{j}})u_{s},u_{s}\bigg\rangle_{\bsy_{\leq s}},
	\label{eq:inneryles}
	\end{align}
	where we used the notation $\sum_{\boldsymbol{\eta}\in\{s+1:\infty\}^{m}}:=\lim_{\tilde{s}\to\infty}\sum_{\boldsymbol{\eta}\in\{s+1:\tilde{s}\}^{m}}$, and the latter product is assumed to respect the non-commutative
	nature of the operators. Introducing
	\[
	\bsnu(\boldsymbol{\eta}):=(\nu_{i}(\boldsymbol{\eta}))_{i\geq{1}}:=(\#\{j=1,\ldots,m:\eta_{j}=i\})_{i\geq{1}}
	\]
	for each $\boldsymbol{\eta}\in\{s+1,s+2,\dots\}^{m}$,
	we have $\nu_{i}(\boldsymbol{\eta})=0$, $i=1,\dots,s$,
	$|\bsnu(\boldsymbol{\eta})|:=\sum_{i=1}^{\infty}\nu_{i}(\boldsymbol{\eta})=m$,
	and
	\[
	\prod_{j=1}^{m}\frac{\xi(y_{\eta_{j}})}{{\|\xi\|_{\infty}}}=\prod_{i=s+1}^{\infty}{\biggl(\frac{\xi(y_{i})}{\|\xi\|_{\infty}}\biggr)^{\nu_{i}(\boldsymbol{\eta})}}.
	\]
	Define
	\[
	c_{\bsnu}:=\bigg|\int_{U_{>s}}\prod_{i\in{\rm supp}(\bsnu)}
	{\biggl(\frac{\xi(y_{i})}{\|\xi\|_{\infty}}\biggr)^{\nu_{i}}}\,{\rm d}\bsy_{>s}\bigg| \le 1,\]
	and note from \eqref{eq:property-xi} that $c_\bsnu = 0$ if  for some $i \in \supp(\bsnu)$ we have $\nu_i  = 1$.
	Then we have, using~\eqref{eq:inneryles},
	\begin{align*}
	&
	\int_{U}\int_{D}|\nabla(u-u_{s})|^{2}\,{\rm d}\bsx\,{\rm d}\bsy
	\leq\frac{
		{1}
	}{a_{\min}}\int_{U}\langle u-u_{s},u-u_{s}\rangle_{\bsy_{\leq s}}\,{\rm d}\bsy\\
	& =\frac{
		{1}
	}{a_{\min}}\sum_{m=2}^{\infty}(-1)^{m}(m-1)\!\!
	\sum_{\boldsymbol{\eta}\in\{s+1:\infty\}^{m}}\int_{U_{>s}}\biggl(\prod_{i=s+1}^{\infty}{\biggl(\frac{\xi(y_{i})}{\|\xi\|_{\infty}}\biggr)^{\nu_{i}(\boldsymbol{\eta})}}\biggr)\,{\rm d}\bsy_{>s}
	\\
	&\qquad\qquad\qquad\qquad\qquad\qquad\qquad\
	\times
	\int_{U_{\leq s}}\!\!\!\biggl\langle\prod_{j=1}^{m}{\|\xi\|_{\infty}}((B^{s})^{-1}B_{\eta_{j}})u_{s},u_{s}\biggr\rangle_{\!\!\!\bsy_{\leq s}}\!{\rm d}\bsy_{\leq s}\\
	& \leq
	\frac{
		{1}
	}{a_{\min}}\sum_{m=2}^{\infty}(m-1)\sum_{\boldsymbol{\eta}\in\{s+1:\infty\}^{m}}
	c_{\boldsymbol{\nu}(\boldsymbol{\eta})}\\
	&\qquad\qquad\qquad\qquad\qquad\qquad
	\times
	\bigg|\int_{U_{\leq s}}\biggl\langle\prod_{j=1}^{m}{\|\xi\|_{\infty}}((B^{s})^{-1}B_{\eta_{j}})u_{s},u_{s}\biggr\rangle_{\!\!\!\bsy_{\leq s}}\!{\rm d}\bsy_{\leq s}\bigg|,
	\end{align*}
	which can be further bounded by
	\begin{align*}
	&
	\leq\frac{
		{1}
	}{a_{\min}}
	\sum_{m=2}^{\infty}(m-1)
	\!\sum_{\boldsymbol{\eta}\in\{s+1:\infty\}^{m}}
	\!
	c_{\boldsymbol{\nu}(\boldsymbol{\eta})}
	a_{\max}\bigg\|\prod_{j=1}^{m}{\|\xi\|_{\infty}}((B^{s})^{-1}B_{\eta_{j}})u_{s}\bigg\|_{H_{0}^{1}}\|u_{s}\|_{H_{0}^{1}}\\
	& \leq\frac{
		{1}
	}{a_{\min}}\sum_{m=2}^{\infty}(m-1)\sum_{\boldsymbol{\eta}\in\{s+1:\infty\}^{m}}
	c_{\boldsymbol{\nu}(\boldsymbol{\eta})}
	a_{\max}\bigg(\prod_{j=1}^{m}b_{\eta_{j}}\bigg)\|u_{s}\|_{H_{0}^{1}}^{2}\\
	& \leq
	\bigg(\frac{\|q\|_{H^{-1}}}{a_{\min}}\bigg)^{2}\frac{a_{\max}}{a_{\min}}\sum_{m=2}^{\infty}(m-1)\sum_{\boldsymbol{\eta}\in\{s+1:\infty\}^{m}}
	c_{\boldsymbol{\nu}(\boldsymbol{\eta})}
	\bigg(\prod_{j=1}^{m}b_{\eta_{j}}\bigg),
	\end{align*}
	where the sum of $\bseta$ simplifies to
	\begin{equation*}
	\sum_{\substack{|\bsnu|=m\\
			\nu_{i}=0,~i\leq s
		}
	}\binom{m}{\bsnu}
	c_{\bsnu}
	\bigg(\prod_{i=s+1}^{\infty}b_{i}^{\nu_{i}}\bigg)
	\leq
	\sum_{\substack{|\bsnu|=m\\
			\nu_{i}=0,~i\leq s\\
			\nu_{i}\neq1,~i>s
	}}\binom{m}{\bsnu}\bigg(\prod_{i=s+1}^{\infty}b_{i}^{\nu_{i}}\bigg).
	\end{equation*}
	The dimension truncation error is estimated by splitting the upper
	bound into two parts. Let $m^{\ast}\geq3$ be an as yet undetermined index. Then
	\begin{align}
	\int_{U}\int_{D}|&\nabla(u-u_{s})|^{2}\,{\rm d}\bsx\,{\rm d}\bsy\nonumber\\
	\leq
	&
	\bigg(\frac{\|q\|_{H^{-1}}}{a_{\min}}\bigg)^{2}\frac{a_{\max}}{a_{\min}}\sum_{m=2}^{m^{\ast}-1}(m-1)\sum_{\substack{|\bsnu|=m\\
			\nu_{i}=0,~i<s\\
			\nu_{i}\neq1,~i>s
		}
	}\binom{m}{\bsnu}\bigg(\sum_{i=s+1}^{\infty}b_{i}^{\nu_{i}}\bigg)\nonumber \\
	& +
	\bigg(\frac{\|q\|_{H^{-1}}}{a_{\min}}\bigg)^{2}\frac{a_{\max}}{a_{\min}}\sum_{m=m^{\ast}}^{\infty}(m-1)\bigg(\sum_{i=s+1}^{\infty}b_{i}\bigg)^{m}.\label{eq:secondline}
	\end{align}
	We can estimate the {sum in the first term of~\eqref{eq:secondline}} by
	\begin{align*}
	\sum_{m=2}^{m^{\ast}-1}(m-1)\sum_{\substack{|\bsnu|=m\\
			\nu_{i}=0,~i<s\\
			\nu_{i}\neq1,~i>s
		}
	}\binom{m}{\bsnu}\bigg(\sum_{i=s+1}^{\infty}b_{i}^{\nu_{i}}\bigg)\leq(m^{\ast}-2)(m^{\ast}-1)!\sum_{\substack{0\neq|\bsnu|_{\infty}\leq m^{\ast}-1\\
			\nu_{i}=0,~i<s\\
			\nu_{i}\neq1,~i\geq1
		}
	}\bsb^{\bsnu},
	\end{align*}
	where $\bsb^{\bsnu}:=\prod_{i\in\mathrm{supp}(\bsnu)}b_{i}^{\nu_{i}}$.
	Furthermore, we obtain
	\begin{align}
	&\sum_{\substack{0\neq|\bsnu|_{\infty}\leq m^{\ast}-1\\
			\nu_{i}=0,~i<s\\
			\nu_{i}\neq1,~i\geq1
		}
	}\bsb^{\bsnu}  =\prod_{j=s+1}^{\infty}\bigg(1+\sum_{\ell=2}^{m^{\ast}-1}b_{j}^{\ell}\bigg)-1=\prod_{j=s+1}^{\infty}\bigg(1+b_{j}^{2}\frac{1-b_{j}^{m^{\ast}-2}}{1-b_{j}}\bigg)-1\nonumber \\
	& \leq\prod_{j=s+1}^{\infty}(1+2b_{j}^{2})-1\leq\exp\bigg(2\sum_{j=s+1}^{\infty}b_{j}^{2}\bigg)-1\leq2({\rm e}-1)\sum_{j=s+1}^{\infty}b_{j}^{2},\label{eq:term1}
	\end{align}
	where we used \eqref{eq:bbound}, \eqref{eq:opbound}, and the inequality
	${\rm e}^{x}\leq1+({\rm e}-1)x$ for all $x\in[0,1]$.

	Recalling~\eqref{eq:btail}, we find the following upper bound for
	the sum in the second term of~\eqref{eq:secondline}:
	\begin{align}
	\sum_{m=m^{\ast}}^{\infty}\!(m-1)
	\bigg(\sum_{j=s+1}^{\infty}b_{j}\bigg)^{m}\!\leq\!\sum_{m=m^{\ast}}^{\infty}
	\bigg(2\sum_{j=s+1}^{\infty}b_{j}\bigg)^{m}\!\leq\frac{(2\sum_{j=s+1}^{\infty}b_{j})^{m^{\ast}}}{1-(2\sum_{j=s+1}^{\infty}b_{j})},\label{eq:term2}
	\end{align}
	where we used the estimates $m-1\leq2^{m}$ and $2\sum_{j=s+1}^{\infty}b_{j}<1$.

	Observing that by~\cite[Theorem~5.1]{Kuo.F_Schwab_Sloan_2012_SINUM}
	it holds
	\[
	\sum_{j=s+1}^{\infty}b_{j}\leq\bigg(\sum_{j=1}^{\infty}b_{j}^{p}\bigg)^{1/p}s^{-\frac{1}{p}+1}
	\]
	and (with $b_j$ replaced by $b_j^2$ and $p$ replaced by
	$p/2$)
	\[
	\sum_{j=s+1}^{\infty}b_{j}^{2}\leq\bigg(\sum_{j=1}^{\infty}b_{j}^{p}\bigg)^{2/p}s^{-\frac{2}{p}+1},
	\]
	so
	we see that the terms~\eqref{eq:term1}--\eqref{eq:term2} can be
	balanced by choosing $m^{\ast}=\lceil\frac{2-p}{1-p}\rceil$. One
	arrives at the dimension truncation bound
	\begin{align*}
	&\sqrt{\int_{U}\int_{D}(u-u_{s})^{2}\,{\rm d}\bsx\,{\rm d}\bsy}\,\\
	&\ \ \leq
	{c_{D}\sqrt{\int_{U}\int_{D}|\nabla(u-u_{s})|^{2}\,{\rm d}\bsx\,{\rm d}\bsy}\leq}
	\,C\,\|q\|_{H^{-1}(D)}\,s^{-\frac{1}{p}+\frac{1}{2}}\quad\text{for all}~s\ge s',
	\end{align*}
	where the constant $C>0$ is independent of $s$ and $q$. This proves
	the theorem for $s\ge s'$. The result can be extended to all $s\ge1$
	by noting that
	\[
	\sqrt{\int_{U}\int_{D}|\nabla(u-u_{s})|^{2}\,{\rm d}\bsx\,{\rm d}\bsy}
	\,\leq\,\frac{2\,
		\,\|q\|_{H^{-1}(D)}}{a_{\min}}\,\leq\,
	\frac{2\,
		\,\|q\|_{H^{-1}(D)}}{a_{\min}\cdot(s')^{-1/p+1/2}}\,s^{-\frac{1}{p}+\frac{1}{2}}
	\]
	for all $1\leq s<s'$, where we used the \emph{a priori} bound identical to~\eqref{eq:apriori}.
\end{proof}
\begin{remark}
	Theorem~\ref{thm:dimensiontruncation} can be generalised further to include a more complex model
	\begin{equation*}
	\tilde{a}(\bsx,\bsy)\,=\,a_{0}(\bsx)+\sum_{j\geq1}\xi_j(y_{j})\,\psi_{j}(\bsx),\qquad\bsx\in D,~\bsy\in U,
	\end{equation*}
	where the function $\xi$ in \eqref{eq:a-tilde} is now replaced by an
	$L_\infty([0,1])$ function $\xi_j$ depending on $j$. Then, assuming that
	we have  $\int_{0}^1\xi_j(y)\,\rd y=0$, $j\geq1$, and that
	$\tilde{b}_{j}:=\|\xi_j\|_{\infty}\|\psi_{j}\|_{L_{\infty}(D)}/a_{\min}$
	is non-increasing in $j$, and moreover that $\tilde{a}$ satisfies
	$\textup{(A2)}$, the same argument as above establishes the same estimate
	as in Theorem~\ref{thm:dimensiontruncation}.
\end{remark}
\subsection{Finite element error}\label{sec:femerror}

Let assumption (A6) be in effect. Let $\{V_h\}_h$ be a family of
conforming finite element subspaces $V_h\subset H_0^1(D)$, parameterized
by the one-dimensional mesh size $h>0$, which are spanned by continuous,
piecewise linear finite element basis functions. It is assumed that the
triangulation corresponding to each $V_h$ is obtained from an initial,
regular triangulation of $D$ by recursive, uniform partition of simplices.

For each $\bsy\in U$, we denote by $u_{h}(\cdot,\bsy)\in V_h$ the finite element solution to the system
\begin{align}
\int_D a(\bsx,\bsy)\,\nabla u_{h}(\bsx,\bsy)\cdot \nabla v_h(\bsx)\,{\rm d}\bsx
=\langle q,v_h\rangle_{H^{-1}(D),H^1_0(D)},\quad\forall~v_h\in V_h,\label{eq:fem}
\end{align}
where $q\in H^{-1}(D)$ and $a$ is defined by~\eqref{eq:a-param}. Under
assumptions (A1)--(A2), this system is uniquely solvable and the finite
element solution $u_h$ satisfies both the {\em a priori}
bound~\eqref{eq:apriori} as well as the partial derivative
bounds~\eqref{eq:pde-der}. In analogy to the previous subsection, we also
define  the dimensionally truncated finite element solution by setting
\begin{equation}
u_{s,h}(\cdot,\bsy):=u_{s,h}(\cdot,(y_1,\ldots,y_s)):=u_h(\cdot,(y_1,\ldots,y_s,0,0,\ldots)),\quad \bsy\in U,\label{eq:dimtruncfem}
\end{equation}
where $u_h(\cdot,\bsy)\in V_h$ is the solution of~\eqref{eq:fem} for $\bsy\in U$.

\begin{theorem}
	Under the assumptions \textup{(A1)}, \textup{(A2)}, \textup{(A4)} and \textup{(A6)},
	for every $\bsy\in U$ and $q\in H^{-1+t}(D)$ with $t\in[0,1]$, there holds
	the asymptotic convergence estimate
	$$
	\|u(\cdot,\bsy)-u_h(\cdot,\bsy)\|_{L_2(D)} \,\leq\, C\, h^{1+t}\,\|q\|_{H^{-1+t}(D)}\qquad\text{as}~h\to 0,
	$$
	where the constant $C>0$ is independent of $h$ and $\bsy$.
\end{theorem}
\begin{proof} Let $\bsy\in U$. From
	\cite[Theorem~7.2]{Kuo.F_Schwab_Sloan_2012_SINUM},  under the assumptions
	(A1), (A2), (A4), and (A6), we have for every $g\in L_2(D)$ the following
	asymptotic convergence estimate as $h\to 0$
	\begin{align}
	|\langle g,u(\cdot,\bsy)-u_h(\cdot,\bsy)\rangle_{L_2(D)}|
	\leq C\,h^{1+t}\,\|q\|_{H^{-1+t}(D)}\,\|g\|_{L^2(D)},\label{eq:preaubinnitsche}
	\end{align}
	where the constant $C>0$ is independent of $h$ and $\bsy$. Therefore
	\begin{align*}
	\|u(\cdot,\bsy)-u_h(\cdot,\bsy)\|_{L_2(D)}
	&= \sup_{g\in L_2(D),\,\|g\|_{L_2(D)}=1} \langle g,u(\cdot,\bsy)-u_h(\cdot,\bsy)\rangle_{L_2(D)} \\
	& \le C\,h^{1+t}\,\|q\|_{H^{-1+t}(D)},
	\end{align*}
	and this concludes the proof.
	\end{proof}

\subsection{Kernel interpolation error}\label{sec:kernelerror}

We focus on approximating the finite element solution of the
problem~\eqref{eq:pdeweak} in the following discussion, since it is
essential for our approximation scheme that the function being
approximated is pointwise well-defined in the physical domain $D$.

Let $H(U_s)=H$ denote the {RKHS} of functions with respect to the stochastic
parameter $\bsy\in U_s$, defined in Section~\ref{sec:space}. For every $\bsx\in D$, let
\[
u_{s,h,n}(\bsx,\cdot) \,:=\, A^\ast_n(u_{s,h}(\bsx,\cdot)) \,\in\, H(U_s)
\]
be the kernel interpolant of the dimensionally truncated finite element
solution~\eqref{eq:dimtruncfem} at $\bsx$ as a function of~$\bsy$. We
measure the $L_2$~approximation error
$\|u_{s,h}(\bsx,\cdot)-u_{s,h,n}(\bsx,\cdot)\|_{L_2(U_s)}$ in $\bsy$ and
then take the $L_2$ norm over $\bsx$, to arrive at the error criterion
\begin{align*}
&\sqrt{\int_D \|u_{s,h}(\bsx,\cdot) - u_{s,h,n}(\bsx,\cdot)\|_{L_2(U_s)}^2 \,\rd\bsx}\\
&\qquad\qquad
=\; \sqrt{\int_{U_s}\int_D  \big(u_{s,h}(\bsx,\bsy) - u_{s,h,n}(\bsx,\bsy)\big)^2 \,\rd\bsx\,\rd\bsy},
\end{align*}
where, observing that $u_{s,h} - u_{s,h,n}$ is jointly measurable, we interchanged the order of integration by appeal to the Fubini's theorem.

\begin{theorem} \label{thm:PDE-error}
	Under the assumptions \textup{(A1)}, \textup{ (A2)} and \textup{ (A6),} let
	$u_{s,h}(\cdot,\bsy)\in H_0^1(D)$ denote the dimensionally truncated
	finite element solution of~\eqref{eq:fem} for $\bsy\in U_s$ and let $q\in
	H^{-1}(D)$ be the corresponding source term. Moreover, for every $\bsx\in
	D$ let $u_{s,h,n}(\bsx,\cdot) := A^\ast_n(u_{s,h}(\bsx,\cdot))$ be the
	kernel interpolant at $\bsx$ based on a lattice rule satisfying the
	assumptions of Theorem~\ref{thm:wce}. Suppose that $\alpha\in 2
	{\mathbb{N}}$ and $\sigma := \frac{\alpha}{2}$. Then we have for all
	$\lambda\in(\frac{1}{\alpha},1]$ that
	\begin{align*}
	\sqrt{\int_{U_s}\int_D  \big(u_{s,h}(\bsx,\bsy) - u_{s,h,n}(\bsx,\bsy)\big)^2 \,\rd\bsx\,\rd\bsy}
	\,\le\,
	\frac{\kappa}{n^{1/(4\lambda)}}
	\frac{c_D\,\|q\|_{H^{-1}}}{a_{\min}}\, C_s(\lambda),
	\end{align*}
	where $c_D>0$ is the Poincar\'{e} constant of the embedding
	$H_0^1(D)\hookrightarrow L_2(D)$, $\kappa>0$ is the constant defined in
	Theorem~\ref{thm:wce}, and
	\begin{align} \label{eq:C-lambda}
	\begin{split}
	[C_s(\lambda)]^{2\lambda}:=
	&\Bigg(\sum_{\setu\subseteq\{1:s\}} \max(|\setu|,1) \gamma_\setu^\lambda\, [2\zeta(\alpha\lambda)]^{|\setu|}
	\Bigg)\\
	&\times
	\Bigg(\sum_{\setu\subseteq\{1:s\}} \frac{1}{\gamma_\setu}
	\bigg(\sum_{\bsm_\setu\in\{1:\sigma\}^{|\setu|}} |\bsm_\setu|!\prod_{j\in \setu} \big(b_j^{m_j}S(\sigma,m_j)\big)\bigg)^2
	\Bigg)^\lambda.
	\end{split}
	\end{align}
\end{theorem}

\begin{proof}
	We can express the squared $L_2$ error as
	\begin{align*}
	\int_{U_s}\int_D \big(u_{s,h}(\bsx,\bsy)& - u_{s,h,n}(\bsx,\bsy)\big)^2 \,\rd\bsx\,\rd\bsy\\
	&\,=\, \int_D \|u_{s,h}(\bsx,\cdot) - u_{s,h,n}(\bsx,\cdot)\|_{L_2(U_s)}^2 \,\rd\bsx \\
	&\,\le\, \int_D \Big( e^{\rm wor}(A^*_n;L_2(U_s)) \, \|u_{s,h}(\bsx,\cdot)\|_{H(U_s)} \Big)^2 \,\rd\bsx \\
	&\,=\, [ e^{\rm wor}(A^*_n;L_2(U_s)) ]^2\, \int_D \|u_{s,h}(\bsx,\cdot)\|_{H(U_s)}^2 \,\rd\bsx.
	\end{align*}
	The first factor is the squared worst case $L_2$ approximation error, which
	can be bounded using Theorem~\ref{thm:wce}. The second factor can be
	estimated using \eqref{eq:Hnorm} by
	\begin{align*}
	&\int_D \|u_{s,h}(\bsx,\cdot)\|_{H(U_s)}^2 \,\rd\bsx \\
	&=\int_D\! \sum_{\setu\subseteq\{1:s\}}\!
	\frac{1}{(2\pi)^{\alpha|\setu|}\,\gamma_\setu}
	\int_{[0,1]^{|\setu|}}
	\biggl(\int_{[0,1]^{s-|\setu|}}\! \biggl(\prod_{j\in\setu}
	\frac{\partial^{\sigma}}{\partial y_j^{\sigma}}\biggr)
	u_{s,h}(\bsx,\bsy)
	\rd\bsy_{-\setu}\biggr)^2\,\rd\bsy_\setu\rd\bsx \\
	&\le\int_D\! \sum_{\setu\subseteq\{1:s\}}\!
	\frac{1}{(2\pi)^{\alpha|\setu|}\,\gamma_\setu}\!\!\;
	\int_{[0,1]^{|\setu|}}
	\int_{[0,1]^{s-|\setu|}}
	\biggl[\biggl(\prod_{j\in\setu} \frac{\partial^{\sigma}}{\partial y_j^{\sigma}}\biggr) u_{s,h}(\bsx,\bsy)\biggr]^2
	\,\rd\bsy_{-\setu}\,\rd\bsy_\setu\,\rd\bsx \\
	&\,=\, \sum_{\setu\subseteq\{1:s\}}\frac{1}{(2\pi)^{\alpha|\setu|}\,\gamma_\setu}
	\int_{[0,1]^s}
	\bigg\|\bigg(\prod_{j\in\setu} \frac{\partial^{\sigma}}{\partial y_j^{\sigma}}\bigg) u_{s,h}(\cdot,\bsy)\bigg\|^2_{L_2(D)}
	\,\rd\bsy \\
	&\,\le\, c_D^2\,\sum_{\setu\subseteq\{1:s\}}\frac{1}{(2\pi)^{\alpha|\setu|}\,\gamma_\setu}
	\int_{[0,1]^s}
	\bigg\|\bigg(\prod_{j\in\setu} \frac{\partial^{\sigma}}{\partial y_j^{\sigma}}\bigg) u_{s,h}(\cdot,\bsy)\bigg\|^2_{H^1_0(D)}
	\,\rd\bsy \\
	&\,\le\, c_D^2\,\frac{\|q\|_{H^{-1}(D)}^2}{a_{\min}^2}
	\sum_{\setu\subseteq\{1:s\}}
	\frac{1}{\gamma_\setu}
	\bigg(\sum_{\bsm_\setu\in\{1:\sigma\}^{|\setu|}} |\bsm_\setu|!\prod_{j\in \setu} \big(b_j^{m_j}S(\sigma,m_j)\big)\bigg)^2,
	\end{align*}
	where we used the Cauchy--Schwarz inequality, Fubini's theorem, the
	Poincar\'{e} constant $c_D>0$ for the embedding $H_0^1(D)\hookrightarrow
	L_2(D)$, together with the PDE derivative bound \eqref{eq:pde-der} applied
	with $\bsnu = (\sigma,\ldots,\sigma) = (\frac{\alpha}{2},
	\ldots,\frac{\alpha}{2})$. The theorem is proved by combining the above
	expressions with Theorem~\ref{thm:wce}.
	\end{proof}

Next, we proceed to choose the weights $\gamma_\setu$ and the parameters
$\lambda$ and $\alpha$ to ensure that the constant $C_s(\lambda)$ can be
bounded independently of $s$, with $\lambda$ as small as possible to yield
the best possible convergence rate.

\subsubsection{Choosing SPOD weights}

One way to choose the weights is to equate the terms inside the two sums
over $\setu$ in the formula \eqref{eq:C-lambda} for $C_s(\lambda)$.
(The value of $C_s(\lambda)$ so obtained minimizes \eqref{eq:C-lambda}
with respect to $\gamma_\setu$ for $\setu\subseteq \{1:s\}$.) It will be
shown that this yields the convergence rate
$\calO(n^{-(\frac{1}{2p}-\frac{1}{4})})$ with an implied constant
independent of the dimension $s$. The rate is precisely the rate of
convergence that we expect to get. However, this choice of weights is too
complicated to allow for efficient CBC construction of the lattice
generating vector. So in the theorem below we propose a choice of SPOD
weights that achieves the same error bound.

\begin{theorem}\label{thm:spod}
	Assume that \textnormal{(A1)--(A3)} and \textnormal{(A6)} hold, and that
	$p$ is as in \textnormal{(A3)}. Take $\alpha := 2\lfloor
	\frac{1}{p}+\frac{1}{2}\rfloor$, $\sigma := \frac{\alpha}{{2}}$,
	$\lambda:= \frac{p}{2-p}$, and define the weights to be
	\begin{equation}
	\gamma_\setu\! :=\! \Biggl(\!
	\frac{1}{\max(|\setu|,1)\,[2\zeta(\alpha\lambda)]^{|\setu|}}
	\biggl(\sum_{\bsm_\setu\in\{1:\sigma\}^{|\setu|}}\!\!\!\! |\bsm_\setu|!\prod_{j\in \setu} \bigl(b_j^{m_j}S(\sigma,m_j)\bigr)
	\biggr)^{\!\!2}\Biggr)^{\frac{1}{1+\lambda}}\!\!
	\label{eq:bad-weights}
	\end{equation}
	for $\emptyset\neq\setu\subset\bbN,\;
	|\setu|<\infty$,
	or SPOD weights
	\begin{equation} \label{eq:spod-weights}
	\gamma_\setu \,:=\,
	\sum_{\bsm_\setu\in\{1:\sigma\}^{|\setu|}} (|\bsm_\setu|!)^{\frac{2}{1+\lambda}}
	\prod_{j\in \setu} \bigg(\frac{b_j^{m_j}S(\sigma,m_j)}{\sqrt{2\mathrm{e}^{1/\mathrm{e}}\zeta(\alpha\lambda)}}
	\bigg)^{\frac{2}{1+\lambda}}
	\end{equation}
	for $\emptyset\neq\setu\subset\bbN,\;
	|\setu|<\infty$, with $\gamma_{\emptyset}:= 1$. Then the kernel interpolant of the finite
	element solution in Theorem~\ref{thm:PDE-error} satisfies
	\begin{align*}
	\sqrt{\int_{U_s}\int_D \big(u_{s,h}(\bsx,\bsy) - u_{s,h,n}(\bsx,\bsy)\big)^2 \,\rd\bsx\,\rd\bsy}
	\,\le\,
	C\,\|q\|_{H^{-1}(D)}\,n^{-(\frac{1}{2p} - \frac{1}{4})},
	\end{align*}
	where the constant $C>0$ is independent of the dimension $s$.
\end{theorem}

\begin{proof}
	We will proceed to justify the two choices of weights
	\eqref{eq:bad-weights} and \eqref{eq:spod-weights}, and show that in both
	cases the term $C_s(\lambda)$ appearing in Theorem~\ref{thm:PDE-error} can
	be bounded independently of~$s$, by specifying $\lambda$ and $\alpha$ as
	in the theorem.

	The first choice of weights \eqref{eq:bad-weights} is obtained by equating
	the terms inside the two sums over~$\setu$ in the formula
	\eqref{eq:C-lambda}. Substituting \eqref{eq:bad-weights} into
	\eqref{eq:C-lambda} yields
	\begin{align}
	&[C_s(\lambda)]^{\frac{2\lambda}{1+\lambda}}\nonumber\\
	&\!=\!\!
	\sum_{\setu\subseteq\{1:s\}}\!\!\!\;
	\biggl(\max(|\setu|,1)\, [2\zeta(\alpha\lambda)]^{|\setu|}\!\!\; \biggr)^{\frac{1}{1+\lambda}}\!
	\biggl(\sum_{\bsm_\setu\in\{1:\sigma\}^{|\setu|}}\! \!
	|\bsm_\setu|!
	\prod_{j\in \setu} \big(b_j^{m_j}S(\sigma,m_j)\big)\biggr)^{\frac{2\lambda}{1+\lambda}} \nonumber\\
	&\!\le \sum_{\setu\subseteq\{1:s\}}
	\bigg( \sum_{\bsm_\setu\in\{1:\sigma\}^{|\setu|}} |\bsm_\setu|!
	\prod_{j\in \setu} \Big(b_j^{m_j}S(\sigma,m_j)
	\big[2\mathrm{e}^{1/\mathrm{e}}\zeta(\alpha\lambda)\big]^{\frac{1}{2\lambda}}\Big)\bigg)^{\frac{2\lambda}{1+\lambda}} \nonumber\\
	&\!\le \sum_{\setu\subseteq\{1:s\}}
	\sum_{\bsm_\setu\in\{1:\sigma\}^{|\setu|}}
	\biggl(|\bsm_\setu|!
	\prod_{j\in \setu} \beta_j^{m_j}\biggr)^{\!\frac{2\lambda}{1+\lambda}}\!
	=
	\! \sum_{\bsm\in\{0:\sigma\}^{{s}}}
	\biggl(|\bsm|!
	\prod_{j=1}^s \beta_j^{m_j}\biggr)^{\frac{2\lambda}{1+\lambda}}, \label{eq:bound}
	\end{align}
	where we used $\max(|\setu|,1)\le [\mathrm{e}^{1/\mathrm{e}}]^{|\setu|} =
	(1.4446\cdots)^{|\setu|}$, and defined
	$S_{\max}(\sigma) := \max_{1\le m\le\sigma}
	S(\sigma,m)$, and $
	\beta_j \,:=\,
S_{\max}(\sigma)\,[2{\rm e}^{1/{\rm e}}\zeta(\alpha\lambda)]^{\frac{1}{2\lambda}}
	b_j$ 
	for all ${j\geq 1}$,
	while applying Jensen's inequality\footnote{Jensen's inequality
		states that $\sum_i a_i \le (\sum_i a_i^t)^{1/t}$ for all $a_i\ge 0$ and
		$t\in (0,1]$.}
	with $0<\frac{2\lambda}{1+\lambda}\le 1$. 	

	The second choice of weights \eqref{eq:spod-weights} is inspired by the
	weights \eqref{eq:bad-weights} but takes the SPOD form
	\begin{equation} \label{eq:guess}
	\gamma_\setu := \frac{1}{\tau^{|\setu|}}\!\!
	\sum_{\bsm_\setu\in\{1:\sigma\}^{|\setu|}} \!\!\! [V(\bsm_\setu)]^{\frac{2}{1+\lambda}},
	\quad\
	V(\bsm_\setu) := |\bsm_\setu|! \prod_{j\in \setu} \big(b_j^{m_j}S(\sigma,m_j)\big),
	\end{equation}
	with $\tau>0$ to be specified below. (The $\tau^{|\setu|}$ factor can be
	merged into the product over~$\setu$, thus giving SPOD weights.)
	Estimating $\max(|\setu|,1)\le [\mathrm{e}^{1/\mathrm{e}}]^{|\setu|}$ in
	\eqref{eq:C-lambda}, plugging in the weights \eqref{eq:guess}, applying
	the Cauchy--Schwarz inequality with $\frac{1}{1+\lambda} +
	\frac{\lambda}{1+\lambda}=1$, and applying Jensen's inequality
	with
	$0<\lambda\le 1$, we obtain from \eqref{eq:C-lambda}
	\begin{align*}
	&[C_s(\lambda)]^{2\lambda}\\
	&
	\le \Biggl(\sum_{\setu\subseteq\{1:s\}} \gamma_\setu^\lambda\, [2\mathrm{e}^{1/\mathrm{e}}\zeta(\alpha\lambda)]^{|\setu|}
	\Biggr)
	\Biggl(\sum_{\setu\subseteq\{1:s\}} \frac{1}{\gamma_\setu}
	\biggl(\sum_{\bsm_\setu\in\{1:\sigma\}^{|\setu|}}\!\!\! [V(\bsm_\setu)]^{\frac{1}{1+\lambda} + \frac{\lambda}{1+\lambda}} \biggr)^2
	\Biggr)^\lambda \notag\\
	&\le
	\Biggl(\sum_{\setu\subseteq\{1:s\}} \biggl(\frac{2\mathrm{e}^{1/\mathrm{e}}\zeta(\alpha\lambda)}{\tau^\lambda}\biggr)^{|\setu|}
	\biggl(\sum_{\bsm_\setu\in\{1:\sigma\}^{|\setu|}} [V(\bsm_\setu)]^{\frac{2}{1+\lambda}}\biggr)^\lambda
	\Biggr) \notag\\
	&\qquad
	\times
	\Biggl(\sum_{\setu\subseteq\{1:s\}} \frac{\tau^{|\setu|}}{\sum_{\bsm_\setu\in\{1:\sigma\}^{|\setu|}}
		[V(\bsm_\setu)]^{\frac{2}{1+\lambda}}}
	\biggl(\sum_{\bsm_\setu\in\{1:\sigma\}^{|\setu|}} [V(\bsm_\setu)]^{\frac{2}{1+\lambda}} \biggr)\\
	&\qquad\qquad\qquad\qquad\qquad\qquad\qquad\qquad\qquad
	\times\biggl(\sum_{\bsm_\setu\in\{1:\sigma\}^{|\setu|}} [V(\bsm_\setu)]^{\frac{2\lambda}{1+\lambda}} \biggr)
	\Biggr)^\lambda,
	\end{align*}
	and further
	\begin{align*}
	[C_s(\lambda)]^{2\lambda}&\le
	\Bigg(\sum_{\setu\subseteq\{1:s\}} \bigg(\frac{2\mathrm{e}^{1/\mathrm{e}}\zeta(\alpha\lambda)}{\tau^\lambda}\bigg)^{|\setu|}\!\!
	\sum_{\bsm_\setu\in\{1:\sigma\}^{|\setu|}}\!\!\! [V(\bsm_\setu)]^{\frac{2\lambda}{1+\lambda}}
	\Bigg)\\
	&\qquad\qquad\times
	\Bigg(\sum_{\setu\subseteq\{1:s\}} \tau^{|\setu|}\!\!
	\sum_{\bsm_\setu\in\{1:\sigma\}^{|\setu|}}\!\!\! [V(\bsm_\setu)]^{\frac{2\lambda}{1+\lambda}}
	\Bigg)^\lambda \notag\\
	&= \Bigg(\sum_{\setu\subseteq\{1:s\}} \tau^{|\setu|}
	\sum_{\bsm_\setu\in\{1:\sigma\}^{|\setu|}}\!\!\! [V(\bsm_\setu)]^{\frac{2\lambda}{1+\lambda}}
	\Bigg)^{1+\lambda},
	\end{align*}
	where equality holds provided that we now choose $\tau :=
	[2\mathrm{e}^{1/\mathrm{e}}\zeta(\alpha\lambda)]^{\frac{1}{1+\lambda}}$.
	This leads to the same upper bound~\eqref{eq:bound} as for the first
	choice of weights.

	It remains to show that the upper bound \eqref{eq:bound} can be bounded
	independently of $s$.
	We define the sequence $d_j :=\beta_{\lceil
		j/\sigma\rceil}$ for $j\geq 1$, so that $d_1=\cdots=d_{\sigma}=\beta_1$,
	$d_{\sigma+1}=\cdots=d_{2\sigma}=\beta_2$, and so on.
	Then for $\boldsymbol m\in\{0:\sigma\}^s$ we can write
	\[
	\prod_{j=1}^s\beta_j^{m_j}=\underset{m_1~\text{factors}}{\underbrace{\beta_1\cdots \beta_1}}\cdots \underset{m_s~\text{factors}}{\underbrace{\beta_s\cdots\beta_s}}=\prod_{j\in\setv}d_j,
	\]
	where $\setv:=\{1,2,\ldots,m_1,\sigma+1,\sigma+2,\ldots,\sigma+m_2,\ldots,(s-1)\sigma+1,\ldots,(s-1)\sigma+m_s\}$.
	Clearly, the set $\setv$ is of cardinality $|\setv|=m_1+\cdots+m_s=|\boldsymbol m|$.
	It follows that
	\begin{align} \label{eq:sum}
	\sum_{\bsm\in\{0:\sigma\}^{s}}
	&\biggl(|\bsm|! \prod_{j=1}^s \beta_j^{m_j}\biggr)^{\frac{2\lambda}{1+\lambda}}
	\le  \sum_{\substack{\setv\subset\mathbb{N}\\ |\setv|<\infty}}  \bigg(|\setv|!
	\prod_{j\in\setv}d_j\bigg)^{\frac{2\lambda}{1+\lambda}}\notag\\
	&=\, \sum_{\ell\geq 0} (\ell !)^{\frac{2\lambda}{1+\lambda}}
	\sum_{\substack{\setv\subset\mathbb{N}\\ |\setv|=\ell}}\prod_{j\in\setv}d_j^{\frac{2\lambda}{1+\lambda}}
	\,\le\, \sum_{\ell\geq 0} (\ell !)^{\frac{2\lambda}{1+\lambda}} \frac{1}{\ell !}
	\bigg(\underbrace{\sum_{j\geq 1}d_j^{\frac{2\lambda}{1+\lambda}}}_{=:\,T}\bigg)^{\ell}.
	\end{align}
	The final inequality holds because $(\sum_{j\geq
		1}d_j^{\frac{2\lambda}{1+\lambda}} )^\ell$ includes all the products of
	the form $\prod_{j\in\mathfrak{v}}d_j^{\frac{2\lambda}{1+\lambda}} $ with
	$|\mathfrak{v}|=\ell$, and moreover includes each such term $\ell !$
	times.

	Recall from \eqref{eq:def-bj} and the assumption (A3) that $\sum_{j\ge 1}
	b_j^p < \infty$. We now choose
	\[
	\frac{2\lambda}{1+\lambda} =p \quad\iff\quad \frac{1}{2\lambda} = \frac{1}{p} - \frac{1}{2}
	\quad\iff\quad \lambda = \frac{p}{2-p}.
	\]
	For the inner sum in \eqref{eq:sum} we now have
	\begin{align*}
	T \,=\, \sum_{j\geq 1} d_j^p
	\,=\, \sigma \sum_{j\geq 1} \beta_j^p
	\,=\, \sigma \max\Big(1,S_{\max}(\sigma)\,[2\zeta(\alpha\lambda)]^{\frac{1}{2\lambda}}\Big)^p \sum_{j\geq 1} b_j^p
	\,<\, \infty,
	\end{align*}
	provided that $\alpha\lambda > 1$, which is equivalent to $\alpha >
	\frac{2}{p}-1$. This latter condition as well as the requirement that
	$\alpha$ be even can be satisfied by taking $\alpha$ such that
	$\frac{\alpha}{2} = \lfloor (\frac{1}{p} - \frac{1}{2}) +1\rfloor$, so we
	take $\alpha := {2\lfloor \frac{1}{p}+\frac{1}{2}\rfloor}$. Finally, the
	ratio test implies convergence of the outer sum in \eqref{eq:sum}, and
	consequently $C_s(\lambda)$ is bounded independently of $s$. Theorem
	\ref{thm:PDE-error} now ensures an error bound independent of $s$, and the
	convergence rate is $\calO(n^{-(\frac{1}{2p}-\frac{1}{4})})$. This
	completes the proof.
	\end{proof}

\subsubsection{Choosing POD weights}

In the next theorem we prove that if the assumption (A3) holds for some
$p\in \bigcup_{k=1}^\infty \big(\frac{2}{2k+1},\frac{1}{k}\big)$, then it
is possible to use POD weights to obtain the same rate of convergence as
in Theorem~\ref{thm:spod}. For this and the next subsections we need the
sequence of \emph{Bell polynomials} (more precisely, \emph{Touchard polynomials}), which we denote by
\[
{\rm Bell}_{\sigma}(x) \,:=\, \sum_{m=0}^{\sigma} S(\sigma,m)\,x^m,\quad \sigma\in\mathbb{N}_{0},
\]
where $S(\sigma,m)$ denotes the Stirling number of the second kind as
before.

\begin{theorem} \label{thm:pod}
	Assume that \textnormal{(A1)--(A3)}, \textnormal{(A5)} and
	\textnormal{(A6)} hold, and further assume that $p\in \bigcup_{k=1}^\infty
	\big(\frac{2}{2k+1},\frac{1}{k}\big)$ in \textnormal{(A3)}. We take
	$\alpha := 2\lfloor \frac{1}{p}\rfloor$, $\sigma := \frac{\alpha}{2}$,
	$\lambda:= \frac{p}{2-p}$, and define POD weights
	\begin{equation} \label{eq:weights2}
	\gamma_{\setu}:=
	\Bigg(\frac{[(\sigma|\setu|)!]^2}{\max(|\setu|,1)\,[2\zeta(\alpha\lambda)]^{|\setu|}}
	\prod_{j\in \setu} {\rm Bell}_{\sigma}(b_j)^2
	\Bigg)^{\frac{1}{1+\lambda}}\quad\text{for }\emptyset\neq\setu\subseteq\{1:s\},
	\end{equation}
	with $\gamma_{\emptyset}:= 1$. Then the kernel interpolant of the PDE solution in
	Theorem~\ref{thm:PDE-error} satisfies
	\begin{align*}
	\sqrt{\int_{U_s}\int_D \big(u_{s,h}(\bsx,\bsy) - u_{s,h,n}(\bsx,\bsy)\big)^2 \,\rd\bsx\,\rd\bsy}
	\,\le\,
	C\,\|q\|_{H^{-1}(D)}\,n^{-(\frac{1}{2p} - \frac{1}{4})},
	\end{align*}
	where the constant $C>0$ is independent of the truncation dimension
	$s$.
\end{theorem}

\begin{proof}
	In \eqref{eq:C-lambda} we can apply the crude upper bound
	\begin{align*}
	\sum_{\bsm_{\setu}\in \{1:\sigma\}^{|\setu|}}|\bsm_{\setu}|!
	\prod_{j\in\setu} \big(b_j^{m_j}S(\sigma,m_j) \big)
	&\,\le\, (\sigma|\setu|)!\,\prod_{j\in\setu} \bigg(\sum_{m=1}^\sigma b_j^{m} S(\sigma,m)\bigg)\\
	&
	\,=\, (\sigma|\setu|)! \prod_{j\in\setu} {\rm Bell}_{\sigma}(b_j),
	\end{align*}
	which leads to
	\begin{align}
	&\hspace{-8pt}[C_s(\lambda)]^{2\lambda}\notag\\
	&\hspace{-9pt}
	\le \bigg(\!\sum_{\setu\subseteq\{1:s\}}
	\!\!\max(|\setu|,1)\gamma_{\setu}^{\lambda}\,[2\zeta(\alpha\lambda)]^{|\setu|}\bigg)
	\bigg(\!\sum_{\setu\subseteq\{1:s\}} \!\!\!
	\frac{\big[(\sigma|\setu|)!
		\prod_{j\in \setu} {\rm Bell}_{\sigma}(b_j)\big]^2}{\gamma_{\setu}} \bigg)^{\lambda}.\label{eq:PODC-lambda}
	\end{align}
	{\sloppy We equate the terms in the two sums in \eqref{eq:PODC-lambda} to obtain
		the weights \eqref{eq:weights2}. Let us again define
		$S_{\max}(\sigma):=\max_{1\leq m\leq \sigma}S(\sigma,m)$, so that ${\rm
			Bell}_{\sigma}(b_j) \le S_{\max}(\sigma) \sum_{m=1}^\sigma b_j^m$.
		Plugging the weights back into~\eqref{eq:PODC-lambda} then yields
		\begin{align*}
		[C_s(\lambda)]^{\frac{2\lambda}{1+\lambda}}
		&\le \sum_{\setu\subseteq\{1:s\}}
		\biggl(\max(|\setu|,1)\, [2\zeta(\alpha\lambda)]^{|\setu|}\biggr)^{\frac{1}{1+\lambda}}
		\biggl((\sigma|\setu|)!
		\prod_{j\in\setu}{\rm Bell}_{\sigma}(b_j)\biggr)^{\frac{2\lambda}{1+\lambda}}\\
		&\le \sum_{\ell=0}^\infty
		\bigg(\max(\ell,1)\, [2\zeta(\alpha\lambda)]^\ell\bigg)^{\frac{1}{1+\lambda}}
		\bigg((\sigma\ell)!\, [S_{\max}(\sigma)]^\ell\bigg)^{\frac{2\lambda}{1+\lambda}}\notag\\
		&\phantom{  \sum_{\ell=0}^\infty
			\bigg(\max(\ell,1)\, [2\zeta(\alpha\lambda)]^\ell\bigg)^{\frac{1}{1+\lambda}}
			\bigg((\sigma\ell)!\,   }
		\times\underset{=:\,V_\ell}{\underbrace{\sum_{\substack{\setu\subseteq\{1:s\}\\\,|\setu|=\ell}}
				\prod_{j\in\setu}\bigg(\sum_{m=1}^{\sigma}b_j^m\bigg)^{\frac{2\lambda}{1+\lambda}}}}.
		\end{align*}
	}
	To estimate $V_\ell$, we have
	\begin{align*}
	V_\ell
	&\le \frac{1}{\ell!}\bigg(\sum_{j=1}^\infty
	\bigg(\sum_{m=1}^{\sigma}b_j^m\bigg)^{\frac{2\lambda}{1+\lambda}}\bigg)^\ell
	=\frac{1}{\ell!}\bigg(\sum_{j=1}^\infty \bigg(\sum_{m=1}^{\sigma}(1+b_1)^m
	\bigg(\frac{b_j}{1+b_1}\bigg)^m\bigg)^{\frac{2\lambda}{1+\lambda}}\bigg)^\ell\\
	&\le \frac{1}{\ell!}\bigg((1+b_1)^{\frac{2\sigma\lambda}{1+\lambda}}
	\sum_{j=1}^\infty \bigg(\frac{b_j}{1+b_1-b_j}\bigg)^{\frac{2\lambda}{1+\lambda}}\bigg)^\ell
	\le \frac{1}{\ell!}\bigg(\underbrace{(1+b_1)^{\frac{2\sigma\lambda}{1+\lambda}}
		\sum_{j=1}^\infty b_j^{\frac{2\lambda}{1+\lambda}}}_{=:\,T}\bigg)^\ell,
	\end{align*}
	where we estimated the sum over $m$ by the geometric series formula and
	used $1+b_1-b_j\ge 1$ as a consequence of the assumption~(A5). 	

	In consequence, we have $[C_s(\lambda)]^{\frac{2\lambda}{1+\lambda}}
	\le
	\sum_{\ell=0}^\infty a_\ell$, with
	\[
	a_\ell \,:=\,
	[\max(\ell,1)]^{\frac{1}{1+\lambda}}\,
	[2\zeta(\alpha\lambda)]^{\frac{\ell}{1+\lambda}}\,
	[(\sigma\ell)!]^{\frac{2\lambda}{1+\lambda}}\,
	[S_{\max}(\sigma)]^{\frac{2\lambda \ell}{1+\lambda}}\,\frac{1}{\ell!}\, T^\ell>0.
	\]
	We can use the ratio test to determine sufficient conditions for the
	convergence of the infinite sum over $\ell$. Letting $\ell>0$, we find
	that
	\begin{align*}
	\frac{a_{\ell+1}}{a_\ell}
	&=
	\Bigl(\frac{\ell+1}{\ell}\Bigr)^{\frac{1}{1+\lambda}}\,
	[2\zeta(\alpha\lambda)]^{\frac{1}{1+\lambda}}\,
	[(\sigma\ell+\sigma)\cdots(\sigma\ell+1)]^{\frac{2\lambda}{1+\lambda}}\,
	[S_{\max}(\sigma)]^{\frac{2\lambda}{1+\lambda}}\,
	\frac{T}{\ell+1} \\
	&\le
	\Bigl(\frac{\ell+1}{\ell}\Bigr)^{\frac{1}{1+\lambda}}\,
	[2\zeta(\alpha\lambda)]^{\frac{1}{1+\lambda}}\,
	(\sigma\ell+\sigma)^{\frac{2\sigma \lambda}{1+\lambda}}\,
	[S_{\max}(\sigma)]^{\frac{2\lambda}{1+\lambda}}\,
	\frac{T}{\ell+1}
	\xrightarrow{\ell\to\infty} 0,
	\end{align*}
	provided that $\frac{2\sigma \lambda}{1+\lambda} =
	\frac{\alpha\lambda}{1+\lambda}<1$ and $\alpha\lambda>1$. In conclusion,
	by choosing $\frac{2\lambda}{1+\lambda}=p$ $\Longleftrightarrow$
	$\lambda=\frac{p}{2-p}$, it follows from Theorem~\ref{thm:PDE-error} that
	the convergence is independent of $s$ with rate
	$\mathcal{O}(n^{-(\frac{1}{2p}-\frac{1}{4})})$, provided that
	$$
	\frac{1}{\alpha}<\lambda<\frac{1}{\alpha-1}
	\quad\Longleftrightarrow\quad
	\alpha - 1 < \frac{2}{p}-1 < \alpha.
	$$
	Unfortunately this condition cannot be fulfilled for all values of $p$,
	since $\alpha=2\sigma$ needs to be an even integer. Indeed, the condition
	is equivalent to
	$$
	\frac{2}{2\sigma+1} < p < \frac{1}{\sigma}
	\quad\Longleftrightarrow\quad
	\frac{1}{p}-\frac{1}{2} < \sigma < \frac{1}{p}.
	$$
	We conclude that this condition is met if $p\in
	\bigcup_{k=1}^\infty\big(\frac{2}{2k+1},\frac{1}{k}\big)$ by choosing
	$\alpha = 2\lfloor\frac{1}{p}\rfloor$.
	\end{proof}

The Lebesgue measure of the set of admissible values for $p$ is precisely
$\mu\bigl(\bigcup_{k=1}^\infty
(\frac{2}{2k+1},\frac{1}{k})\bigr)=2-\log(4)\approx 0.61$.
Nevertheless, even if $p\not\in \bigcup_{k=1}^\infty
\bigl(\frac{2}{2k+1},\frac{1}{k}\bigr)$ we can always choose $\tilde{p}>p$
such that $\tilde{p}\in
\bigcup_{k=1}^\infty\big(\frac{2}{2k+1},\frac{1}{k}\big)$ and a
correspondingly larger value of $\lambda$. The theorem then holds but with
some loss in the rate of convergence.

\subsubsection{Choosing product weights}

In the next theorem we increase our error bounds to obtain product
weights, which have the benefit of a lower computational cost (see
Section~\ref{sec:cost}), but with the disadvantage of a compromised
theoretical convergence rate.

\begin{theorem} \label{thm:prod}
	Assume that \textnormal{(A1)--(A3)}, \textnormal{(A5)} and
	\textnormal{(A6) hold, and further assume that $p < \frac{1}{2}$
		in~\textnormal{(A3)}.} If $p\in \bigcup_{k=1}^\infty
	[\frac{2}{4k+3},\frac{2}{4k+1}]$ we take $\alpha :=
	2\lfloor\tfrac{1}{2p}-\tfrac14\rfloor$, $\sigma := \frac{\alpha}{2}$, and
	$\lambda:= \frac{1}{2\sigma-4\delta}$ for arbitrary $\delta\in
	(0,\frac{\sigma}{2}-\frac14)$. If $p\in (\frac{2}{5},\frac{1}{2}) \cup
	\bigcup_{k=1}^\infty (\frac{2}{4k+5},\frac{2}{4k+3})$ we take
	$\alpha:=2\lceil \frac{1}{2p} - \frac{1}{4}\rceil$,
	$\sigma:=\frac{\alpha}{2}$, and $\lambda:=
	\frac{1}{2/p-1-2\sigma-4\delta}$ for arbitrary $\delta\in
	(0,\frac{1}{2p}-\frac{1}{2}-\frac{\sigma}{2})$. We define product weights
	\begin{equation} \label{eq:prod-weights}
	\gamma_{\setu} \,:=\,
	\prod_{j\in\setu}\bigg(\frac{\big[(j\sigma)^{\sigma}{\rm Bell}_{\sigma}(b_j)\big]^2}{2{\rm e}^{1/{\rm e}}\zeta(\alpha\lambda)}
	\bigg)^{\frac{1}{1+\lambda}}
	\qquad\text{for }\emptyset\neq\setu\subseteq\{1:s\},
	\end{equation}
	with $\gamma_{\emptyset}:= 1$. Then the kernel interpolant of the PDE
	solution in Theorem~\ref{thm:PDE-error} satisfies
	\begin{align*}
	& \sqrt{\int_{U_s}\int_D \big(u_{s,h}(\bsx,\bsy) - u_{s,h,n}(\bsx,\bsy)\big)^2 \,\rd\bsx\,\rd\bsy}\\
	&\le
	\begin{cases}
	C\,\|q\|_{H^{-1}(D)}\,n^{-(\frac{1}{2}\lfloor\frac{1}{2p}-\frac14\rfloor-\delta)}
	&\text{for}\ \ p\in\, \bigcup_{k=1}^\infty [\frac{2}{4k+3},\frac{2}{4k+1}], \\
	C\,\|q\|_{H^{-1}(D)}\,n^{-(\frac{1}{2p}-\frac14-\frac12\lceil \frac{1}{2p}-\frac14\rceil-\delta)}
	&\text{for}\ \  p \in\, (\tfrac{2}{5},\tfrac{1}{2})\cup \bigcup_{k=1}^\infty( \frac{2}{4k+5},\frac{2}{4k+3}),
	\end{cases}
	\end{align*}
	where the constant $C>0$ is independent of the truncation dimension $s$.
\end{theorem}

\begin{proof}
	Starting again from the equation \eqref{eq:PODC-lambda}, we apply further crude upper
	bounds $\max(|\setu|,1)\le [\mathrm{e}^{1/\mathrm{e}}]^{|\setu|}$ and
	\begin{align*}
	(\sigma |\setu|)! \,=\, \prod_{j=1}^{|\setu|}\prod_{k=0}^{\sigma-1}(j\sigma-k)
	\,\le\, \prod_{j=1}^{|\setu|} (j\sigma)^\sigma \,\le\, \prod_{j\in \setu}(j\sigma)^{\sigma},
	\end{align*}
	to arrive at
	\begin{align} \label{eq:loose}
	[C_s(\lambda)]^{2\lambda}
	\le \bigg(\!\sum_{\setu\subseteq\{1:s\}}
	\gamma_{\setu}^{\lambda}\,[2\mathrm{e}^{1/\mathrm{e}}\zeta(\alpha\lambda)]^{|\setu|}\bigg)
	\bigg(\!\sum_{\setu\subseteq\{1:s\}} \!\!\!
	\frac{
		\prod_{j\in\setu} \big[(j\sigma)^{\sigma}{\rm Bell}_{\sigma}(b_j)\big]^2}{\gamma_{\setu}} \bigg)^{\lambda}.
	\end{align}
	We equate the terms in the two sums in \eqref{eq:loose} to obtain the
	product weights \eqref{eq:prod-weights}. Plugging the weights back
	into~\eqref{eq:loose} and following the argument in the proof of
	Theorem~\ref{thm:pod}, we obtain
	\begin{align*}
	& [C_s(\lambda)]^{\frac{2\lambda}{1+\lambda}}
	\leq\sum_{\setu\subseteq\{1:s\}}\prod_{j\in \setu} \bigg(
	[2{\rm e}^{1/{\rm e}}\zeta(\alpha\lambda)]^{\frac{1}{1+\lambda}}\,
	[(j\sigma)^{\sigma}{\rm Bell}_{\sigma}(b_j)]^{\frac{2\lambda}{1+\lambda}} \bigg)\\
	&\leq\sum_{\ell=0}^\infty[2{\rm e}^{1/{\rm e}}\zeta(\alpha\lambda)]^{\frac{\ell}{1+\lambda}}\,
	\big[ \sigma^{\sigma} S_{\max}(\sigma)\big]^{\frac{2\lambda}{1+\lambda}\ell}
	\underbrace{
		\sum_{\substack{\setu\subseteq\{1:s\}\\ |\setu|=\ell}}\prod_{j\in \setu}
		\bigg(j^{\sigma}\sum_{m=1}^\sigma b_j^{m}\bigg)^{\frac{2\lambda}{1+\lambda}}}_{=:\, V_\ell},
	\end{align*}
	with
	\begin{align*}
	V_\ell \,\le\, \frac{1}{\ell!}
	\bigg((1+b_1)^{\frac{2\sigma\lambda}{1+\lambda}}
	\sum_{j=1}^\infty(j^{\sigma}b_j)^{\frac{2\lambda}{1+\lambda}}\bigg)^\ell,
	\end{align*}
	Now one can easily check using the ratio test that the term $C_s(\lambda)$
	can be bounded independently of $s$ as long as the series
	$\sum_{j=1}^\infty (j^\sigma b_j)^{\frac{2\lambda}{1+\lambda}}$ is
	convergent.

	From the monotonicity of $(b_j)_{j\ge 1}$ in the assumption (A5) it
	follows that $b_j\leq j^{-1/p}(\sum_{k=1}^\infty b_k^p)^{1/p}$ for all
	$j\ge1$, implying
	$$
	\sum_{j=1}^\infty (j^\sigma b_j)^{\frac{2\lambda}{1+\lambda}}
	\leq \bigg(\sum_{k=1}^\infty b_k^p\bigg)^{\frac{2\lambda p}{1+\lambda}}
	\sum_{j=1}^\infty j^{-(\frac{1}{p}-\sigma)\frac{2\lambda}{1+\lambda}},
	$$
	which is finite provided that
	\[
	\bigg(\frac{1}{p}- \sigma\bigg)\frac{2\lambda}{1+\lambda} > 1
	\quad\iff\quad
	\lambda > \frac{1}{\frac{2}{p}-1-2\sigma}.
	\]
	Taking into account also the requirement that $\frac{1}{\alpha} <
	\lambda\le 1$ and that $\alpha = 2\sigma$ be an even integer, we have the
	constraint
	\begin{align} \label{eq:max}
	\max\bigg(\frac{1}{2\sigma},\frac{1}{\frac{2}{p}-1-2\sigma}\bigg) < \lambda \le 1.
	\end{align}
	We consider two scenarios below depending on the value of the maximum.

	\underline{Scenario A}. If $2\sigma \le \frac{2}{p}-1-2\sigma$ then $p\le
	\frac{2}{4\sigma+1}$ and $\sigma\le \frac{1}{2p} - \frac{1}{4}$, while the
	condition \eqref{eq:max} simplifies to $\frac{1}{2\sigma}< \lambda \le 1$.
	Since $\sigma$ must be an integer and at least~$1$, this scenario applies
	only when $p\in (0,\frac{2}{5}]$. In this case the best convergence rate
	is obtained by taking $\lambda$ as close to $\frac{1}{2\sigma}$ as
	possible and $\sigma$ as large as possible. Hence we take $\sigma :=
	\lfloor \frac{1}{2p} - \frac{1}{4}\rfloor$ and $\lambda:=
	\frac{1}{2\sigma-4\delta}$ for arbitrary $\delta\in
	(0,\frac{\sigma}{2}-\frac14)$. By Theorem~\ref{thm:PDE-error} this yields
	the convergence rate $\mathcal{O}(n^{-(\frac{1}{2} \lfloor \frac{1}{2p} -
		\frac{1}{4}\rfloor - \delta)})$ with the implied constant independent of
	the dimension~$s$, but approaching $\infty$ as $\delta \to 0$.

	\underline{Scenario B}. On the other hand, if $2\sigma >
	\frac{2}{p}-1-2\sigma$ then $p > \frac{2}{4\sigma+1}$ and $\sigma >
	\frac{1}{2p} - \frac{1}{4}$, while the condition \eqref{eq:max} becomes
	$\frac{1}{2/p-1-2\sigma}< \lambda \le 1$. Additionally, for the latter
	condition on $\lambda$ to hold we require that $\frac{2}{p} -1 - 2\sigma >
	1$, which means $p < \frac{1}{\sigma+1}$ and $\sigma< \frac{1}{p}-1$.
	Combining all constraints we have
	\[
	\frac{1}{\frac{2}{p}-1-2\sigma}< \lambda \le 1
	\quad\mbox{and}\quad
	\frac{2}{4\sigma+1} < p < \frac{1}{\sigma+1}
	\quad\mbox{and}\quad
	\frac{1}{2p} - \frac{1}{4} < \sigma < \frac{1}{p}-1.
	\]
	Since $\sigma$ must be an integer and at least $1$, this scenario applies
	only when $p\in \bigcup_{k=1}^\infty (\frac{2}{4k+1},\frac{1}{k+1}) =
	(0,\tfrac{1}{3}) \cup (\tfrac{2}{5},\tfrac{1}{2})$. In this case the best
	convergence rate is obtained by taking $\lambda$ as close to
	$\frac{1}{2/p-1-2\sigma}$ as possible but now with $\sigma$ as small as
	possible. Hence we take $\sigma := \lceil \frac{1}{2p} -
	\frac{1}{4}\rceil$ and $\lambda:= \frac{1}{2/p-1-2\sigma-4\delta}$ for
	arbitrary $\delta\in (0,\frac{1}{2p}-\frac{1}{2}-\frac{\sigma}{2})$. This
	yields the convergence rate $\mathcal{O}(n^{- (\frac{1}{2p} - \frac{1}{4}
		- \frac{1}{2} \lceil \frac{1}{2p} - \frac{1}{4}\rceil - \delta)})$, with
	the implied constant independent of the dimension~$s$.

	If $p \in (\frac{2}{5},\frac{1}{2})$ then only Scenario B applies.
	\vspace{0.1cm}

	If $p\in [\frac{1}{3},\frac{2}{5}]$ then only Scenario A applies.
	\vspace{0.1cm}

	If $p\in (0,\tfrac{1}{3})$ then both scenarios apply, and it remains to
	resolve which scenario to use in order to obtain the better convergence
	rate. For convenience we abbreviate $x := \frac{1}{2p}-\frac14$ and $m :=
	\lfloor\frac{1}{2p}-\frac14\rfloor$, noting that $m\ge 1$ since $p <
	\frac{1}{3}$. Scenario~B has a better convergence rate than Scenario~A if
	and only if $\frac{1}{2} \lfloor x\rfloor < x - \frac{1}{2}\lceil
	x\rceil$. The latter condition is not satisfied if $x\in \mathbb{Z}$,
	while for $x\notin \mathbb{Z}$ the condition is equivalent to $\lfloor x
	\rfloor + \tfrac12 < x < \lceil x \rceil $. Hence the condition is
	equivalent to
	\[
	m + \frac{1}{2} < \frac{1}{2p}-\frac14 < m + 1
	\quad\iff\quad
	\frac{2}{4m+5} < p < \frac{2}{4m+3}.
	\]
	We conclude that for the case  $p<\tfrac{1}{3}$ we should use Scenario~B when $p\in
	\bigcup_{k=1}^\infty \Bigl(\frac{2}{4k+5},\frac{2}{4k+3}\Bigr)$  and use Scenario~A
	when $p\in [\frac{2}{7},\frac{1}{3}) \cup \bigcup_{k=2}^\infty
	[\frac{2}{4k+3},\frac{2}{4k+1}]$.

	Combining the above analysis, we should apply Scenario~B when $p\in
	(\frac{2}{5},\frac{1}{2}) \cup \bigcup_{k=1}^{\infty}
	(\frac{2}{4k+5},\frac{2}{4k+3})$ and apply Scenario~A when $p \in
	[\frac{2}{7},\frac{2}{5}] \cup \bigcup_{k=2}^{\infty}$
	$[\frac{2}{4k+3},\frac{2}{4k+1}] = \bigcup_{k=1}^{\infty}
	[\frac{2}{4k+3},\frac{2}{4k+1}]$.
\end{proof}
\subsection{Combined approximation error}

The combined approximation error of the PDE problem~\eqref{eq:pdeweak} can be decomposed as
\begin{align*}
&\sqrt{\int_U\int_D(u(\bsx,\bsy)-u_{s,h,n}(\bsx,\bsy))^2\,{\rm d}\bsx\,{\rm d}\bsy}\\
&\qquad\qquad\qquad\qquad
\,\leq\, \: \sqrt{\int_U\int_D(u(\bsx,\bsy)-u_s(\bsx,\bsy))^2\,{\rm d}\bsx\,{\rm d}\bsy}\\
&\qquad\qquad\qquad\qquad\qquad
+\sqrt{\int_U\int_D(u_s(\bsx,\bsy)-u_{s,h}(\bsx,\bsy))^2\,{\rm d}\bsx\,{\rm d}\bsy}\\
&\qquad\qquad\qquad\qquad\qquad
+\sqrt{\int_U\int_D(u_{s,h}(\bsx,\bsy)-u_{s,h,n}(\bsx,\bsy))^2\,{\rm d}\bsx\,{\rm d}\bsy},
\end{align*}
where the first term is the dimension truncation error, the second term
is the finite element error, and the final term is the kernel
interpolation error. Combining the results developed in
Sections~\ref{sec:dimtruncerror}--\ref{sec:kernelerror}, we arrive at the
following result.

\begin{theorem}\label{thm:combined_error}
	Assume that \textnormal{(A1)--(A6)} hold. For any $\bsy\in U$, let
	$u(\cdot,\bsy)\in H_0^1(D)$ denote the solution to~\eqref{eq:pdeweak} with
	the source term $q\in H^{-1+t}(D)$ for some $0\leq t\leq 1$. Let
	$u_{s,h}(\cdot,\bsy)\in V_h$ be the corresponding dimensionally truncated
	finite element solution and let
	$u_{s,h,n}(\bsx,\cdot)=A^\ast_n(u_{s,h}(\bsx,\cdot))$ be its kernel
	interpolant constructed using the weights described in
	Theorems~\ref{thm:spod}, \ref{thm:pod}, or~\ref{thm:prod}. Then we have
	the combined error estimate
	\begin{align*}
	&\sqrt{\int_U\int_D (u(\bsx,\bsy)-u_{s,h,n}(\bsx,\bsy))^2\,{\rm d}\bsx\,{\rm d}\bsy}\\
	&\qquad\qquad
	\,\le\, C\,\Big( \big(s^{-(\frac{1}{p}-\frac{1}{2})}+n^{-r}\big)\|q\|_{H^{-1}(D)}+h^{1+t}\|q\|_{H^{-1+t}(D)}\Big),
	\end{align*}
	where $0\leq t\leq 1$, $h$ denotes the  mesh size of the piecewise linear
	finite element mesh, $C>0$ is a constant independent of $s$, $h$, $n$,
	$q$, and
	\begin{align*}
	r =
	\begin{cases}
	\frac{1}{2p}-\frac14
	& \text{with weights \eqref{eq:bad-weights} or SPOD weights \eqref{eq:spod-weights}}, \\[2.5pt]
	\frac{1}{2p}-\frac14
	& \text{with POD weights \eqref{eq:weights2} for }
	p \,\in\,\bigcup_{k=1}^\infty (\frac{2}{2k+1},\frac{1}{k}), \\[3pt]
	\frac{1}{2}\lfloor\frac{1}{2p}-\frac14\rfloor -\delta
	&
	\text{with product weights } \text{\eqref{eq:prod-weights}} \text{ for } \\&
	\qquad\qquad\qquad{\textstyle p \,\in\,\bigcup_{k=1}^\infty [\frac{2}{4k+3},\frac{2}{4k+1}]},\\
	\frac{1}{2p}-\frac14-\frac12\lceil \frac{1}{2p}-\frac14\rceil -\delta\!\!\!\!\!
	&
	\text{with product weights } \text{\eqref{eq:prod-weights}} \text{ for } \\[1.5pt]
	&
	\qquad\qquad\qquad{\textstyle p \,\in\, (\tfrac{2}{5},\tfrac{1}{2}) \,\cup\, \bigcup_{k=1}^\infty( \frac{2}{4k+5},\frac{2}{4k+3})},
	\end{cases}
	\end{align*}
	and $\delta>0$ is sufficiently small in each case.
\end{theorem}

\section{Cost analysis}\label{sec:cost}
\subsection{What is the point set at which values are wanted?}
In this section we consider the cost of evaluating the kernel interpolant
\[
f_n(\bsy) \,=\, \sum_{k=1}^n a_k\, K(\bst_k,\bsy),
\]
as an approximation to the periodic function~$f$, with lattice points
$\bst_k = \{\frac{k\bsz}{n}\}$, $k=1,\ldots,n$, and
$\bst_n=\bst_0=\bszero$. Recall that all our functions including the
kernel are $1$-periodic with respect to~$\bsy$. For the linear system
\eqref{eq:lin_sys}, as observed already, the matrix $\calK =
[K(\bst_k-\bst_{k'},\bszero)]_{k,k'=1,\ldots,n}$ is circulant, thus we
need to compute only its first column (see the cost for evaluating the
kernel in the next subsection) and then solve for the coefficients $a_k$
with a cost of $\calO(n\log(n))$.

First, however, it turns out to be useful to ask: what is the set of
points, say $\{\bsy_1, \bsy_2,\ldots\}$, at which the values of the
interpolant are desired?  If $L$ such  points $\bsy_\ell$, $\ell =
1,\ldots, L$, are chosen arbitrarily then the cost, naturally, is $L$
times the cost of a single evaluation. On the other hand, for a set of
$Ln$ points formed by the \emph{union of shifted lattices} $\bsy_\ell +
\bst_{k'}$, $\ell = 1,\ldots,L$, $k' = 1,\ldots n$, it turns out that the
cost for $Ln$ evaluations is little more than the cost of the $L$
evaluations at arbitrary points.

The reason for the low cost lies in the shift invariance of the kernel and
the group nature of the lattice. For a single given $\bsy$ the principal
costs for evaluating the kernel interpolant come from evaluating
$K(\bst_k,\bszero)$ and $f(\bst_k)$ at the $n$ lattice points; then
solving the circulant linear system \eqref{eq:lin_sys} for the $n$ values
of $a_k$; from evaluating $K(\bst_k,\bsy)$ at the $n$ lattice points; and
finally from assembling $f_n(\bsy)$ with a cost of $\calO(n)$. (The
precise cost breakdown is given in Table~\ref{tab:cost1} below after we
discuss the cost for evaluating the kernel in the next subsection.)

But for evaluation of $K(\bst_k,\bsy + \bst_{k'})$ for all $n$ values $k'=
1,\ldots, n$ we observe that $K(\bst_k,\bsy + \bst_{k'}) = K(\bst_k
-\bst_{k'},\bsy)$, and hence
\begin{equation}\label{eq:efficientfn}
f_n(\bsy + \bst_{k'}) \,=\, \sum_{k=1}^n a_k\,K(\bst_k - \bst_{k'}, \bsy).
\end{equation}
Since the right-hand side has the form of a circulant $n\times n$ matrix
multiplying a vector of length~$n$, the $n$ values $f_n(\bsy +
\bst_{k'})$ for $k'=1,\dots,n$ can be assembled with a cost of $\calO(n
\log(n))$, compared with the $\calO(n)$ cost of assembling $f_n$ at a
single value of $\bsy$.
\subsection{Cost for evaluating the kernel for a single \texorpdfstring{$\bsy$}{bsy}}Now consider the cost of computing $K(\bst,\bsy)$ for a single arbitrary
value of $\bsy$ and arbitrary $\bst$,
\begin{align*}
K(\bst,\bsy)
\,=\, \sum_{\setu\subseteq\{1:s\}} \gamma_\setu  \prod_{j\in\setu} \eta_\alpha(t_j,y_j)
\qquad\mbox{for}\quad k=1,\ldots,n.
\end{align*}
In the following, we assume that evaluating $\eta_\alpha$ can be treated
as having constant cost. For example, when $\alpha$ is even we have an
analytic formula for $\eta_\alpha$ in terms of the Bernoulli polynomial.

If the weights have no special structure then the cost to evaluate
$K(\bst,\bsy)$ would be exponential in $s$ because of the sum over subsets
of $\{1:s\}$, but the cost is much reduced in special cases:
\setlength{\leftmargini}{1.5em}
\begin{itemize}
	\item[\textbullet] With product weights we have $K(\bst,\bsy) = \prod_{j=1}^s (1 +
	\gamma_j\eta_\alpha(t_j,y_j))$, which can be evaluated for a pair
	$(\bst,\bsy)$ at the cost of $\calO(s)$.

	\item[\textbullet] With POD weights we have
	\begin{align*}
	K(\bst,\bsy)
	=\!\!\!\;
	\sum_{\setu\subseteq\{1:s\}}\!\!\!\;
	\Gamma_{|\setu|}\!\!\; \prod_{j\in\setu} \Big(\gamma_j\, \eta_\alpha(t_j,y_j)\Big)
	= \sum_{\ell=0}^s \Gamma_\ell\,\underbrace{\sum_{\substack{\setu\subseteq\{1:s\}\\ |\setu|=\ell}}\;
		\prod_{j\in\setu} \Big(\gamma_j\, \eta_\alpha(t_j,y_j)\Big)}_{=:\, P_{s,\ell}},
	\end{align*}
	where $P_{s,\ell}$ is defined for $\ell=0,\ldots,s$, and can be
	computed recursively using
	\[
	P_{s,\ell} \,=\, P_{s-1,\ell} + \gamma_s\,\eta_\alpha(t_s,y_s)\, P_{s-1,\ell-1},
	\]
	together with $P_{s,0}:=1$ for all $s$ and $P_{s,\ell}:=0$ for all
	$\ell>s$. The cost to evaluate this for a pair $(\bst,\bsy)$ is
	$\calO(s^2)$.

	\item[\textbullet] With SPOD weights we have
	\begin{align*}
	K(\bst,\bsy)&
	\,=\, \sum_{\setu\subseteq\{1:s\}} \sum_{\bsnu_\setu\in \{1:\sigma\}^{|\setu|}} \Gamma_{|\bsnu_\setu|}
	\prod_{j\in\setu} \Bigl(\gamma_{j,\nu_j}\, \eta_\alpha(t_j,y_j)\Bigr) \\
	&\,=\, \sum_{\bsnu\in \{0:\sigma\}^s} \Gamma_{|\bsnu|}
	\prod_{j:\,\nu_j>0} \Bigl(\gamma_{j,\nu_j}\, \eta_\alpha(t_j,y_j)\Bigr)\\
	&
	\,=\, \sum_{\ell=0}^{s\sigma}
	\Gamma_\ell
	\underbrace{\sum_{\substack{\bsnu\in \{0:\sigma\}^s\\ |\bsnu|=\ell}} \;
		\prod_{j:\,\nu_j>0} \Bigl(\gamma_{j,\nu_j}\, \eta_\alpha(t_j,y_j)\Bigr)}_{=:\,P_{s,\ell}},
	\end{align*}
	where $P_{s,\ell}$ is now defined for $\ell=0,\ldots,s\sigma$, and can
	be computed recursively using
	\begin{align*}
	P_{s,\ell}
	\,=\, P_{s-1,\ell} + \eta_\alpha(t_s,y_s)\sum_{\nu=1}^{\min(\sigma,\ell)} \gamma_{s,\nu}\, P_{s-1,\ell-\nu},
	\end{align*}
	together with $P_{s,0}:=1$ for all $s$ and $P_{s,\ell}:=0$ for all
	$\ell>s\sigma$. The cost to evaluate this for a pair
	$(\bst,\bsy)$ is now
	$\calO(s^2\,\sigma^2)$.
\end{itemize}

\subsection{Cost for the kernel interpolant}  \label{sec:cost2}
We now summarize the cost for the kernel interpolant and different weights
using the results of the preceding two subsections. Let $X$ denote the
cost for one evaluation of~$f$. The cost breakdown is shown in
Table~\ref{tab:cost1}. The first four rows are considered to be
pre-computation cost while the last three rows are the running cost for
sampling. The cost for the fast CBC construction based on the
criterion $\calS_s(\bsz)$ with different weight parameters is analyzed in
\cite{CKNS-part2}.

For the PDE application, our kernel method is
\[
u_{s,h}(\bsx_i,\bsy) \,\approx\,
u_{s,h,n}(\bsx_i,\bsy) \,=\, \sum_{k=1}^n a_k(\bsx_i)\, K(\bst_k,\bsy),
\]
where $\{\bsx_i : i=1,\ldots,M\} \subset D$ is the set of finite element
nodes in the physical domain, and $a_k(\bsx_i)$ for $k = 1,\ldots n$ is
the solution for fixed $\bsx_i$ of the linear system
\[
\sum_{k=1}^n \calK_{k,k^{'}}\,a_k(\bsx_i)\,=\,u_{s,h}(\bsx_i,\bst_{k'}), \quad k'=1,\ldots,n.
\]

\begin{table}[t]
	\caption{Cost breakdown for the kernel interpolant $f_n$ based on $n$
		lattice points $\bst_k$ in $s$ dimensions, evaluated at $L$
		arbitrary
		points $\bsy_\ell$. Here $X$ is the cost for one evaluation
		of $f$.} \label{tab:cost1} \setlength{\arrayrulewidth}{0.2pt}
	\medskip\centering
	\begin{tabular}{|l|l@{\,}|l@{\,}|l@{\,}|}
		\hline
		\hfill Operation $\backslash$ Weights & Product & POD & SPOD\Tstrut\\
		\hline
		Fast CBC construction for $\bsz$
		& $s\,n\log(n)$
		&$s\,n\log(n)$&$s\,n\log(n)+ s^3\sigma^2\, n$ \Tstrut\Tstrut\\
		&&$+ s^2\log(s)\, n$&\Bstrut\\[1.5pt]
		Compute $K(\bst_k,\bszero)$ for all $k$
		& $s\,n$ & $s^2\, n$ & $s^2\,\sigma^2\,n$ \\
		Evaluate $f(\bst_k)$ for all $k$
		& $X\,n$ & $X\,n$ & $X\,n$ \\
		Linear solve for all coefficients $a_k$
		& $n\log(n)$ & $n\log(n)$ & $n\log(n)$ \\
		\hline
		Compute $K(\bst_k,\bsy_\ell)$ for all $k,\ell$
		& $s\,n\,L$ & $s^2\, n\,L$ & $s^2\,\sigma^2\,n\,L$ \Tstrut\\
		Assemble $f_n(\bsy_\ell)$ for all $\ell$
		& $n\,L$ & $n\,L$ & $n\,L$ \\
		OR Assemble $f_n(\bsy_\ell+\bst_{k})$ for all $\ell,k$
		& $n\,\log(n)\,L$ & $n\,\log(n)\,L$ & $n\,\log(n)\,L$ \\
		\hline
	\end{tabular}
\end{table}
\begin{table}[t] \setlength{\arrayrulewidth}{0.2pt}
	\caption{Cost breakdown for the kernel interpolant $u_{s,h,n}$ based
		on
		$n$ lattice points $\bst_k$ in $s$ dimensions, evaluated at $M$
		finite
		element nodes $\bsx_i$ and $L$ arbitrary points $\bsy_\ell$. Here
		$M^a$
		for some positive $a$ is the cost for one finite element solve
		with $M$
		nodes.} \label{tab:cost2}
	\medskip\centering
	\begin{tabular}{|l|l@{\,}|l@{\,}|l@{\,}|}
		\hline
		\hfill Operation $\backslash$ Weights & Product & POD & SPOD \Tstrut\\
		\hline
		Fast CBC construction for $\bsz$
		& $s\,n\log(n)$
		& $s\,n\log(n)$&
		$s\,n\log(n)+ s^3\sigma^2\, n$ \Tstrut\Tstrut\\
		&&$+ s^2\log(s)\, n$&\\
		[1.5pt]
		Compute $K(\bst_k,\bszero)$ for all $k$
		& $s\,n$ & $s^2\, n$ & $s^2\,\sigma^2\,n$ \\
		Evaluate $u_{s,h}(\bsx_i,\bst_k)$ for all $i, k$
		& $M^a\,n$ & $M^a\,n$ & $M^a\,n$ \\
		Linear solve for all coeff. $a_k(\bsx_i)$
		& $M\,n\log(n)$ & $M\,n\log(n)$ & $M\,n\log(n)$ \\
		\hline
		Compute $K(\bst_k,\bsy_\ell)$ for all $k,\ell$
		& $s\,n\,L$ & $s^2\, n\,L$ & $s^2\,\sigma^2\,n\,L$\Tstrut \\
		Assemble $u_{s,h,n}(\bsx_i,\bsy_\ell)$ for all $i,\ell$
		& $M\,n\,L$ & $M\,n\,L$ & $M\,n\,L$ \\
		OR Assemble $u_{s,h,n}(\bsx_i,\bsy_\ell+\bst_k)$
		& $M\,n\,\log(n)\,L$ & $M\,n\,\log(n)\,L$ & $M\,n\,\log(n)\,L$\\[-0.05pt]
		\qquad\qquad\qquad\qquad\quad\ \   for all $i,\ell,k$ & & & \\
		\hline
	\end{tabular}
\end{table}

Let $M^a$ for some $a\,\ge\,1$ denote the cost of the finite element solve
to obtain all $\bsx_i$ for one~$\bsy$. The cost breakdown for obtaining
the kernel interpolant at all $M$ nodes for all $L$ samples is shown in
Table~\ref{tab:cost2}. Note in this case that the coefficients
$a_k(\bsx_i)$ need to be computed for every finite element node $\bsx_i$,
hence the scaling of the cost in line 4 of Table~\ref{tab:cost2} by $M$.
If the quantity of interest is a \emph{linear functional} of the PDE
finite element solution (no need for the solution at every node), then the
cost is reduced to be as in Table~\ref{tab:cost1} with $X=M^a$.

\section{Numerical experiments}\label{sec:numerical}

We consider the parametric PDE
problem~\eqref{eq:pdestrong}--\eqref{eq:pdestrong2} in the physical domain
$D=(0,1)^2$  with the source term $q(\bsx)=x_2$ and the diffusion coefficient periodic in the parameters  $\bsy$ given by~\eqref{eq:a-param}.

For each fixed $\bsy\in U_s$ (i.e. with the sum in \eqref{eq:a-param} truncated to $s$ terms), we solve the PDE using a piecewise linear
finite element method with $h=2^{-5}$ as the finite element mesh size. As
the stochastic fluctuations, we consider the functions
$$
\psi_j(\bsx) := c\, j^{-\theta}\sin(j\pi x_1)\sin(j\pi x_2),\qquad\bsx=(x_1,x_2)\in D,~j\geq 1,
$$
where $c>0$ is a constant,  $\theta>1$ is the decay rate of the stochastic
fluctuations, and $s\in\mathbb N$ is the truncation dimension. Following
\eqref{eq:def-bj}, the sequence $(b_j)_{j\geq 1}$ is taken to be
\[
b_j := \frac{c\, j^{-\theta}}{\sqrt{6}\;a_{\min}}, \quad\!\mbox{with}\quad\!
a_{\min}:=1-\frac{c}{\sqrt 6}\,\zeta(\theta)
\quad\mbox{as well as}\quad
a_{\max}:=1+\frac{c}{\sqrt 6}\,\zeta(\theta),
\]
and $c<\frac{\sqrt{6}}{\zeta(\theta)}$, ensuring that the assumption (A2) is satisfied.
\begin{figure}[t]
	\centering
	\includegraphics[height=.36\textwidth]{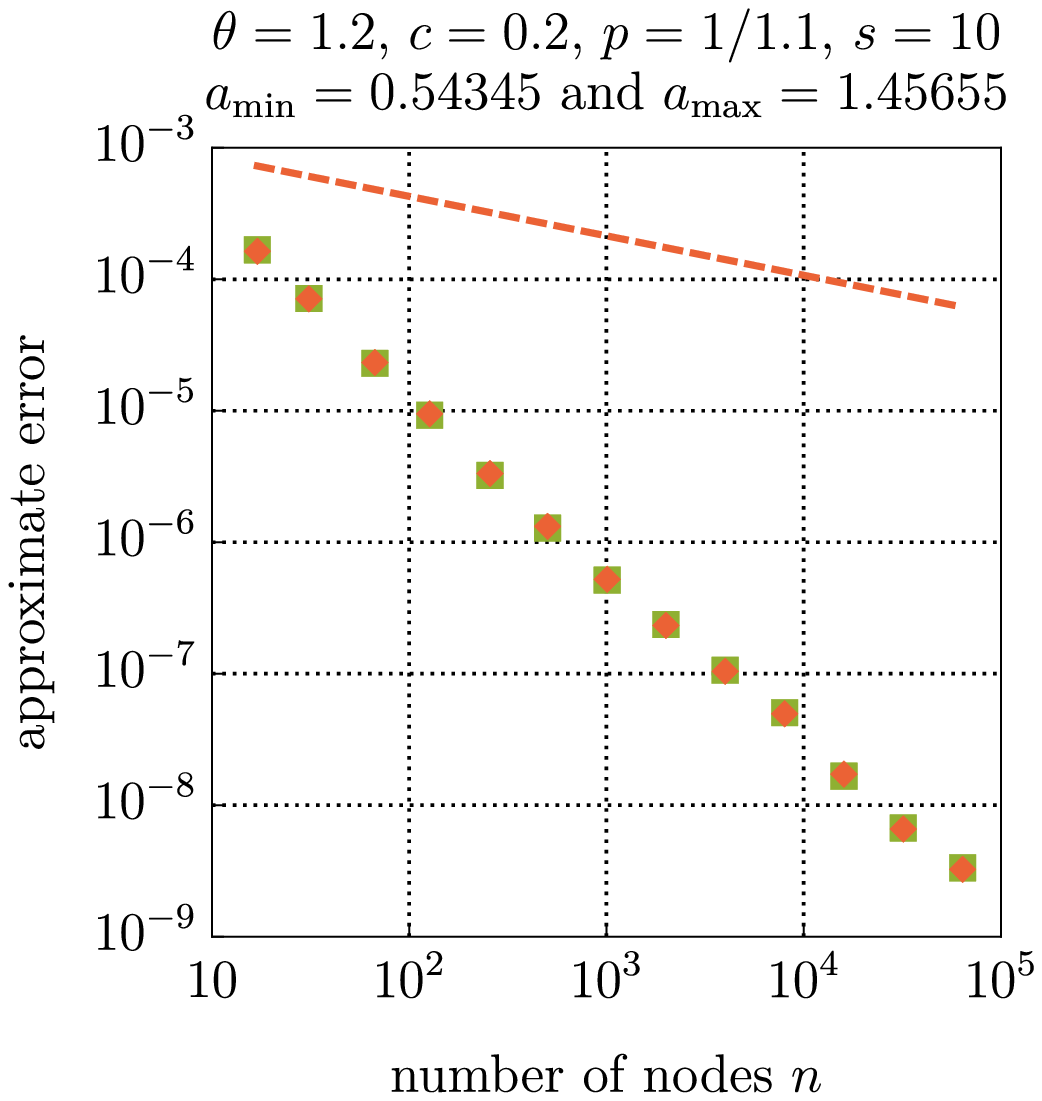}\includegraphics[height=.36\textwidth]{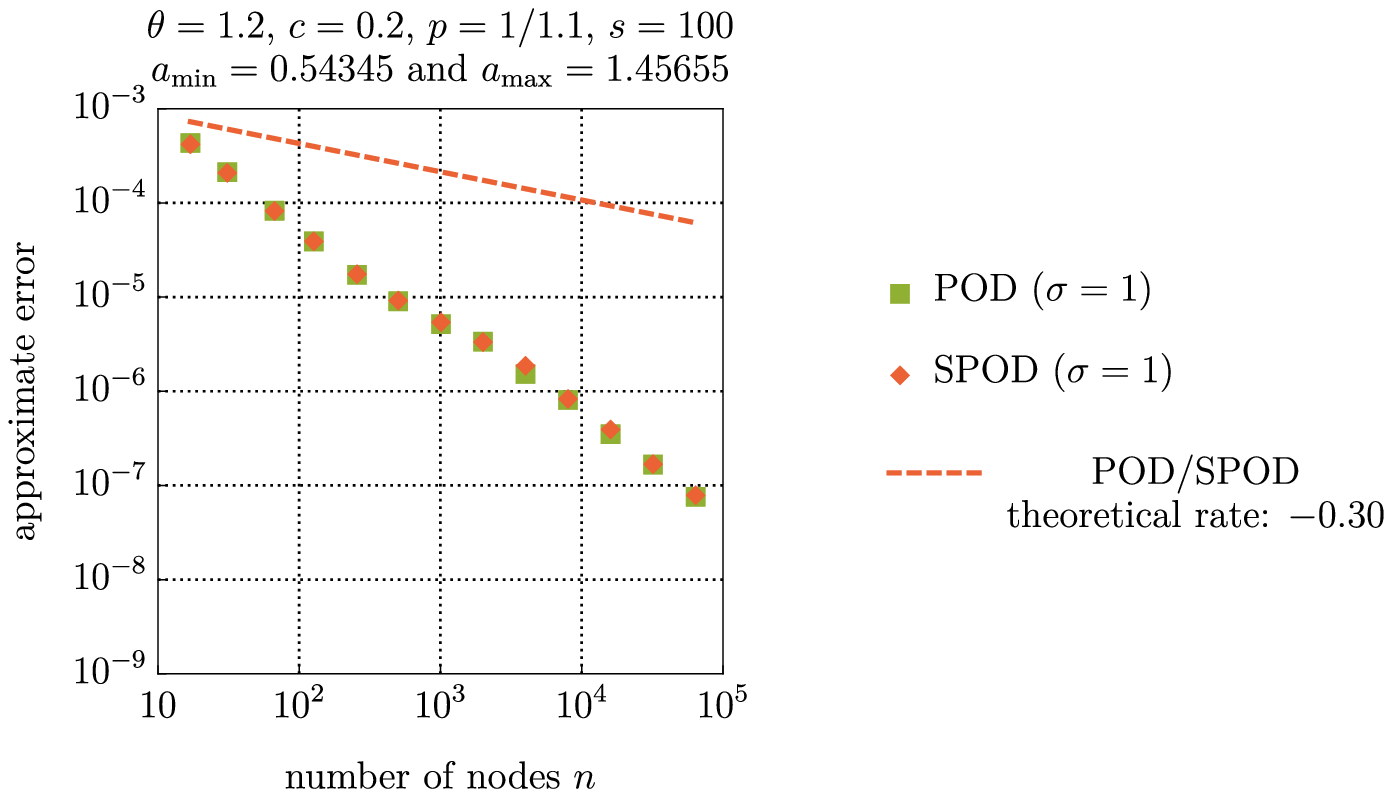}\\
	\bigskip
	\includegraphics[height=.36\textwidth]{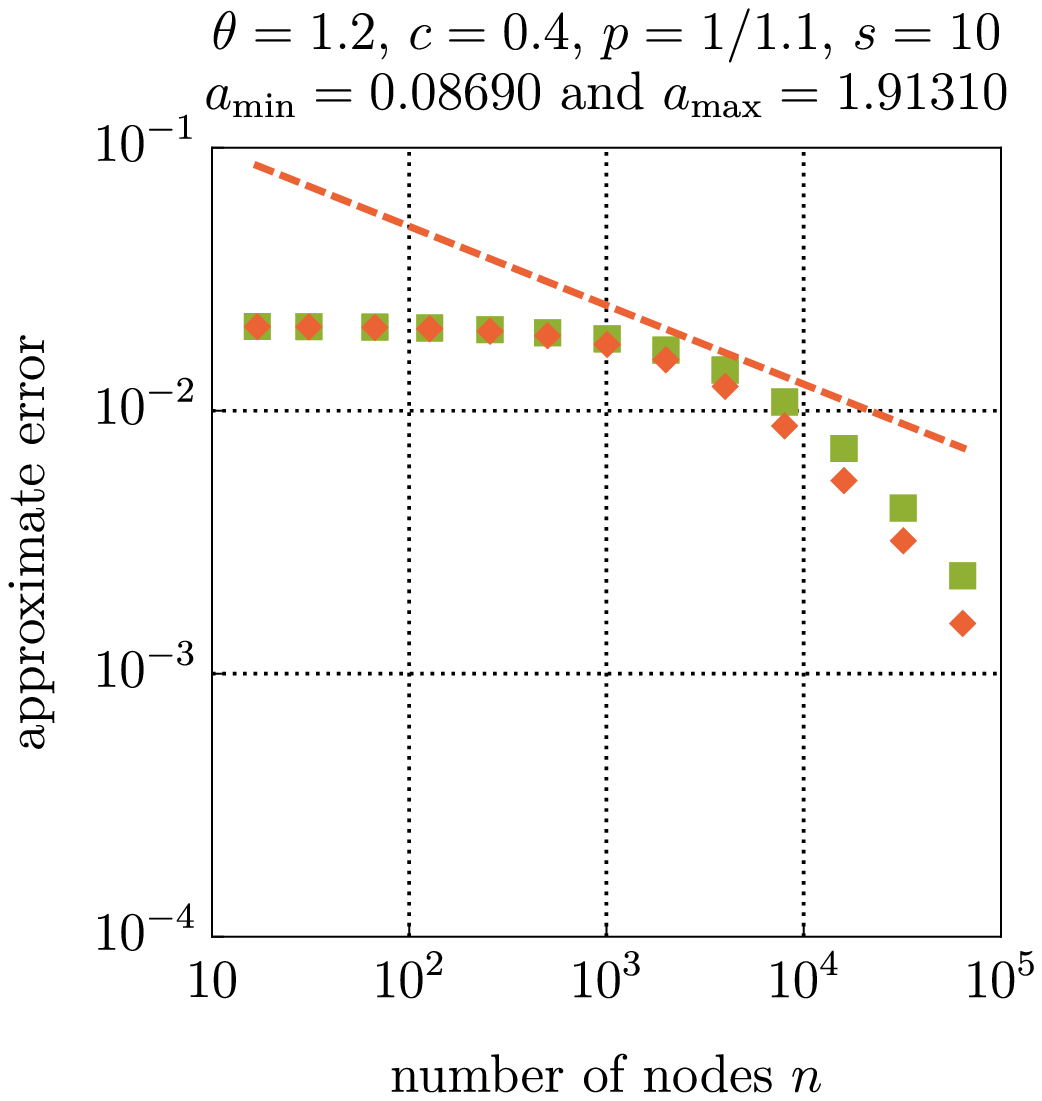}\includegraphics[height=.36\textwidth]{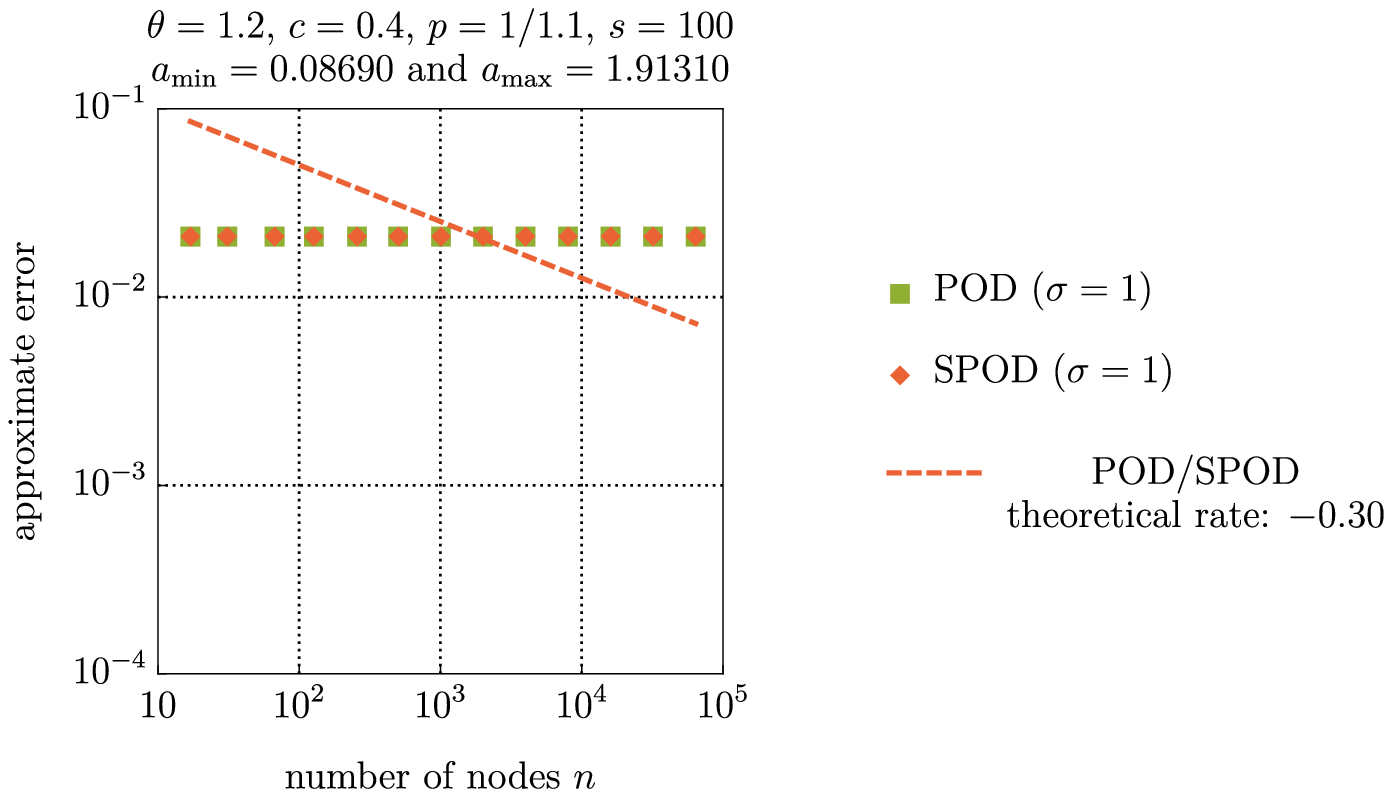}
	\caption{The kernel interpolation errors of the PDE problem~\eqref{eq:pdestrong}--\eqref{eq:pdestrong2}
		with $\theta=1.2$, $p=1/1.1$, $c\in\{0.2,0.4\}$, and $s\in\{10,100\}$.
		Results are displayed for kernel interpolants constructed using POD
		and SPOD weights. (Product weights~\eqref{eq:prod-weights} are not well-defined in this case.)}\label{fig:1}
\end{figure}

\begin{figure}[t]
	\centering
	\includegraphics[height=.36\textwidth]{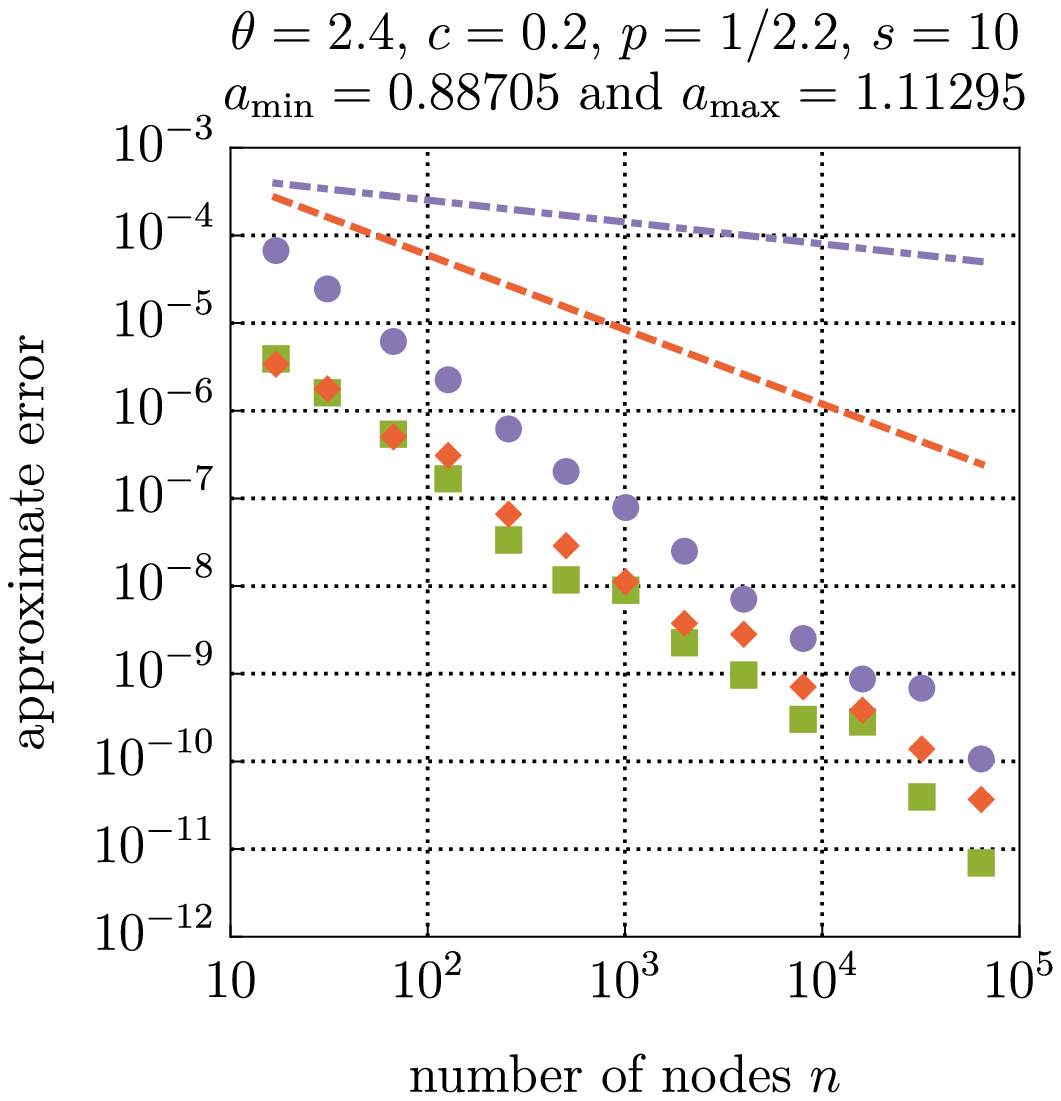}\includegraphics[height=.36\textwidth]{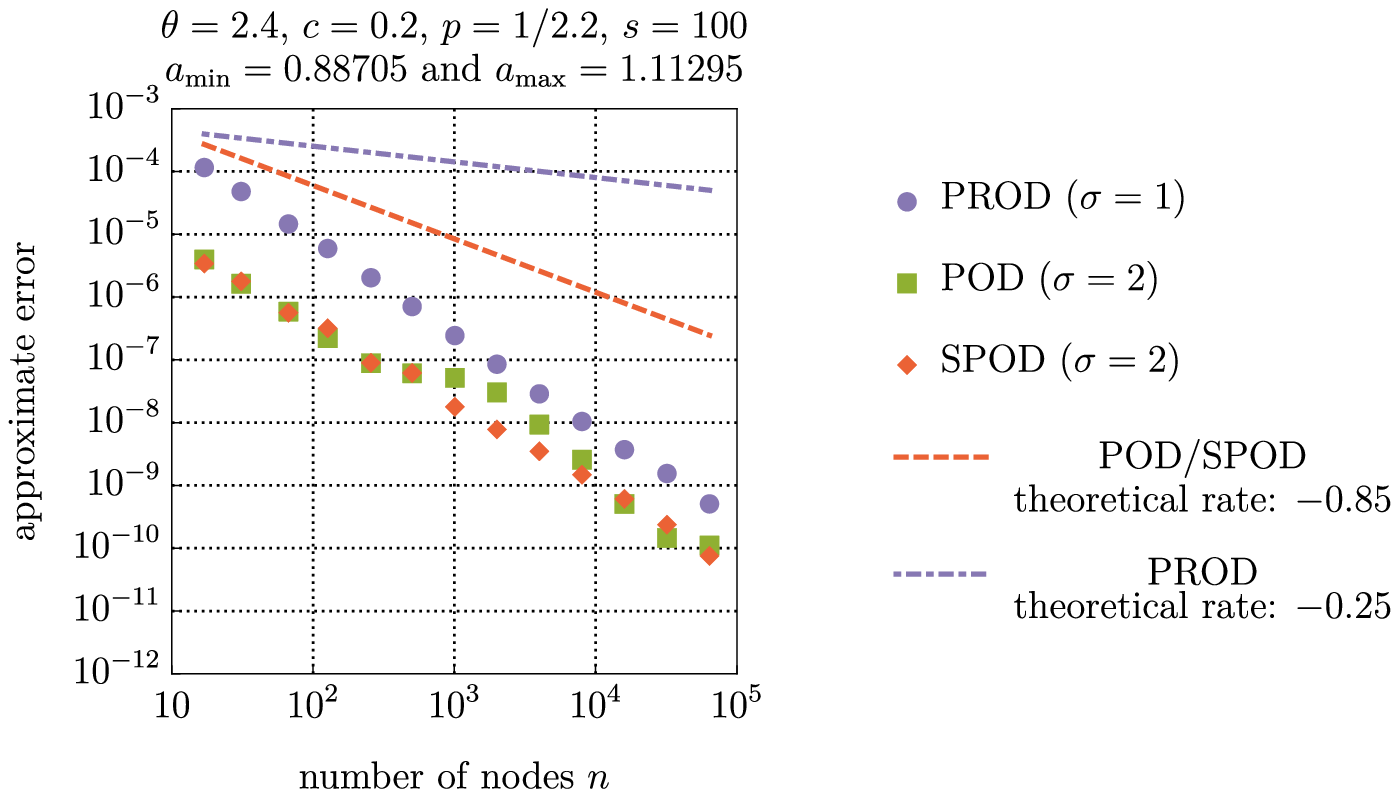}\\
	\bigskip
	\hspace*{.2cm}
	\includegraphics[height=.36\textwidth]{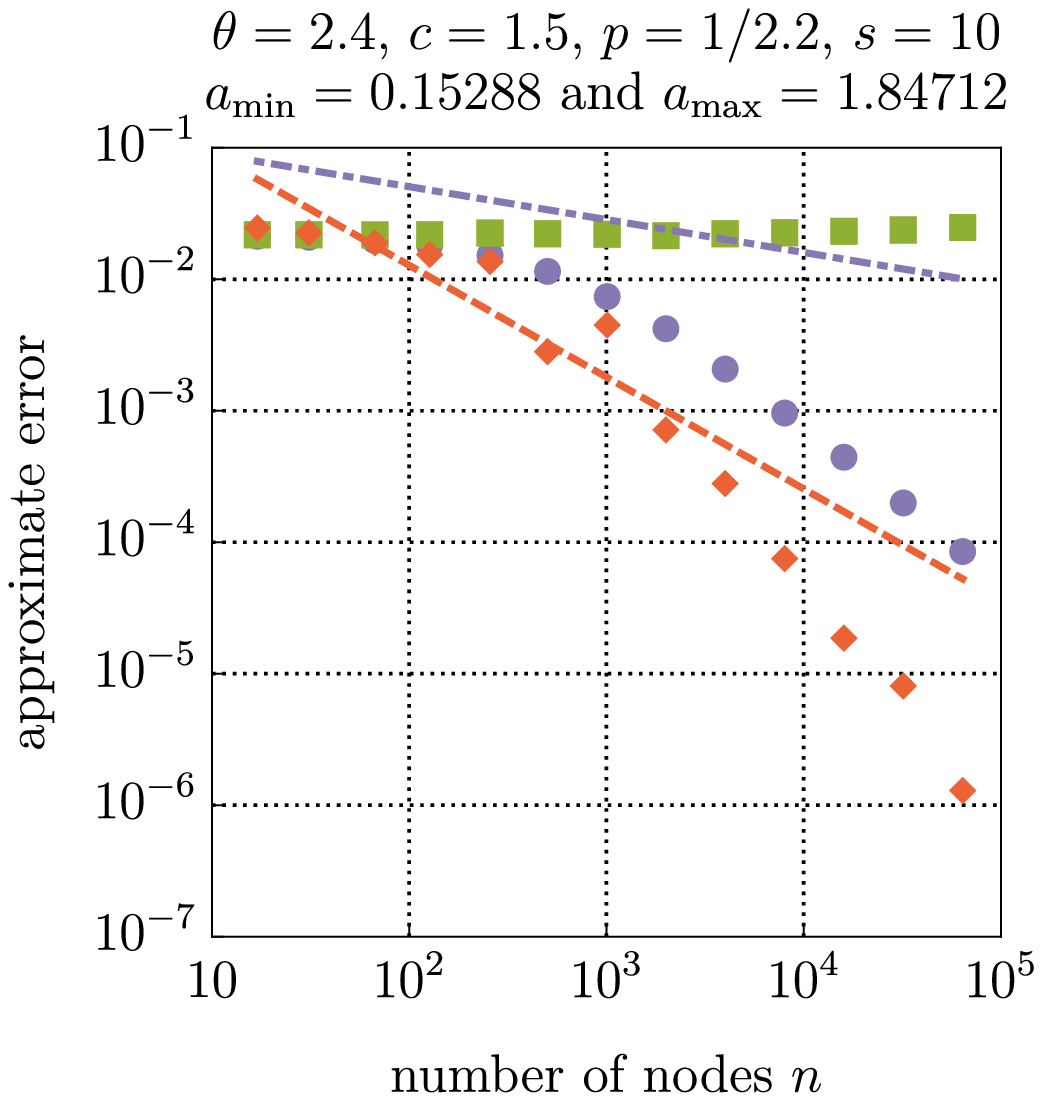}\includegraphics[height=.36\textwidth]{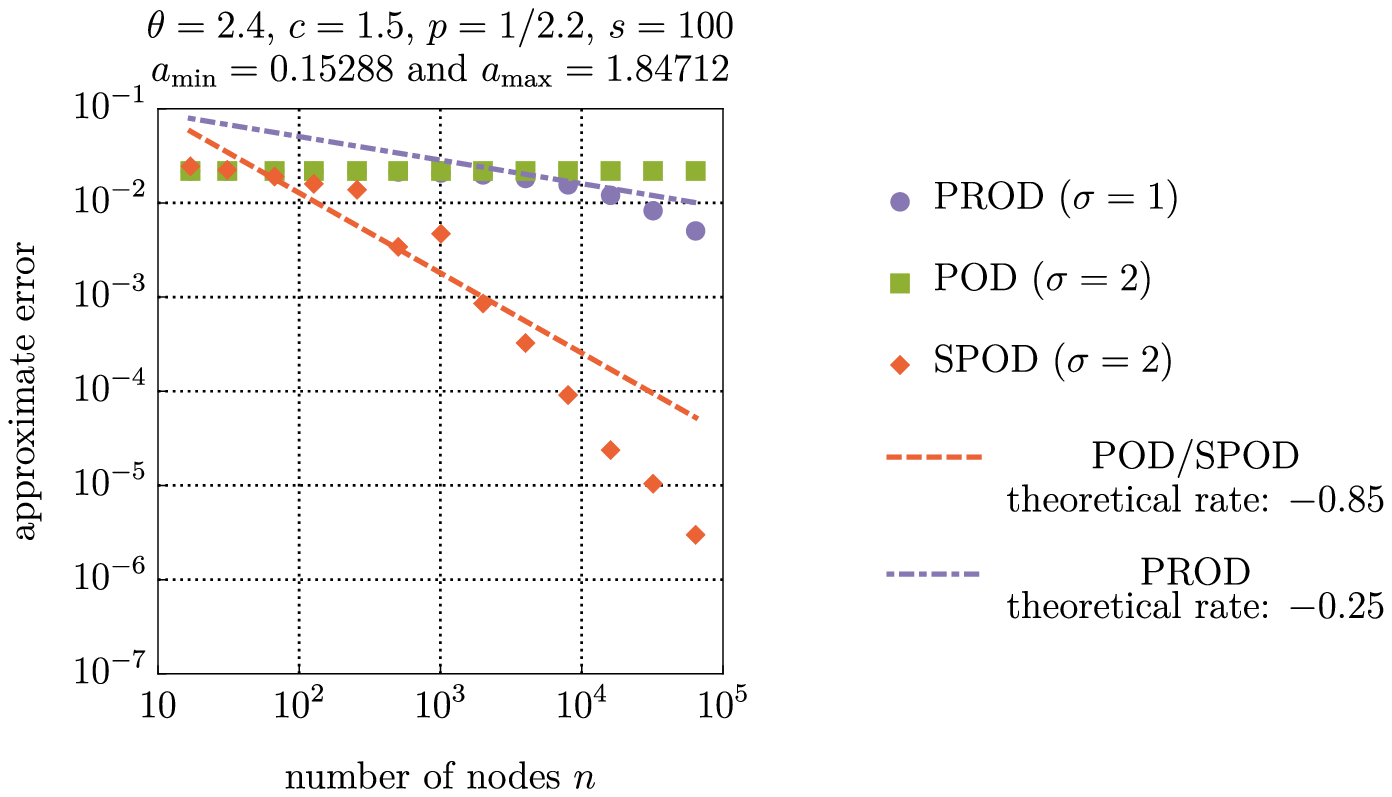}
	\caption{The kernel interpolation errors of the PDE problem~\eqref{eq:pdestrong}--\eqref{eq:pdestrong2}
		with  $\theta=2.4$, $p=1/2.2$, $c\in\{0.2,1.5\}$, and $s\in\{10,100\}$.
		Results are displayed for kernel interpolants constructed using product (PROD), POD, and SPOD weights.}\label{fig:2}
\end{figure}

We approximate the dimensionally truncated finite element solution
$u_{s,h}$ of the PDE~\eqref{eq:pdestrong}--\eqref{eq:pdestrong2} by
constructing a kernel interpolant
$u_{{s,h,n}}(\bsx,\bsy):=A^\ast_n(u_{s,h}(\bsx,\bsy))$, $\bsx\in D$ and
$\bsy\in U_s$ using SPOD weights, POD weights, and product weights chosen
according to Theorem~\ref{thm:spod}, Theorem~\ref{thm:pod}, and
Theorem~\ref{thm:prod}, respectively. The same weights appear in the
formula for the kernel as well as the search criterion for finding good
lattice generating vectors. The kernel interpolant is constructed over a
lattice point set $\boldsymbol t_k:=\{k\boldsymbol z/n\}$, $k\in
\{1,\ldots,n\}$, where the generating vector $\bsz\in\{1,\ldots,n-1\}^s$
has been obtained separately for each weight type using the fast CBC
algorithm detailed in~\cite{CKNS-part2}. We assess the kernel
interpolation error by computing
\begin{align*}
\text{error}
&\,=\, \sqrt{\int_{U_s}\int_D \big(u_{s,h}(\bsx,\bsy)-u_{s,h,n}(\bsx,\bsy)\big)^2\,{\rm d}\bsx\,{\rm d}\bsy}\\
&\,\approx\, \sqrt{\frac{1}{Ln}\sum_{\ell=1}^L\sum_{k=1}^n \int_D
	\big(u_{s,h}\big(\bsx,\bsy_\ell+\boldsymbol{t}_k\big)-u_{s,h,n}(\bsx,\bsy_\ell+\boldsymbol{t}_k)\big)^2\,{\rm d}\bsx}\,,
\end{align*}
where $\bsy_\ell$ for $\ell=1,\ldots,L$ is a sequence of Sobol$'$ nodes in
$[0,1]^s$, with $L=100${, and we recall  that all our functions including
	$u_{s,h}(\bsx,\bsy)$ and $u_{s,h,n}(\bsx,\bsy)$ are $1$-periodic with
	respect to $\bsy$}. The kernel interpolant in the formula above can be
evaluated efficiently over the union of shifted lattices
$\bsy_\ell+\bst_k$, $\ell=1,\ldots,L$, $k=1,\ldots,n$, by making use of
formula~\eqref{eq:efficientfn} in conjunction with the fast Fourier
transform, requiring only the evaluation of the values
$K(\bst_k,\bsy_\ell)$.

We compute the approximation error when $\theta\in \{1.2,2.4,3.6\}$,
choosing $p\in\{\frac{1}{1.1},\frac{1}{2.2},\frac{1}{3.3}\}$,
respectively, which are all $p$ values ensuring that (A3) is satisfied. We
also use several values of the parameter $c\in\{0.2, 0.4, 1.5\}$ to
control the difficulty of the problem. We set $\delta=0.1$ in the product
weights~\eqref{eq:prod-weights}. The numerical experiments have been
carried out by using both  $s=10$ and $s=100$ as  the truncation
dimensions. Selected results are displayed in
Figures~\ref{fig:1}--\ref{fig:3}, where the corresponding values of
$a_{\min}$ and $a_{\max}$ are listed to give insights to the difficulty of
the problem in each case, as well as the parameter $\sigma$ which shows
the ``order'' of the lattice rule.
Note that as $s$ increases the problem does not change, but the computation becomes harder because the diffusion coefficient takes a wider range of values, with small values of $a(\bsx,\bsy)$ being especially challenging.

The empirically obtained convergence rates appear to exceed the
theoretically expected  rates  once the kernel interpolant enters the
asymptotic regime of convergence.
The convergence behavior of the kernel interpolant with SPOD weights is good across all experiments, except for the most difficult PDE problem of the lot corresponding to parameters $\theta=1.2$ and $c=0.4$, illustrated in the bottom row of Figure~\ref{fig:1}.
On the other hand, the POD weights and, to a lesser extent, the product weights appear to be somewhat sensitive to the effective dimension of the PDE problem, either leading to a longer pre-asymptotic regime compared to SPOD weights (see ``PROD'' in the bottom row of Figure~\ref{fig:2}) or no apparent convergence (see ``POD'' in the bottom rows of Figures~\ref{fig:2} and~\ref{fig:3}).

\begin{figure}[t]
	\centering
	\includegraphics[height=.36\textwidth]{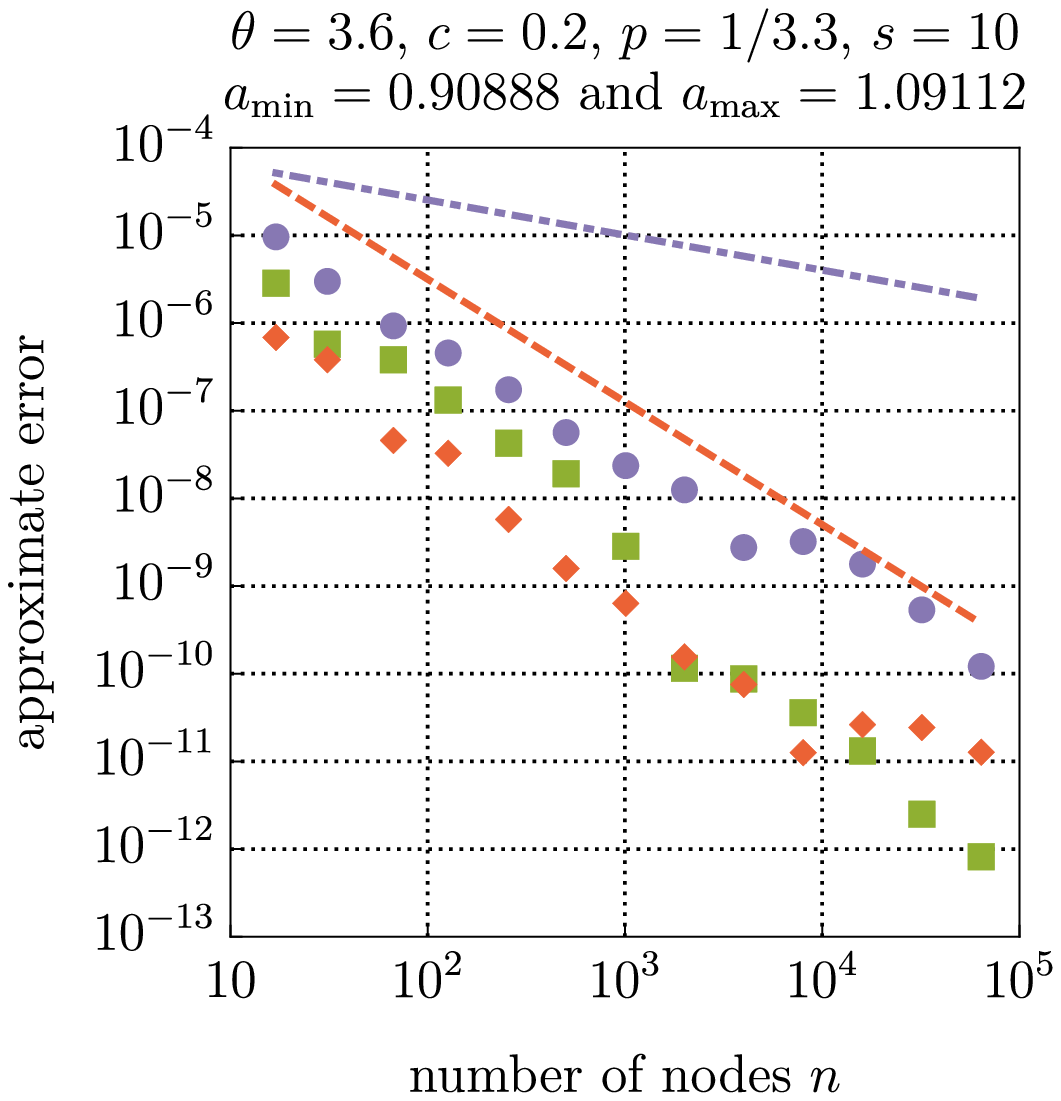}\includegraphics[height=.36\textwidth]{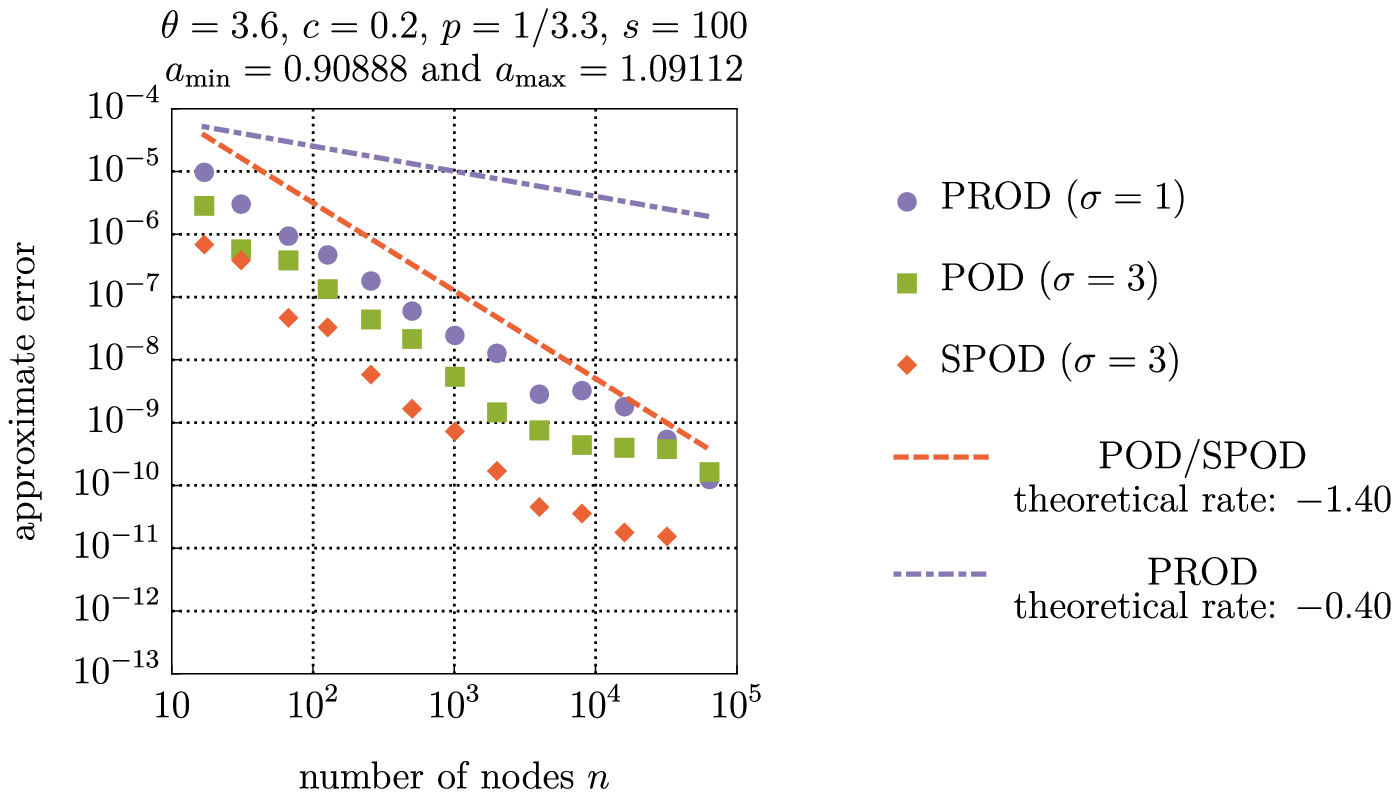}\\
	\bigskip
	\includegraphics[height=.36\textwidth]{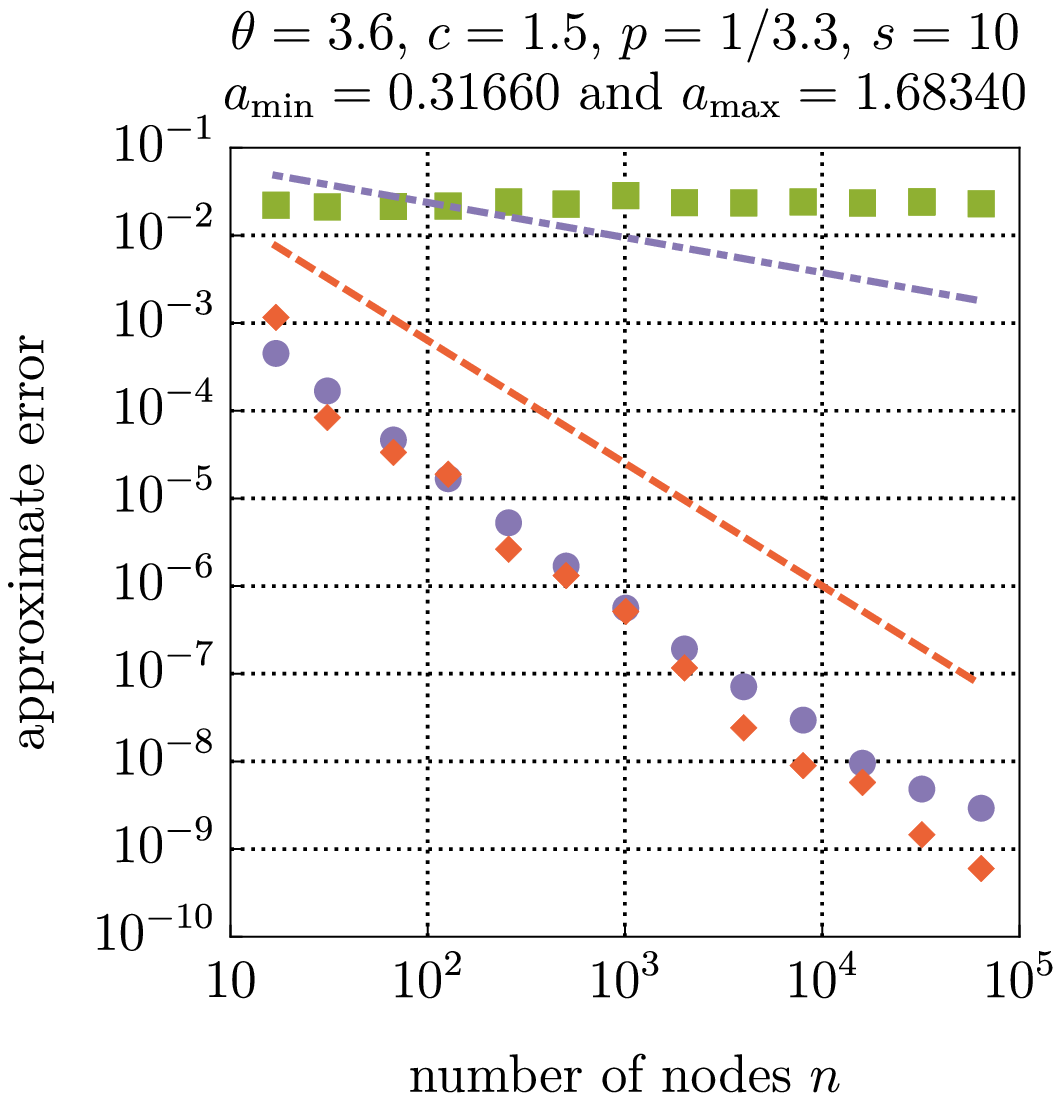}\includegraphics[height=.36\textwidth]{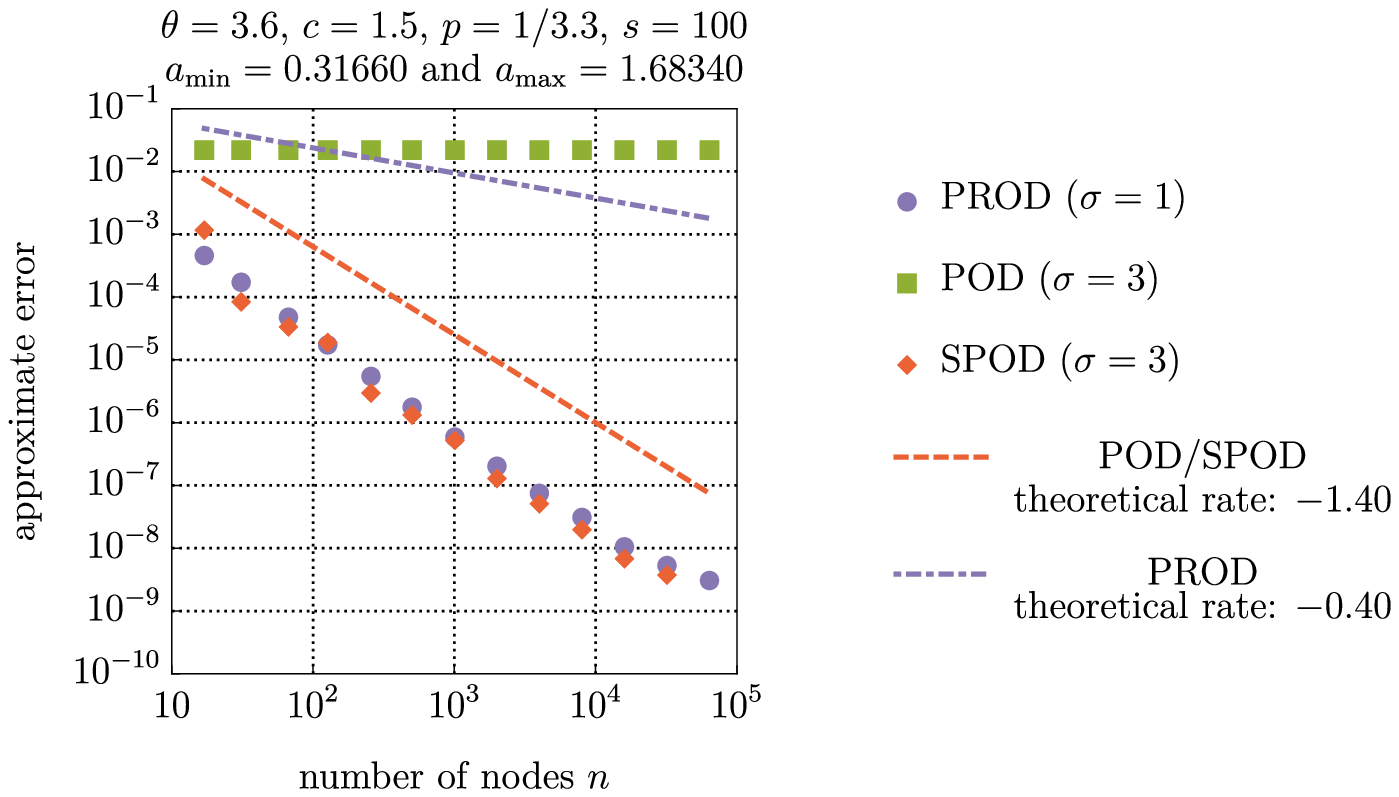}
	\caption{The kernel interpolation errors of the PDE problem~\eqref{eq:pdestrong}--\eqref{eq:pdestrong2}
		with $\theta=3.6$, $p=1/3.3$, $c\in\{0.2,1.5\}$, and $s\in\{10,100\}$.
		Results are displayed for kernel interpolants constructed using product (PROD), POD, and SPOD weights.}
	\label{fig:3}
\end{figure}

\begin{figure}[!t]
	\includegraphics[width=.9\textwidth]{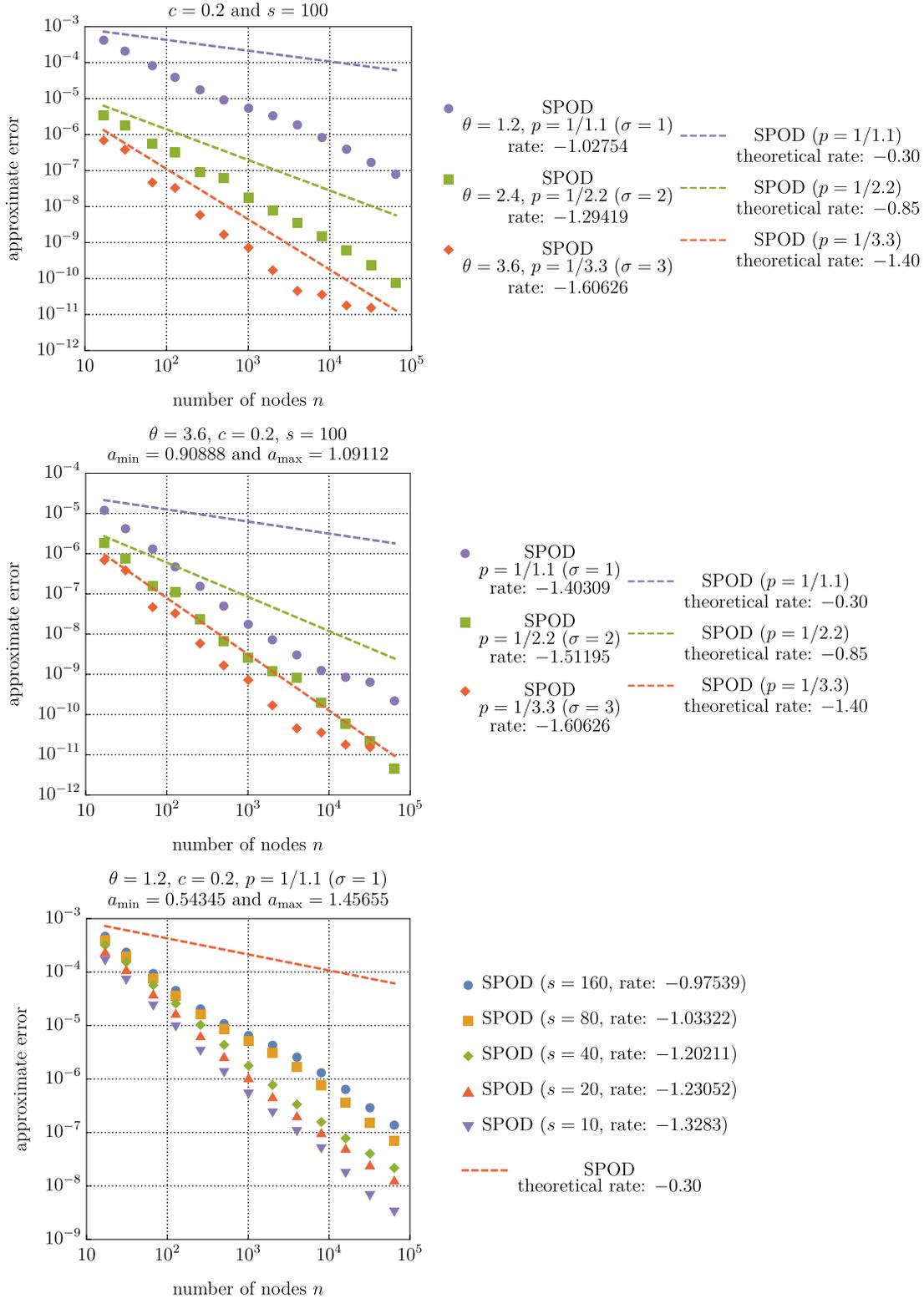}
	\caption{
		The kernel interpolation errors of the PDE problem~\eqref{eq:pdestrong}--\eqref{eq:pdestrong2}
		for kernel interpolants constructed using SPOD weights and varying parameters.
		Top: \emph{fixed} $s=100$ and $c=0.2$ and different values of $\theta$.
		Middle: \emph{fixed} $s=100$ and $\theta=3.6$, different values of $p$, and corresponding
		$\sigma=\sigma(p)$.
		Theoretical error-decay rate is $-\frac1{2p}+\frac14=-0.3,-0.85,-1.4$ for $p=\frac{1}{1.1},\frac{1}{2.2},\frac{1}{3.3}$.
		Bottom: \emph{fixed} $\theta=1.2$, $c=0.2$, $p=1/1.1$, and $\sigma=1$ with different values of $s\in\{10,20,40,80,160\}$.}
	\label{fig:bonus}
\end{figure}

\begin{figure}[!t]
	\includegraphics[width=1\textwidth]{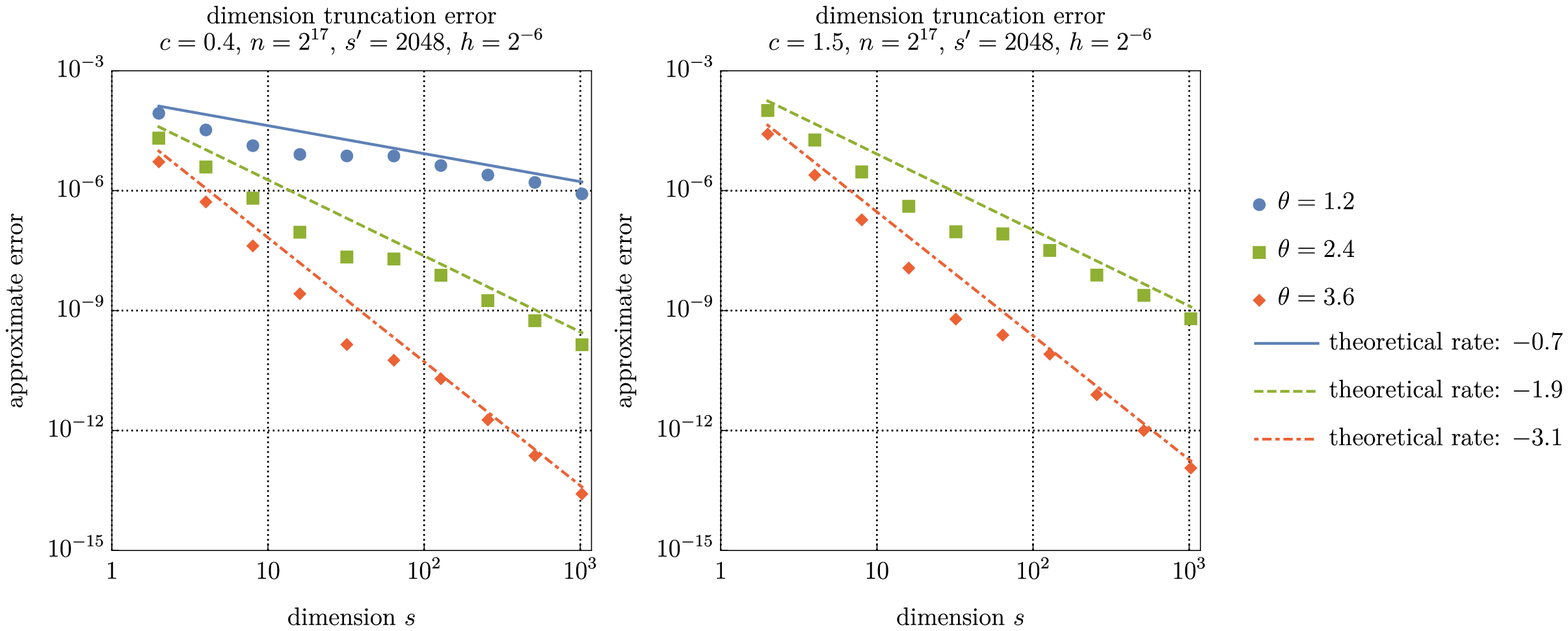}
	\caption{
		The dimension truncation errors of the PDE problem~\eqref{eq:pdestrong}--\eqref{eq:pdestrong2}. Left: $c=0.4$ and $\theta\in\{1.2,2.4,3.6\}$. Right: $c=1.5$ and $\theta\in\{2.4,3.6\}$.}
	\label{fig:dimtrunc}
\end{figure}

In the top graph of Figure~\ref{fig:bonus} we compare the results in
Figures~\ref{fig:1}--\ref{fig:3} from SPOD weights with truncation
dimension $s=100$ for the same damping parameter $c=0.2$ and different
$\theta\in\{1.2,2.4,3.6\}$, listing the estimated convergence rate in each
case. In the middle graph of Figure~\ref{fig:bonus} we show the results of
an additional experiment, namely, that for $s=100$ where we fix the decay
rate $\theta=3.6$ of the stochastic fluctuations and solve the parametric
PDE problem using different $\sigma\in\{1,2,3\}$ in the formula for SPOD
weights, which correspond to  $p\in
\{\frac{1}{1.1},\frac{1}{2.2},\frac{1}{3.3}\}$, see
Theorem~\ref{thm:spod}. Finally, in the bottom graph of
Figure~\ref{fig:bonus}, we return to the experimental setup illustrated in
Figure~\ref{fig:1} except this time we carry out the experiment using the
truncation dimensions $s\in\{10,20,40,80,160\}$.

In all cases displayed in Figure~\ref{fig:bonus}, the observed error
decays faster than the rate implied by Theorem~\ref{thm:spod}. We also see
that increasing $\sigma$ improves the error and mildly improves the rate
of convergence. Moreover, we observe from the graph in the middle that the
parameter $\theta$ that governs the decay of  $\|\psi_j\|_{L_{\infty}(D)}$
is more important in determining the rate than the choice of $\sigma$.
This observation suggests that the kernel interpolation with the rank-$1$
lattice points are robust in $\sigma$.
Notice that $\sigma$ appears in the definition  \eqref{eq:spod-weights} of
the SPOD weights; and that the weights are an input of the CBC
construction, and are used to define the kernel $K(\cdot,\cdot)$.
These observed error decay rates and the robustness are encouraging, but
also suggest that the worst-case error estimates may be pessimistic in
practical situations. The bottom graph in Figure~\ref{fig:bonus}
illustrates the effect that the truncation dimension has on the obtained
convergence rates{: we see that the observed convergence rate remains
	reasonable even when $s=160$.}

Finally, we present numerical experiments that assess the dimension
truncation error rate given in Theorem~\ref{thm:dimensiontruncation}. We
consider the same PDE and stochastic fluctuations $(\psi_j)_{j\geq 1}$
which were stated at the beginning of this section. We choose the
parameters $c=0.4$ with $\theta\in\{1.2, 2.4,3.6\}$ and $c=1.5$ with
$\theta\in\{2.4,3.6\}$. The PDE is discretized using piecewise linear
finite element method with mesh size $h=2^{-6}$ and the integral over the
computational domain $D$ is computed exactly for the finite element
solutions. As the reference solution, we use the finite element solution
with truncation dimension $s'=2^{11}$. The dimension truncation error is
then estimated by computing
\[
\sqrt{\int_{U_{s'}} \int_D (u_{s',h}(\bsx,\bsy)-u_{s,h}(\bsx,\bsy))^2\,{\rm d}\bsx\,{\rm d}\bsy}
\]
\sloppy{
for $s=2^k$, $k=2,\ldots,10$, where the value of the parametric integral
is computed approximately by means of a rank-1 lattice rule based on the
off-the-shelf generating vector \texttt{ lattice-39101-1024-1048576.3600}
downloaded from \url{https://web.maths.unsw.edu.au/~fkuo/lattice/} with
$n=2^{17}$ nodes. The results are displayed in Figure~\ref{fig:dimtrunc}.
The theoretically expected rate, which is essentially $\mathcal
O(s^{-\theta+1/2})$, is clearly observed in all cases.}

\section{Conclusions}

In this paper we have developed an approximation scheme for periodic
multivariate functions based on kernel approximation at lattice points, in
the setting of weighted Hilbert spaces of dominating mixed smoothness.  We
have developed $L_2$ error estimates that are independent of dimension,
for three classes of weights: product weights, POD (product and order
dependent) weights and SPOD (smoothness driven product and order
dependent) weights.  Numerical experiments for 10 and 100 dimensions give
results that (with the possible exception of POD weights) are generally
satisfactory, and that exhibit better than predicted rates of convergence.

Nevertheless, there is room for future improvement.  First, the error
analysis is based on the principle that the $L_2$ error is bounded above
by the worst case $L_2$ error multiplied by the norm of the function being
approximated; yet it is known (see Section~\ref{sec:wce-lower-bound}) that
the worst-case error has a poor rate of convergence. It may be possible to
obtain improved error rates by making better use of the special properties
of the minimum norm interpolant in conjunction with the analytic parameter
dependence of the PDE solution
of~\eqref{eq:pdestrong}--\eqref{eq:pdestrong2}.

\section*{Acknowledgements}
We gratefully acknowledge the financial support from the Australian Research Council for the project DP180101356. This research includes computations using the computational cluster Katana supported by Research Technology Services at UNSW Sydney.


\begin{thebibliography} {99}

	\bibitem{BKUV17}
	G. Byrenheid, L. K\"ammerer, T. Ullrich, and T. Volkmer.
	{Tight error bounds for rank-$1$ lattice sampling in spaces of hybrid mixed smoothness}.
	{\em Numer. Math.}, {\bf 136}:993--1034 (2017)

	\bibitem{ciarlet} P.~G. Ciarlet. \emph{The Finite Element Method for Elliptic Problems}. North-Holland Publishing Company (1978)

	\bibitem{CKNS-part1}
	R. Cools, F.~Y. Kuo, D. Nuyens, and I.~H. Sloan.
	{\it Lattice algorithms for multivariate approximation in periodic spaces with general weight
		parameters},
	In: Celebrating 75 Years of Mathematics of Computation
	(S. C. Brenner, I. Shparlinski, C.-W. Shu, and D. Szyld, eds.), Contemporary Mathematics, 754, AMS, 93--113 (2020)

	\bibitem{CKNS-part2}
	R. Cools, F.~Y. Kuo, D. Nuyens, and I.~H. Sloan.
	\newblock{Fast component-by-component construction of lattice algorithms for multivariate approximation with POD and SPOD weights},
	\newblock \textit{Math. Comp.}, \textbf{90}: 787--812 (2021)

	\bibitem{deBoor.C_Lynch_1966_spline_minimum}
	C. de~Boor and R.~E. Lynch.
	\newblock {On splines and their minimum properties}.
	\newblock {\em J. Math. Mech.}, {\bf 15}:953--969 (1966)

	\bibitem{gantnermcqmc2018} R.~N. Gantner. \newblock {Dimension truncation in {QMC} for affine-parametric operator equations}. \newblock {In {\em Monte Carlo and Quasi-Monte Carlo Methods 2016}}, pp.~249--264, Stanford, CA, August 14--19 (2018)

	\bibitem{Gantner.R_etal_2018_affine_local}
	R.~N. Gantner, L. Herrmann, and C. Schwab.
	\newblock {Quasi--Monte Carlo integration for affine-parametric, elliptic PDEs:
		local supports and product weights}.
	\newblock {\em SIAM J. Numer. Anal.}, {\bf 56}(1):111--135 (2018)

	\bibitem{Golomb.M_Weinberger_1959_optimal}
	M. Golomb and H.~F. Weinberger.
	\newblock {Optimal approximation and error bounds}.
	\newblock In {\em Numer. Approx. {P}roceedings a {S}ymposium, {M}adison,
		{A}pril 21--23, 1958}, edited by R. E. Langer. Publication No. 1 of the
	Mathematics Research Center, U.S. Army, the University of Wisconsin, pp.
	117--190. The University of Wisconsin Press, Madison, Wis. (1959)

	\bibitem{Graham.I_eta_2015_Numerische}
	I.~G. Graham, F.~Y. Kuo, J.~A. Nichols, R.~Scheichl, C. Schwab, and I.~H. Sloan.
	\newblock {Quasi-Monte Carlo finite element methods for elliptic PDEs with
		lognormal random coefficients}.
	\newblock {\em Numer. Math.}, {\bf 131}(2):329--368 (2015)

	\bibitem{Griebel.M_Rieger_2017_kernel}
	M. Griebel and C. Rieger.
	\newblock {Reproducing kernel Hilbert spaces for parametric partial
		differential equations}.
	\newblock {\em SIAM/ASA J. Uncertain. Quantif.}, {\bf 5}(1):111--137 (2017)

	\bibitem{Herrmann.L_Schwab_2019_local_lognormal}
	L. Herrmann and C. Schwab.
	\newblock {QMC integration for lognormal-parametric, elliptic PDEs: local
		supports and product weights}.
	\newblock {\em Numer. Math.}, {\bf 141}(1):63--102 (2019)

	\bibitem{KKS}
	V. Kaarnioja, F.~Y. Kuo, and I.~H. Sloan.
	Uncertainty quantification using periodic random variables.
	{\em SIAM J.\ Numer.\ Anal.,} {\bf 58}(2):1068--1091 (2020)

	\bibitem{KPV15}
	L.~K\"ammerer, D.~Potts, and T.~Volkmer.
	Approximation of multivariate periodic functions by trigonometric polynomials based on rank-$1$ lattice sampling.
	{\em J.\ Complexity,} \textbf{31}:543--576 (2015)

	\bibitem{Kazashi.Y_2019_product}
	Y. Kazashi.
	\newblock {Quasi-Monte Carlo integration with product weights for elliptic
		PDEs with log-normal coefficients}.
	\newblock {\em IMA J. Numer. Anal.}, {\bf 39}(3):1563--1593 (2019)

	\bibitem{Kempf.R_etal_2019_kernel_pde}
	R. Kempf, H. Wendland, and C. Rieger.
	\newblock {Kernel-based reconstructions for parametric PDEs}.
	\newblock In {\em Meshfree Methods
		Partial Differ. Equations IX. IWMMPDE 2017}, M.~Griebel and M.~Schweitzer, eds.,  pp. 53--71. Springer (2019)

	\bibitem{KMNN}
	F.~Y. Kuo, G. Migliorati, F. Nobile, and D. Nuyens.
	Function integration, reconstruction and approximation using rank-$1$
	lattices.
	{\em Math. Comp.}, {\bf 90}:1861--1897 (2021)

	\bibitem{Kuo.F_Schwab_Sloan_2012_SINUM}
	F.~Y. Kuo, C. Schwab, and I.~H. Sloan.
	\newblock {Quasi-Monte Carlo finite element methods for a class of elliptic
		partial differential equations with random coefficients}.
	\newblock {\em SIAM J. Numer. Anal.}, {\bf 50}(6):3351--3374 (2012)

	\bibitem{KSW06} F.~Y. Kuo, I.~H. Sloan, and H. Wo\'zniakowski.
	{Lattice rules for multivariate approximation in
		the worst case setting}. In
	{\em Monte Carlo and Quasi-Monte Carlo Methods 2004}, H.~Niederreiter and D.~Talay, eds., pp.~289--330,
	Springer (2006)

	\bibitem{KSW08} F.~Y. Kuo, I.~H. Sloan, and H. Wo\'zniakowski.
	{Lattice rule algorithms for multivariate approximation in the
		average case setting}.
	{\em J.~Complexity}, {\bf 24}:283--323 (2008)

	\bibitem{MicRiv77} C.~A. Micchelli and T.~J. Rivlin.
	\newblock {A survey of optimal recovery}.
	\newblock {\em Optim. Estim. Approx. theory ({P}roc. {I}nternat. {S}ympos.,
		{F}reudenstadt, 1976)}, pp. 1--54 (1977)

	\bibitem{MicRiv85} C.~A. Micchelli and T.~J. Rivlin.
	\newblock {\em Lectures on optimal recovery}.
	In: Turner P.R. (eds) Numerical Analysis Lancaster 1984. Lecture Notes in Mathematics, vol 1129. Springer (1985)

	\bibitem{Rauhut.H_Schwab_2016_Cheb}
	H. Rauhut and C. Schwab.
	\newblock {Compressive sensing Petrov-Galerkin approximation of
		high-dimensional parametric operator equations}.
	\newblock {\em Math. Comp.}, {\bf 86}(304):661--700 (2016)

	\bibitem{Rogers.L.C.G_Williams_2000_book_I}
	L.~C.~G. Rogers and D. Williams. {\em Diffusions, {{Markov}} Processes, and Martingales. {{Vol}}. 1}., 2nd edition, {Cambridge University Press} (2017)

	\bibitem{SJ94}
	I.~H. Sloan and S. Joe.
	\textit{Lattice Methods for Multiple Integration}.
	Oxford University Press (1994)

	\bibitem{SW01}
	I.~H. Sloan and H. Wo\'zniakowski.
	{Tractability of multivariate integration for weighted Korobov classes}.
	{\em J. Complexity}, \textbf{17}:697--721 (2001)

	\bibitem{Wen05} H.~Wendland.
	\newblock {\em {Scattered Data Approximation}}.
	\newblock Cambridge University Press (2005)

	\bibitem{XK02}
	D. Xiu and G.~E. Karniadakis.
	{The Wiener--Askey polynomial chaos for stochastic differential equations}.
	{\em SIAM J. Sci. Comput.},
	\textbf{24}:{619--644} (2002)

	\bibitem{ZLH06} X.~Y. Zeng, K.~T. Leung, and F.~J. Hickernell.
	{Error analysis of splines for periodic problems using lattice designs}.
	In
	{\em Monte Carlo and Quasi-Monte Carlo Methods 2004}, H.~Niederreiter and D.~Talay, eds., pp.~501--514, Springer (2006)

	\bibitem{ZKH09} X.~Y. Zeng, P. Kritzer, and F.~J. Hickernell.
	{Spline methods using integration lattices and digital nets}.
	{\em Constr. Approx.}, {\bf 30}:529--555 (2009)

\end{thebibliography}
\end{document}